\documentclass{article}

\usepackage[final]{neurips_2025}
\usepackage{amsmath}
\usepackage{amssymb}
\usepackage{mathtools}
\usepackage{amsthm}
\usepackage{nicefrac}
\usepackage{multirow}
\usepackage{tikz-cd} %for tikzcd
\usepackage{subcaption}
\usepackage{multicol,caption}

\usepackage{wrapfig}
\usepackage{algorithm}
\usepackage[noend]{algorithmic}

\usepackage[flushleft]{threeparttable} % http://ctan.org/pkg/threeparttable
\usepackage{colortbl}
\definecolor{bgcolor}{rgb}{0.8,1,1}
\definecolor{bgcolor2}{rgb}{0.8,1,0.8}
\definecolor{niceblue}{rgb}{0.0,0.19,0.56}

\usepackage{hyperref}
\hypersetup{colorlinks,linkcolor={blue},citecolor={niceblue},urlcolor={blue}}

\usepackage{pifont}
\definecolor{PineGreen}{RGB}{0,110,51}
\definecolor{BrickRed}{RGB}{143,20,2}

\usepackage{thmtools}

\usepackage[colorinlistoftodos,bordercolor=orange,backgroundcolor=orange!20,linecolor=orange,textsize=scriptsize]{todonotes}

\definecolor{shadecolor}{gray}{0.9}
\declaretheoremstyle[
headfont=\normalfont\bfseries,
notefont=\mdseries, notebraces={(}{)},
bodyfont=\normalfont,
postheadspace=0.5em,
spaceabove=\topsep,
mdframed={
  skipabove=8pt,
  skipbelow=8pt,
  hidealllines=true,
  backgroundcolor={shadecolor},
  innerleftmargin=4pt,
  innerrightmargin=4pt}
]{shaded}

\declaretheorem[style=shaded,within=section]{definition}
\declaretheorem[style=shaded,sibling=definition]{theorem}
\declaretheorem[style=shaded,sibling=definition]{proposition}
\declaretheorem[style=shaded,sibling=definition]{assumption}
\declaretheorem[style=shaded,sibling=definition]{corollary}

\declaretheorem[style=shaded,sibling=definition]{lemma}

\usepackage[capitalize,noabbrev]{cleveref}

\usepackage{xspace}
\usepackage{xspace}

\newcommand{\algname}[1]{{\color{PineGreen}\sf  #1}\xspace}

\newcommand{\R}{\mathbb{R}}

\newcommand{\B}{\mathbf{B}}
\newcommand{\A}{\mathbf{A}}
\newcommand{\C}{\mathbf{C}}
\newcommand{\M}{\mathbf{M}}

\newcommand{\hx}{\hat{x}}

\newcommand{\la}{\left\langle}

\newcommand{\ra}{\right\rangle}

\newcommand{\eqdef}{:=}
\newcommand{\argmin}{\text{argmin}}

\usepackage{hyperref}
\hypersetup{colorlinks,linkcolor={blue},citecolor={niceblue},urlcolor={blue}}

\usepackage{natbib}
\usepackage{sidecap}

\usepackage{wrapfig}

\definecolor{shadecolor}{gray}{0.9}
\declaretheoremstyle[
headfont=\normalfont\bfseries,
notefont=\mdseries, notebraces={(}{)},
bodyfont=\normalfont,
postheadspace=0.5em,
spaceabove=1pt,
mdframed={
  skipabove=8pt,
  skipbelow=8pt,
  hidealllines=true,
  backgroundcolor={shadecolor},
  innerleftmargin=4pt,
  innerrightmargin=4pt}
]{shaded}

\usepackage{xspace}

\usepackage[utf8]{inputenc} % allow utf-8 input
\usepackage[T1]{fontenc}    % use 8-bit T1 fonts
\usepackage{hyperref}       % hyperlinks
\usepackage{url}            % simple URL typesetting
\usepackage{booktabs}       % professional-quality tables
\usepackage{amsfonts}       % blackboard math symbols
\usepackage{nicefrac}       % compact symbols for 1/2, etc.
\usepackage{microtype}      % microtypography
\usepackage{xcolor}         % colors
\usepackage{enumitem}       % itemize
\usepackage{algorithmic}    % algorithm

\title{Extragradient Method for $(L_0, L_1)$-Lipschitz Root-finding Problems}%: \\Novel Adaptive Step Size for Improved Performance}

\author{%
  Sayantan Choudhury \\
  AMS \& MINDS \\
  Johns Hopkins University \\
  \texttt{schoudh8@jhu.edu} \\
  \And
  Nicolas Loizou\\
  AMS \& MINDS\\
  Johns Hopkins University\\
  \texttt{nloizou1@jhu.edu} \\
}

\begin{document}

\maketitle

\begin{abstract}
Introduced by Korpelevich in 1976, the extragradient method (EG) has become a cornerstone technique for solving min-max optimization, root-finding problems, and variational inequalities (VIs). Despite its longstanding presence and significant attention within the optimization community, most works focusing on understanding its convergence guarantees assume the strong $L$-Lipschitz condition. In this work, building on the proposed assumptions by \cite{zhang2019gradient} for minimization and \cite{pmlr-v235-vankov24a} for VIs, we focus on the more relaxed $\alpha$-symmetric $(L_0, L_1)$-Lipschitz condition. This condition generalizes the standard Lipschitz assumption by allowing the Lipschitz constant to scale with the operator norm, providing a more refined characterization of problem structures in modern machine learning. Under the $\alpha$-symmetric $(L_0, L_1)$-Lipschitz condition, we propose a novel step size strategy for EG to solve root-finding problems and establish sublinear convergence rates for monotone operators and linear convergence rates for strongly monotone operators. Additionally, we prove local convergence guarantees for weak Minty operators. We supplement our analysis with experiments validating our theory and demonstrating the effectiveness and robustness of the proposed step sizes for EG.
\end{abstract}

\section{Introduction}\label{sec:intro}
Min-max optimization problems have recently attracted significant interest due to their widespread applications in machine learning, such as reinforcement learning~\citep{brown2020combining,sokotaunified}, distributionally robust optimization~\citep{namkoong2016stochastic}, and generative adversarial network training~\citep{goodfellow2020generative}. These problems are often formulated as variational inequalities (VIs) ~\citep{ryu2022large, gidel2018variational,sokotaunified}. In the unconstrained case, the VI problem simplifies to the root-finding problem~\citep{luo2022extragradient, tran2024variance}, defined as follows~\citep{gorbunov2022extragradient}:
\begin{equation}\label{eq:vip}
\textstyle
    \text{Find } x_* \in \R^d \text{ such that } F(x_*) =  0,
\end{equation}
where $F: \R^d \to \R^d$ is an operator. Root finding problems of the form~\eqref{eq:vip} encompass a variety of problems as special cases, such as: (i) \textbf{Unconstrained minimization:} finding a stationary point of $\min_{x \in \R^d} f(x)$, 
is equivalent to solving~\eqref{eq:vip} with $F(x) = \nabla f(x)$, (ii) \textbf{Min-max optimization:} Let $\min_{w_1 \in \mathbb{R}^{d_1}} \max_{w_2 \in \mathbb{R}^{d_2}} \mathcal{L}(w_1, w_2)$ where $\mathcal{L}: \mathbb{R}^{d_1} \times \mathbb{R}^{d_2} \rightarrow \mathbb{R}$. In this scenario, if in \eqref{eq:vip}, the operator is selected as follows: 
\vspace{-2mm}
\begin{equation}\label{eq:minmax_operator}
    \textstyle
        F(x) \equiv \bigl(
            \nabla_{w_1} \mathcal{L}(w_1, w_2)^\top,\;
            -\nabla_{w_2} \mathcal{L}(w_1, w_2)^\top
        \bigr)^\top,
    \end{equation}
then solving~\eqref{eq:vip} amounts to finding a stationary point $x_* = (w_{1*}^\top, w_{2*}^\top)^\top \in \mathbb{R}^{d_1 + d_2}$ of the min-max problem, which for convex-concave functions $\mathcal{L}$ is a global solution (Nash equilibrium), i.e., $\mathcal{L}(w_{1*}, w_2) \le \mathcal{L}(w_{1*}, w_{2*}) \le \mathcal{L}(w_1, w_{2*})$~\citep{luo2025adaptive, choudhury2024single}, 
and (iii) \textbf{Multiplayer games:} A Nash equilibrium $x_*=(w_{1*}^\top,\dots, w_{N*}^\top)^\top$ of an $N$-player game in which each player $i$ minimizes
their own convex objective $\mathcal{L}_i(w_i, w_{-i})$ with respect to $w_i$ (here $w_{-i}$ denotes
the actions of all players except $i$) is also captured by~\eqref{eq:vip}, with an operator  
$
    F(x) = \bigl(
        \nabla_{w_1} \mathcal{L}_1(w_1, w_{-1})^\top,\;
        \dots,\;
        \nabla_{w_N} \mathcal{L}_N(w_N, w_{-N})^\top
    \bigr)^\top
$~\citep{yoon2025multiplayer}.

Problem~\eqref{eq:vip} and algorithms for solving it have been studied extensively in recent years under different conditions on the operator $F$~\citep{loizou2021stochastic, loizou2020stochastic, gorbunov2022stochastic,diakonikolas2021efficient, choudhury2024single}. One of the well-known algorithms for solving VIs and root-finding problems of the form ~\eqref{eq:vip} is the extragradient (\algname{EG}) method~\citep{korpelevich1977extragradient} due to its superior convergence guarantees~\citep{gorbunov2022extragradient}. The algorithm is defined as follows
\begin{eqnarray}\label{eq:extragradient}
\textstyle
    \hx_k & = & x_k - \gamma_k F(x_k), \notag \\
    x_{k+1} & = & x_k - \omega_k F(\hx_k)
\end{eqnarray}
where $\gamma_k>0$ and $\omega_k > 0$ are the extrapolation step size and update step size, respectively. Since its original inception by Korpelevich, the \algname{EG} method was revisited and
extended in various ways, e.g., non-monotone operators~\citep{diakonikolas2021efficient, fan2023weaker} stochastic \citep{mishchenko2020revisiting, gorbunov2022stochastic,choudhury2024single,li2022convergence}, distributed~\citep{beznosikov2022decentralized}. Despite a rich literature for analysing \algname{EG} and its variants, most of the existing convergence guarantees heavily rely on the $L$-Lipschitz assumption of the operator $F$~\citep{korpelevich1977extragradient, diakonikolas2021efficient}, i.e.
\begin{equation}\label{eq:L-Lipschitz}
\textstyle
    \|F(x) - F(y)\| \leq L\| x - y\|
\end{equation}
for all $x, y \in \R^d$. However, this assumption can be restrictive; for instance, the operator $F(x) = x^2$ for $x \in \R$ does not satisfy \eqref{eq:L-Lipschitz} for any finite $L$~\citep{zhang2019gradient}. \emph{The primary goal of this work is to relax this assumption and establish convergence guarantees under a more general framework.} \\
\newline
\textbf{Relaxing the $L$-Lipschitz Assumption.} Recently, \cite{zhang2019gradient} introduced the $(L_0, L_1)$-smoothness assumption for the minimization problems. Specifically, for $\min_{x \in \R^d} f(x)$, \cite{zhang2019gradient} assume $ \| \nabla^2 f(x) \| \leq L_0 + L_1 \| \nabla f(x)\|$ (when $f$ is twice differentiable) and later~\citep{chen2023generalized} proved that this is equivalent to:
\begin{equation}\label{eq:(L0,L1)}
\textstyle
    \|\nabla f(x) - \nabla f(y)\| \leq (L_0 + L_1 \|\nabla f(x)\|) \|x - y\|.    
\end{equation}
\cite{zhang2019gradient} demonstrated that modern neural networks, such as LSTMs (Long Short-Term Memorys)~\citep{hochreiter1997long}, 
align with $(L_0, L_1)$-smoothness assumption rather than the traditional $L$-smoothness assumption (i.e. $\|\nabla f(x) - \nabla f(y)\| \leq L \|x - y\|$). Moreover, they used this assumption to justify why gradient clipping speeds up neural network training. Later, \cite{ahn2023linear} showed that similar trends hold for the transformer~\citep{vaswani2017attention} architecture.

\begin{wrapfigure}{r}{0.4\textwidth} 
    \centering
 \includegraphics[width=0.42\textwidth]{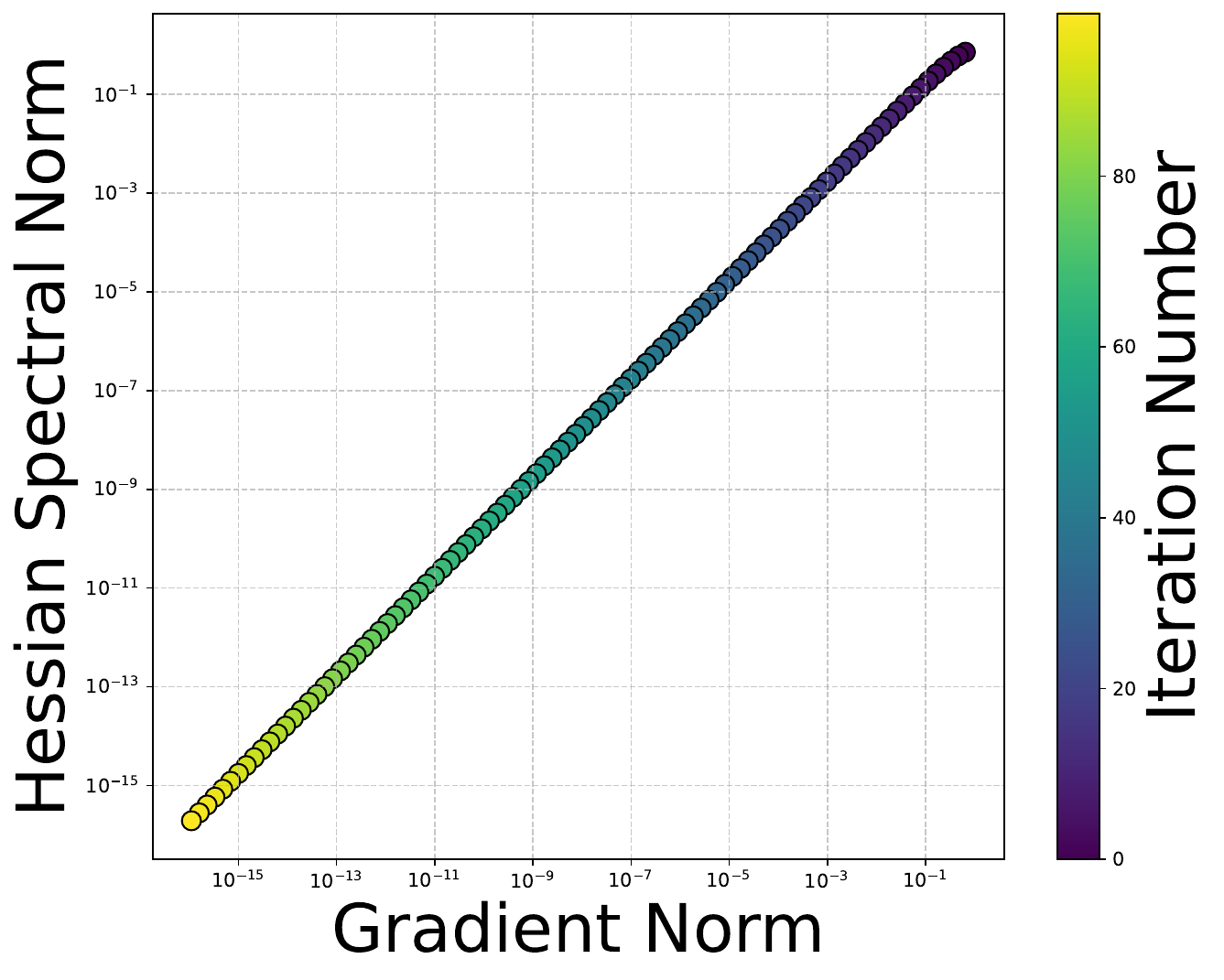}
    \caption{\small{Scatter plot of $\| \nabla^2 f(x_k)\|$ on $y$-axis and $\|\nabla f(x_k)\|$ on $x$-axis.}}
    \label{fig:hessian_vs_gradient_plot}
    \vspace{-4mm}
\end{wrapfigure}

In Figure~\ref{fig:hessian_vs_gradient_plot}, we present an example demonstrating the validity of the $(L_0, L_1)$-smoothness condition~\eqref{eq:(L0,L1)} for the iterates of the \algname{EG}. Similar plots have been presented for gradient descent methods~\citep{zhang2019gradient,gorbunov2024methods}, but to the best of our knowledge, this linear connection between $\|\nabla^2 f(x_k)\|$ and $\|\nabla f(x_k)\|$ for \algname{EG} iterates was never reported before. Following~\citet{gorbunov2024methods}, we consider the minimization problem $\min_{x \in \R^d} f(x) \eqdef \log \left(1 + \exp(- a^{\top}x) \right)$, and plot the values of $\|\nabla^2 f(x_k)\|$ on the $y$-axis against $\|\nabla f(x_k)\|$ on the $x$-axis. Each point is colored according to the iteration index $k$, as indicated by the accompanying colorbar. The resulting plot reveals an approximately linear relationship between $\|\nabla^2 f(x_k)\|$ and $\|\nabla f(x_k)\|$, thereby supporting the modeling of this function within the $\| \nabla^2 f(x)\| \leq L_0 + L_1 \| \nabla f(x)\|$ or $(L_0, L_1)$-smoothness framework.

Now, let us consider the min-max optimization problem $\min_{w_1 \in \R^{d_1}} \max_{w_2 \in \R^{d_2}} \mathcal{L}(w_1, w_2)$, which is captured by \eqref{eq:vip} with the operator \eqref{eq:minmax_operator}. If the operator $F$ is $L$-Lipschitz, then its Jacobian matrix $\mathbf{J}(x)$, defined in \eqref{eq:jacobian}, satisfies $\|\mathbf{J}(x)\| \leq L$ for all $x = (w_1^{\top}, w_2^{\top})^{\top}$(follows from Theorem \ref{thm:equiv_formulation} with $L_1 = 0$). For example, consider the quadratic min-max objective $\min_{w_1} \max_{w_2} \mathcal{L}(w_1, w_2) = \frac{1}{2}w_1^2 + w_1 w_2 - \frac{1}{2}w_2^2$. In this case, implementing the \algname{EG} method and plotting the Jacobian norm $\| \mathbf{J}(x_k) \|$ (on the $y$-axis) against the operator norm $\|F(x_k) \|$ (on the $x$-axis) yields a horizontal line parallel to $x$-axis (see Appendix \ref{appendix:tech_lemma}).

However, this behaviour does not persist for more complex problems. For instance, for the cubic objective
\begin{equation}\label{eq:min_max_cubic}
\textstyle
    \min_{w_1} \max_{w_2} \mathcal{L}(w_1, w_2) = \frac{1}{3}w_1^3 + w_1w_2 - \frac{1}{3}w_2^3,
\end{equation}
$\|\mathbf{J}(x_k)\|$ increases with the $\|F(x_k)\|$. This observation suggests that the standard Lipschitz assumption may be overly restrictive for capturing the structure of such problems (check Figure \ref{fig:jacobian_vs_operator_plot}). \\
\begin{wrapfigure}{r}{0.4\textwidth} 
\vspace{-5mm}
    \centering
\includegraphics[width=0.42\textwidth]{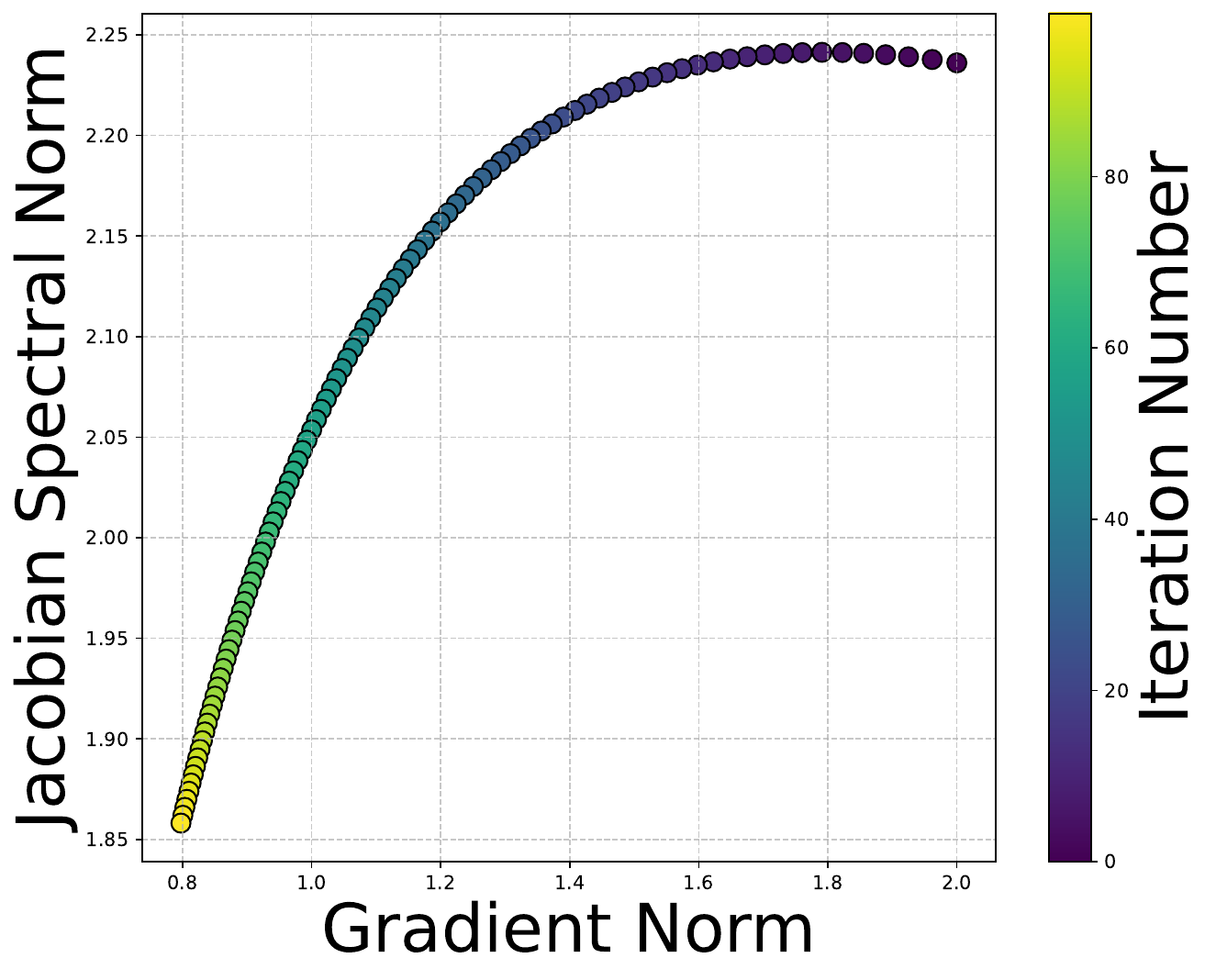}
    \caption{\small Scatter plot of $\| \mathbf{J}(x_k)\|$ on $y$-axis and $\| F(x_k)\|$ on $x$-axis.}\label{fig:jacobian_vs_operator_plot}
    \vspace{-4mm}
\end{wrapfigure}
To better model this relationship, we investigate a relaxed condition of the form $\| \mathbf{J}(x) \| \leq L_0 + L_1 \|F(x)\|^\alpha$ with $\alpha \in (0, 1]$ which generalizes the standard Lipschitz bound (for $L_1 = 0$, this boils down to $\|\mathbf{J}(x)\| \leq L_0$, which is the Lipschitz property). Note that, instead of $\alpha = 1$, our formulation permits $\alpha$ to lie in the broader interval $(0, 1]$. This condition is motivated by the plot in Figure~\ref{fig:jacobian_vs_operator_plot}, which suggests a sublinear relationship, resembling the form $h(r) = L_0 + L_1 r^{\alpha}$ for some $\alpha \in (0, 1)$, rather than a linear trend. 

As we will prove later (in Theorem~\ref{thm:equiv_formulation}), for any  doubly differentiable min-max optimization problem $\min_{w_1} \max_{w_2} \mathcal{L}(w_1, w_2)$, the condition $\| \mathbf{J}(x) \| \leq L_0 + L_1 \| F(x) \|^{\alpha}$ is equivalent to the operator $F$ satisfying the $\alpha$-symmetric $(L_0, L_1)$-Lipschitz condition (see Assumption~\ref{asdioa}). 
The equivalent $\alpha$-symmetric $(L_0, L_1)$-Lipschitz condition does not rely on second-order information and applies to a broader class of problems (no need for double differentiability). Therefore, in the remainder of this work, we focus on analyzing the convergence of \algname{EG} under the $\alpha$-symmetric $(L_0, L_1)$-Lipschitz assumption~\citep{pmlr-v235-vankov24a} on the operator $F$, defined below.

\begin{assumption}
\label{asdioa}
$F$ is called $\alpha$-symmetric $(L_0, L_1)$-Lipschitz operator if for some $L_0, L_1 \geq 0$ and $\alpha \in (0, 1]$,
    \begin{equation}\label{eq:(L0,L1)-Lipschitz}
    \textstyle
        \| F(x) - F(y)\| \leq \left( L_0 + L_1 \max_{\theta \in [0, 1]} \left\|F(\theta x + ( 1- \theta) y) \right\|^{\alpha} \right) \| x- y\| \quad \forall x, y \in \R^d.
    \end{equation}
\end{assumption}
Instead of a fixed Lipschitz constant in \eqref{eq:L-Lipschitz}, Assumption~\ref{asdioa} allows the Lipschitz-like quantity to depend on the norm of the operator itself along the path from $x$ to $y$. This assumption generalizes the standard $L$-Lipschitz condition \eqref{eq:L-Lipschitz}, corresponding to the special case where $L_0 = L$ and $L_1 = 0$. 
Moreover, the $\alpha$-symmetric $(L_0, L_1)$-Lipschitz condition provides a more refined characterisation of operators whose Lipschitz constant depends on their norm, offering a tighter bound by balancing $L_0 \ll L$ and $L_1 \ll L$. Additionally, \eqref{eq:(L0,L1)-Lipschitz} provides a more relaxed bound compared to \eqref{eq:(L0,L1)} with $\alpha = 1$. \\
\newline
\textbf{Classes of Root-finding Problems.} Apart from the condition \eqref{eq:(L0,L1)-Lipschitz}, we will also assume additional structure on the operator $F$ to prove convergence. We say the operator $F$ is monotone or strongly monotone if it satisfies the following assumption.

\begin{assumption} \label{assume:strong_monotone}
    $F$ is called monotone if
    \begin{equation}\label{eq:monotone}
    \textstyle
        \la F(x) - F(y), x - y \ra \geq 0 \quad \forall x, y \in \R^d
    \end{equation}
    and strongly monotone if there is some $\mu > 0$ such that
    \begin{equation}\label{eq:strong_monotone}
    \textstyle
        \la F(x) - F(y), x - y \ra \geq \mu \|x - y\|^2 \quad \forall x, y \in \R^d.
    \end{equation}

\end{assumption}

This captures convex minimization and convex-concave min-max optimization problems as a special case. Apart from the monotone operators, we are also interested in some non-monotone operators, weak Minty operators~\citep{diakonikolas2021efficient}, which satisfy the following assumption. 
\begin{assumption}
    Operator $F$ is called weak Minty if for some $\rho \geq 0$, 
    \begin{equation}\label{eq:weak_minty}
    \textstyle
        \la F(x), x - x_* \ra \geq -\rho \|F(x)\|^2  \qquad \forall x \in \R^d.
    \end{equation}
\end{assumption} 
\vspace{-3mm}
\subsection{Main Contributions}
We summarize the main contributions of our work below.
\begin{itemize}[itemsep=0pt, leftmargin=*]
    \item \textbf{Tighter analysis for strongly monotone:} We establish linear convergence guarantees for strongly monotone~\eqref{eq:strong_monotone} $\alpha$-symmetric $(L_0, L_1)$-Lipschitz problems (see Theorem \ref{theorem:1symm_strongmonotone}, \ref{theorem:1symm_strongmonotone_alhpha01}). In contrast to the results in~\citet{pmlr-v235-vankov24a} for $\alpha = 1$, our analysis shows that linear convergence can be achieved without incurring exponential dependence on the initial distance to the solution $\|x_0 - x_*\|$ (see Corollary~\ref{corollary:1symm_strongmonotone}).

    \item \textbf{First analysis for monotone and weak Minty:} We provide the first convergence analysis of \algname{EG} for solving monotone~\eqref{eq:monotone} and weak Minty \eqref{eq:weak_minty} problems under $\alpha$-symmetric $(L_0, L_1)$-Lipschitz assumption. We establish global sublinear convergence for monotone problems (see Theorem \ref{theorem:1symm_monotone}, \ref{theorem:alpha01}) and local sublinear convergence for weak Minty problems (see Theorem \ref{theorem:weak_minty_alpha1}, \ref{theorem:weak_minty_alpha01}).

    \item \textbf{Novel step size for \algname{EG}:} We propose a novel adaptive step size strategy for the \algname{EG} method designed to handle $\alpha$-symmetric $(L_0, L_1)$-Lipschitz operators. Specifically, all our step size schemes adopt the general form $\gamma_k = \frac{1}{c_0 + c_1 \|F(x_k)\|^{\alpha}},$ where $c_0, c_1 > 0$ are constants determined by the problem-dependent parameters $L_0$, $L_1$, and $\alpha$. In Table \ref{tab:comparison_of_rates}, we included a detailed summary of our proposed step size selection for different classes of operators and compared it with closely related works.

    \item \textbf{Numerical experiments:} 
    Finally, in Section~\ref{sec:experiments}, we present experiments validating different aspects of our theoretical results. We compare our proposed step size selections with existing alternatives, demonstrating the effectiveness and robustness of our approach.

\end{itemize}

\begin{table*}[t]
    \centering
    % \scriptsize
    \caption{
    \small
       Summary of step size selection for \algname{EG} under the $L$-Lipschitz and $\alpha$-symmetric $(L_0, L_1)$-Lipschitz assumptions. Our proposed step size strategy is of the general form $\gamma_k = \frac{1}{c_0 + c_1 \|F(x_k)\|^{\alpha}}$, tailored for solving problems involving $\alpha$-symmetric $(L_0, L_1)$-Lipschitz operators.
    }
    \label{tab:comparison_of_rates}
    \begin{threeparttable}
    \resizebox{\textwidth}{!}{%
        \begin{tabular}{|c|c|c c c|}
        \hline
        Setup & Assumption & $\alpha$ & $\gamma_k$ & $\omega_k$ 
        \\
        \hline\hline
        \multirow{7}{2cm}{\centering Strongly \\ Monotone \tnote{{\color{blue}(1)}}} & \begin{tabular}{c}
            $L$-Lipschitz \tnote{{\color{blue}(2)}}\\
            \citep{mokhtari2020unified}
        \end{tabular} & - & $\frac{0.25}{L}$ & $\gamma_k$\\[10pt]
        
        & \begin{tabular}{c}
            $\alpha$-symmetric $(L_0, L_1)$-Lipschitz \\
            \citep{pmlr-v235-vankov24a} 
        \end{tabular} & $1$ & $\min \left\{ \frac{1}{4 \mu}, \frac{1}{2\sqrt{2} e L_0}, \frac{1}{2 \sqrt{2} e L_1 \|F(x_k)\|} \right\}$ & $\gamma_k$ \\[10pt]

        &\cellcolor{bgcolor2}\begin{tabular}{c}
            $\alpha$-symmetric $(L_0, L_1)$-Lipschitz \\
            (Theorem \ref{theorem:1symm_strongmonotone})
        \end{tabular} & \cellcolor{bgcolor2} $1$ & \cellcolor{bgcolor2} $\frac{0.21}{L_0 + L_1 \|F(x_k)\|}$ & \cellcolor{bgcolor2} $\gamma_k$ \\[10pt]
        &\cellcolor{bgcolor2}\begin{tabular}{c}
            $\alpha$-symmetric $(L_0, L_1)$-Lipschitz\\
            (Theorem \ref{theorem:1symm_strongmonotone_alhpha01})
        \end{tabular} & \cellcolor{bgcolor2} $(0, 1)$ & \cellcolor{bgcolor2} $\frac{0.61}{2K_0 + \left( 2 K_1 + 2^{1 - \alpha} K_2^{1 - \alpha} \right) \| F(x_k)\|^{\alpha}}$ \tnote{{\color{blue}(2)}} & \cellcolor{bgcolor2} $\gamma_k$ \\[10pt]

        \hline\hline
        \multirow{6}{2cm}{\centering Monotone \tnote{{\color{blue}(1)}}} & \begin{tabular}{c}
            $L$-Lipschitz\\
            \citep{gorbunov2022extragradient}
        \end{tabular} & - & $\frac{1}{L}$  & $\frac{\gamma_k}{2}$ \\[10pt]
    
        &\cellcolor{bgcolor2}\begin{tabular}{c}
            $\alpha$-symmetric $(L_0, L_1)$-Lipschitz\\
            (Theorem \ref{theorem:1symm_monotone})
        \end{tabular} & \cellcolor{bgcolor2} $1$ & \cellcolor{bgcolor2} $\frac{0.45}{L_0 + L_1 \|F(x_k)\|}$ & \cellcolor{bgcolor2} $\gamma_k$  \\[10pt]

        &\cellcolor{bgcolor2}\begin{tabular}{c}
            $\alpha$-symmetric $(L_0, L_1)$-Lipschitz\\
            (Theorem \ref{theorem:alpha01})
        \end{tabular} & \cellcolor{bgcolor2} $(0, 1)$ & \cellcolor{bgcolor2} $\frac{1}{2 \sqrt{2} K_0  + \left( 2\sqrt{2} K_1 + 2^{\nicefrac{3 (1 - \alpha)}{2}} K_2^{1 - \alpha} \right) \|F(x_k)\|^{\alpha}}$ & \cellcolor{bgcolor2} $\gamma_k$ \\[10pt]

        \hline\hline
        \multirow{8}{2cm}{\centering Weak Minty \tnote{{\color{blue}(1)}}} & \begin{tabular}{c}
            $L$-Lipschitz\\
            \citep{diakonikolas2021efficient}
        \end{tabular} & - & $\frac{1}{L}$ & $\frac{\gamma_k}{2}$ \\[10pt]
        & \begin{tabular}{c}
            $L$-Lipschitz\\
            \citep{pethick2023escaping}
        \end{tabular} & - & $\frac{1}{L}$ & $\rho + \frac{\la F(\hx_k), x_k - \hx_k\ra}{\| F(\hx_k)\|^2}$ \\[10pt]

        &\cellcolor{bgcolor2}\begin{tabular}{c}
            $\alpha$-symmetric $(L_0, L_1)$-Lipschitz\\
            (Theorem \ref{theorem:weak_minty_alpha1})
        \end{tabular} & \cellcolor{bgcolor2} $1$ & \cellcolor{bgcolor2} $\frac{0.56}{L_0 + L_1 \|F(x_k)\|}$ & \cellcolor{bgcolor2} $\frac{\gamma_k}{2}$ \\[10pt]
        
        & \cellcolor{bgcolor2}\begin{tabular}{c}
            $\alpha$-symmetric $(L_0, L_1)$-Lipschitz \\
            (Theorem \ref{theorem:weak_minty_alpha01})
        \end{tabular} & \cellcolor{bgcolor2} $(0, 1)$ & \cellcolor{bgcolor2} $\frac{1}{2 \sqrt{2} K_0 + \left(2 \sqrt{2}K_1 + 2^{\nicefrac{3 (1 - \alpha)}{2}} K_2^{1 - \alpha} \right)\| F(x_k)\|^{\alpha}}$ &\cellcolor{bgcolor2} $\frac{\gamma_k}{2}$ \; \\[10pt]
        \hline
    \end{tabular}%
    }
    \begin{tablenotes}
        {\scriptsize
        \item [{\color{blue}(1)}] Convergence measure $\bullet$ strongly monotone: $\|x_K - x_*\|^2$, $\bullet$ monotone: $ \min_{0 \leq k \leq K}\|F(x_k)\|^2$, $\bullet$ weak minty: $ \min_{0 \leq k \leq K}\|F(\hx_k)\|^2$.
        \item [{\color{blue}(2)}] For $K_0, K_1, K_2$, check Proposition \ref{prop:equiv_formulation}. Note that, for $L_1 = 0$ we have $K_1 = K_2 = 0$.
        }
    \end{tablenotes}
    \end{threeparttable}
\end{table*}

\vspace{-3mm}
\section{On the $\alpha$-Symmetric $(L_0, L_1)$-Lipschitz Assumption}
We divide this section into two parts. In the first subsection, we present an equivalent reformulation of the $\alpha$-symmetric $(L_0, L_1)$-Lipschitz condition~\eqref{eq:(L0,L1)-Lipschitz} in the context of min-max optimization. In the second subsection, we provide some examples of operators that satisfy \eqref{eq:(L0,L1)-Lipschitz} and highlight its significance.

\subsection{Equivalent Formulation of $\alpha$-Symmetric $(L_0, L_1)$-Lipschitz Assumption}
In this subsection, we consider the min-max optimization problem $\min_{w_1} \max_{w_2} \mathcal{L}(w_1, w_2)$. The corresponding operator $F$ and Jacobian $\mathbf{J}$ are defined as
\begin{equation}\label{eq:jacobian}
\textstyle
    F(x) = \begin{bmatrix}
        \nabla_{w_1} \mathcal{L}(w_1, w_2) \\
        - \nabla_{w_2} \mathcal{L}(w_1, w_2)
    \end{bmatrix} \text{ and  } 
    \mathbf{J}(x) = \begin{bmatrix}
    \nabla^2_{w_1 w_1} \mathcal{L}(w_1, w_2) & \nabla^2_{w_2 w_1} \mathcal{L}(w_1, w_2) \\
    -\nabla^2_{w_1 w_2} \mathcal{L}(w_1, w_2) & -\nabla^2_{w_2 w_2} \mathcal{L}(w_1, w_2)
    \end{bmatrix},
\end{equation}
where $x = (w_1^{\top}, w_2^{\top})^{\top}$. Assuming that $F$ is $\alpha$-symmetric $(L_0, L_1)$-Lipschitz, we obtain the following theorem.

\begin{theorem}\label{thm:equiv_formulation}
    Suppose $F$ is the differentiable operator associated with the problem $\min_{w_1} \max_{w_2} \mathcal{L}(w_1, w_2)$. Then $F$ satisfies the $\alpha$-symmetric $(L_0, L_1)$-Lipschitz condition~\eqref{eq:(L0,L1)-Lipschitz} if and only if
    \begin{equation}\label{eq:equiv_formulation}
    \textstyle
        \|\mathbf{J}(x)\| \leq L_0 + L_1 \|F(x)\|^{\alpha}.
    \end{equation}
    Here $\mathbf{J}(x)$ is the Jacobian defined in \eqref{eq:jacobian} and $\| \mathbf{J}(x)\| = \sigma_{\max}(\mathbf{J}(x))$ i.e. maximum singular value of $\mathbf{J}(x)$. In particular, we have $\| \mathbf{J}(x)\| \leq L$ when operator $F$ is $L$-Lipschitz.
\end{theorem}
This result provides an equivalent characterization of the $\alpha$-symmetric $(L_0, L_1)$-Lipschitz condition~\eqref{eq:(L0,L1)-Lipschitz} for min-max optimization problems. In practice, it is often easier to verify~\eqref{eq:equiv_formulation} than to directly check~\eqref{eq:(L0,L1)-Lipschitz}. In Appendix~\ref{appendix:equiv_formulation}, we provide an example where we use Theorem \ref{thm:equiv_formulation} to verify if an operator satisfies \eqref{eq:(L0,L1)-Lipschitz}. 
\vspace{-2mm}
\subsection{Examples of $\alpha$-Symmetric $(L_0, L_1)$-Lipschitz Operators}\label{sec:examples}
To motivate the significance of this relaxed assumption~\eqref{eq:(L0,L1)-Lipschitz}, we present a few instances of $\alpha$-symmetric $(L_0, L_1)$-Lipschitz operators that highlight its advantages over the conventional $L$-Lipschitz assumption.\\
\newline
\textbf{Example 1}~\citep{gorbunov2024methods}: We start with an example from the minimization setting. Consider the logistic regression loss function $f(x) = \log \left( 1 + \exp{\left( - a^{\top}x \right)}\right)$. Then the corresponding gradient operator $F = \nabla f$ satisfies the $L$-Lipschitz assumption with $L = \|a\|^2$ and $\alpha$-symmetric $(L_0, L_1)$-Lipschitz assumption with $L_0 = 0, L_1 = \|a\|, \alpha = 1$. Therefore, when $\|a\| \gg 1$, the bound provided by $1$-symmetric $(L_0, L_1)$-Lipschitz can be significantly tighter than the one imposed by the $L$-Lipschitz condition. This example emphasizes the benefit of the $\alpha$-symmetric $(L_0, L_1)$-Lipschitz framework in scenarios where standard Lipschitz constants are overly pessimistic.

\textbf{Example 2}: Consider the operator $F(x) = (u_1^2, u_2^2)$ for $x = (u_1, u_2) \in \R^2$ with $x_* = (0, 0)$. Then we can show that
\begin{equation*}
\textstyle
    \|F(x) - F(y)\| \leq 2 \left\|F \left(\frac{x+y}{2} \right) \right\|^{\nicefrac{1}{2}} \|x - y\| \leq 2 \left\| \max_{\theta \in [0, 1]} F \left(\theta x + (1 - \theta) y \right) \right\|^{\nicefrac{1}{2}} \|x - y\|.
\end{equation*}
This establishes that $F$ is $\nicefrac{1}{2}$-symmetric $(0, 2)$-Lipschitz operator. However, this operator $F$ does not satisfy the standard $L$-Lipschitz assumption for any finite choice of $L$. We add the related details to Appendix \ref{appendix:examples}. Therefore, this example highlights the need for relaxed assumptions on operators beyond standard $L$-Lipschitz~\eqref{eq:L-Lipschitz}.

\textbf{Example 3.} Consider the following min-max optimization problem, for any $p > 1$,
\begin{equation*}
\textstyle
    \min_{w_1} \max_{w_2} \;
    \mathcal{L}(w_1, w_2)
    =
    \frac{1}{p+1} \|w_1\|^{p+1}
    + w_1^\top \B w_2
    - \frac{1}{p+1} \|w_2\|^{p+1}.
\end{equation*}
The corresponding operator $F(x)$ defined in~\eqref{eq:minmax_operator} is
$1$-symmetric
$\left(2\tau_0 + \|\M\|, 2^{\frac{2p^2 - 1}{p^2}} \tau_1 \right)$-Lipschitz,
for any choice of $\tau_1 > 0$ and $\tau_0 = \left( \frac{p-1}{\tau_1}\right)^{p-1}, \M = \begin{bmatrix}
            0 & \B \\
            -\B^\top & 0
        \end{bmatrix}$.
Moreover, $F$ is not $L$-Lipschitz for any finite $L$. We have added the proof in Appendix \ref{appendix:examples}. 

In Appendix~\ref{appendix:examples}, we also provide additional examples that illustrate cases where the operator associated with a general bilinearly coupled min-max optimization problem or an $N$-player game satisfies the $\alpha$-symmetric $(L_0, L_1)$-Lipschitz condition.

\section{Convergence Analysis}\label{sec:convergence_analysis}
\vspace{-2mm}
In this section, we present the convergence guarantees of \algname{EG} for solving monotone, strongly monotone, and weakly Minty operators. For strongly monotone operators, we have linear convergence, while for monotone and weak Minty operators, we provide sublinear convergence guarantees. To prove these results, we rely on the similar expression presented in~\cite{chen2023generalized} for the $(L_0, L_1)$-smooth minimization problem. For completeness, we include the proof for $\alpha$-symmetric $(L_0, L_1)$-Lipschitz operators in the Appendix. 
\begin{proposition} \label{prop:equiv_formulation}
    Suppose $F$ is $\alpha$-symmetric $(L_0, L_1)$-Lipschitz operator. Then, for $\alpha = 1$
    \begin{equation}\label{eq:alpha=1}
    \textstyle
        \| F(x) - F(y) \| \leq (L_0 + L_1 \| F(x)\|) \exp{(L_1 \|x - y\|)} \| x- y \|,
    \end{equation}
    and for $\alpha \in (0, 1)$ we have
    \begin{equation}\label{eq:alpha(0,1)}
    \textstyle
        \| F(x) - F(y) \| \leq \left(K_0 + K_1 \|F(x)\|^{\alpha} + K_2 \|x - y\|^{\nicefrac{\alpha}{1 - \alpha}} \right) \|x - y\|
    \end{equation}
    where $K_0 = L_0 (2^{\nicefrac{\alpha^2}{1 - \alpha}} + 1)$, $K_1 = L_1 \cdot 2^{\nicefrac{\alpha^2}{1 - \alpha}}$ and $K_2 = L_1^{\nicefrac{1}{1 - \alpha}} \cdot 2^{\nicefrac{\alpha^2}{1 - \alpha}} \cdot 3^{\alpha} (1 - \alpha)^{\nicefrac{\alpha}{1 - \alpha}}$.
\end{proposition}
Proposition \ref{prop:equiv_formulation} eliminates the maximum over $\theta \in [0, 1]$ in \eqref{eq:(L0,L1)-Lipschitz} and provides a simpler upper bound on the $\|F(x) - F(y)\|$. We divide the rest of the section into three subsections based on the structure of operators. Moreover, each of these subsections is divided into two parts depending on the value of $\alpha$, i.e. $\alpha = 1$ and $\alpha \in (0, 1)$.

\vspace{-2mm}
\subsection{Convergence Guarantees for Strongly Monotone Operators}
\vspace{-2mm}
In case of strongly monotone operators~\eqref{eq:strong_monotone}, we achieve linear convergence rates, analogous to those obtained under standard $L$-Lipschitz assumptions~\citep{tseng1995linear, mokhtari2020unified}. For $\alpha = 1$, operator $F$ satisfies the condition \eqref{eq:alpha=1}. To guarantee convergence for this class of operators, we use the \algname{EG} with step size $\gamma_k = \omega_k = \nicefrac{\nu}{\left( L_0 + L_1 \|F(x_k)\| \right)}$ and $\nu > 0$. \cite{gorbunov2024methods} used similar step sizes for the Gradient Descent algorithm to solve convex minimization problems. 
\begin{theorem}\label{theorem:1symm_strongmonotone}
    Suppose $F$ is $\mu$-strongly monotone and $1$-symmetric $(L_0, L_1)$-Lipschitz operator. Then \algname{EG} with step size $\gamma_k = \omega_k = \frac{\nu}{L_0 + L_1 \|F(x_k)\|}$ satisfy
    \begin{eqnarray*}
        \|x_{k+1} - x_*\|^2 & \leq & \left( 1 - \frac{\nu \mu}{L_0 \left( 1 + L_1 \exp{\left( L_1 \|x_0 - x_*\|\right)} \|x_0 - x_* \|\right)} \right)^{k+1} \|x_0 - x_*\|^2
    \end{eqnarray*}
    where $\nu > 0$ satisfy $ 1 - 2 \nu - \nu^2 \exp{2\nu} = 0$.
\end{theorem}
The equation $1 - 2 \nu - \nu^2 \exp{(2\nu)} = 0$ admits a positive solution, approximately $\nu \approx 0.363$. 
Specifically, to ensure $\|x_K - x_*\|^2 \leq \varepsilon$, we require $K = \mathcal{O} \left( \left(\frac{L_0}{\mu} + \frac{L_0 L_1 \|x_0 - x_*\| \exp{ \left(L_1\|x_0 - x_*\| \right)}}{\mu} \right) \log \frac{1}{\varepsilon} \right)$ iterations. When $L_1 = 0$, we recover the best-known results for the strongly monotone $L$-Lipschitz setting~\citep{tseng1995linear, mokhtari2020unified}. \citet{pmlr-v235-vankov24a} also studied constrained strongly monotone problems and obtained similar guarantees with an alternative step size scheme.

However, using a refined proof technique, we can eliminate the $\exp(L_1\|x_0 - x_*\|)$ term from the convergence rate and establish a tighter bound.
One of the intermediate steps of Theorem \ref{theorem:1symm_strongmonotone} is proving a lower bound on the step size $\gamma_k$, which can be very loose for large $k$. 
We now show that after a certain number of iterations $K'$~\eqref{eq:1symm_strongmonoton_noexp_eq2}, the operator norm satisfies $\|F(x_k)\| \leq \nicefrac{L_0}{L_1}$ for all $k \geq K'$, which implies $\gamma_k = \omega_k \geq \nicefrac{\nu}{2L_0}$ for all $k \geq K'$.

\begin{corollary}\label{corollary:1symm_strongmonotone}
    Suppose $F$ is a $\mu$-strongly monotone and $1$-symmetric $(L_0, L_1)$-Lipschitz operator. Then, \algname{EG} with step sizes $\gamma_k = \omega_k = \frac{\nu}{L_0 + L_1 \|F(x_k)\|}$ guarantees $\|x_{K+1} - x_*\|^2 \leq \varepsilon$ after at most
    \begin{equation}\label{corollary:1symm_strongmonotone_eq1}
    \textstyle
        K = \underbrace{\frac{2 L_0}{\nu \mu} \log \left( \frac{\|x_0 - x_*\|^2}{\varepsilon}\right)}_{\text{Term I}} + \underbrace{\frac{1}{\zeta \mu} \log \left( \frac{2L_1 \|x_0 - x_*\|^2}{\zeta^2 L_0} \right)}_{\text{Term II}} 
    \end{equation}
    iterations, where we have $\zeta \coloneqq \nicefrac{\nu}{L_0 \left( 1 + L_1 \exp\left( L_1 \|x_0 - x_*\|\right) \|x_0 - x_* \| \right)}$, and $\nu > 0$ satisfies $1 - 4 \nu - 2\nu^2 \exp(2\nu) = 0$. 
\end{corollary}
This result shows that to reach an accuracy of $\varepsilon > 0$, we need \eqref{corollary:1symm_strongmonotone_eq1} iterations. Importantly, Term II in \eqref{corollary:1symm_strongmonotone_eq1} is independent of $\varepsilon$, and the Term I of \eqref{corollary:1symm_strongmonotone_eq1} no longer depends on $\exp(L_1\|x_0 - x_*\|)$. Technically, Term II corresponds to the number of iterations required for the step sizes $\gamma_k$ and $\omega_k$ to exceed $\nicefrac{\nu}{2L_0}$, while Term I captures the iteration complexity of \algname{EG} with a fixed step size $\nicefrac{\nu}{2L_0}$.

Now, we investigate the behavior of $\alpha$-symmetric $(L_0, L_1)$-Lipschitz operators for $0 < \alpha < 1$. In this regime, we adopt a step size of the order $\mathcal{O}\left( \|F(x_k)\|^{-\alpha} \right)$ and prove the following result.

\begin{theorem}\label{theorem:1symm_strongmonotone_alhpha01}
    Suppose $F$ is $\mu$-strongly monotone and $\alpha$-symmetric $(L_0, L_1)$-Lipschitz operator with $\alpha \in (0, 1)$. Then \algname{EG} with $\gamma_k = \omega_k = \frac{\nu}{2K_0 + \left( 2 K_1 + 2^{1 - \alpha} K_2^{1 - \alpha} \right) \| F(x_k)\|^{\alpha}}$ satisfy
    \begin{eqnarray*}
    \textstyle
        \|x_{k+1} - x_*\|^2 \leq \left( 1 - \frac{\nu \mu}{2 K_0 + (2 K_1 + 2^{1 - \alpha} K_2^{1 - \alpha})(K_0 + K_2 \|x_0 - x_*\|^{\nicefrac{\alpha}{1 - \alpha}})^{\alpha} \|x_0 - x_*\|^{\alpha}}\right)^{k+1} \|x_0 - x_*\|^2
    \end{eqnarray*}
    where $\nu \in (0, 1)$ is a constant such that $1 - \nu - \nu^2 = 0$.
\end{theorem}
This result establishes linear convergence. In particular, to ensure $\|x_{K} - x_*\|^2 \leq \varepsilon$, it suffices to run $K = \mathcal{O} \left(\left( \frac{K_0}{\mu} + \frac{(K_1 K_2^{\alpha} + K_2)\|x_0 - x_*\|^{\nicefrac{\alpha}{1- \alpha}}}{\mu}\right) \log \frac{1}{\varepsilon} \right)$ iterations. Compared to the $L$-Lipschitz setting, the bound here includes an additional dependence on $\|x_0 - x_*\|^{\nicefrac{\alpha}{1 - \alpha}}$, which grows larger as $\alpha \to 1$. 

\subsection{Convergence Guarantees for Monotone Operators} 
In this subsection, we focus on the monotone operators~\eqref{eq:monotone}. Here we provide the first analysis for the monotone $1$-symmetric $(L_0, L_1)$-Lipschitz operators.

\begin{theorem}\label{theorem:1symm_monotone}
    Suppose $F$ is monotone and $1$-symmetric $(L_0, L_1)$-Lipschitz operator. Then \algname{EG} with step size $\gamma_k = \omega_k = \frac{\nu}{L_0 + L_1 \|F(x_k)\|}$ satisfy 
    \begin{equation}\label{eq:1symm_monotone}
    \textstyle
        \min_{0 \leq k \leq K} \|F(x_k)\|^2 \leq \frac{2L_0^2 \left( 1 + L_1 \exp{\left( L_1 \|x_0 - x_*\|\right) \|x_0 - x_*\|}\right)^2 \|x_0 - x_*\|^2}{\nu^2(K+1)}.
    \end{equation}
    where $\nu \exp{\nu} = \nicefrac{1}{\sqrt{2}}$.
\end{theorem}
Note that the solution of $\nu \exp{\nu} = \nicefrac{1}{\sqrt{2}}$ is approximately $0.45$. Hence, this result proves sublinear convergence of \algname{EG} when $F$ is monotone. Moreover, \eqref{eq:1symm_monotone} implies, \algname{EG} will  need $K = \mathcal{O} \left( \frac{L_0^2 \|x_0 - x_*\|^2}{\varepsilon} + \frac{L_0^2 L_1^2 \exp{\left(2 L_1 \|x_0 - x_*\| \right) \|x_0 - x_*\|^4}}{\varepsilon} \right)$ iterations to get $\|F(x_k)\|^2 \leq \varepsilon$ for some $k \leq K$. Therefore the convergence rate exponentially depends on $\|x_0 - x_*\|$ when $L_1 > 0$. This shows that $1$-symmetric $(L_0, L_1)$-Lipschitz operators potentially require more iterations of \algname{EG} compared to $L$-Lipschitz operators when initialization $x_0$ is far from $x_*$. However, \eqref{eq:1symm_monotone} recovers the best known dependence on $\|x_0 - x_*\|$ as a special case when $L_1 = 0$, i.e. $F$ is a standard Lipschitz operator \citep{gorbunov2022extragradient}.

Theorem \ref{theorem:1symm_monotone} shows that the \algname{EG}'s convergence rate has an extra term $\exp{(L_1\|x_0 - x_*\|)}$ compared to the results of the Lipschitz setting. One of the intermediate steps in this proof involves an upper bound on $\sum_{k = 0}^K \gamma_k^2 \|F(x_k)\|^2$ (see \eqref{eq:1symm_monotone_eq3} in Appendix \ref{appendix:convergence_analysis}). Then the simple approach is to get a lower bound on $\gamma_k^2$ for all $k$ and derive \eqref{eq:1symm_monotone}. This lower bound on $\gamma_k^2$ involves the $\exp{(L_1\|x_0 - x_*\|)}$ term (see \eqref{eq:1symm_monotone_eq2} in Appendix \ref{appendix:convergence_analysis}) and can be potentially very small. However, it is possible to eliminate this exponential dependence using a refined proof technique.
\begin{theorem}\label{theorem:1symm_monotone_noexp}
    Suppose $F$ is monotone and $1$-symmetric $(L_0, L_1)$-Lipschitz operator. Then \algname{EG} with step size $\gamma_k = \omega_k = \frac{\nu}{L_0 + L_1 \|F(x_k)\|}$ satisfy 
    \begin{equation*}
    \textstyle
        \min_{0 \leq k \leq K} \| F(x_k)\| \leq \frac{\sqrt{2} L_0 \|x_0 - x_*\|}{\nu \sqrt{K+1} - \sqrt{2} L_1 \|x_0 - x_*\|}
    \end{equation*}
    where $\nu \exp{\nu} = \nicefrac{1}{\sqrt{2}}$ and $K+1 \geq \frac{2L_1^2 \|x_0 - x_*\|^2}{\nu^2}$.
\end{theorem}
Note that to obtain this convergence guarantee, a sufficiently large number of iterations is required, specifically $K+1 \geq \nicefrac{\left(2L_1^2 \|x_0 - x_*\|^2 \right)}{\nu^2}$. \cite{gorbunov2024methods} employed a similar proof technique to eliminate the exponential dependence on the initial distance $\exp(L_1\|x_0 - x_*\|)$ in the context of the Adaptive Gradient method.

Next we state our result for $\alpha$-symmetric $(L_0, L_1)$-Lipschitz monotone operator with $\alpha \in (0, 1).$
\begin{theorem}\label{theorem:alpha01}
    Suppose $F$ is monotone and $\alpha$-symmetric $(L_0, L_1)$-Lipschitz operator with $\alpha \in (0, 1)$. Then \algname{EG} with $\gamma_k = \omega_k = \frac{1}{2 \sqrt{2} K_0  + \left( 2\sqrt{2} K_1 + 2^{\nicefrac{3 (1 - \alpha)}{2}} K_2^{1 - \alpha} \right) \|F(x_k)\|^{\alpha}}$ satisfy 
    \begin{eqnarray*}
    \textstyle
        \min_{0 \leq k \leq K} \|F(x_k)\|^2 \leq \frac{16 \left( K_0 + (K_1 + 2^{\nicefrac{-3}{2}} K_2^{1 - \alpha}) (K_0 + K_2 \|x_0 - x_*\|^{\nicefrac{\alpha}{1 - \alpha}})^{\alpha}\|x_0 - x_*\|^{\alpha}\right)^2 \|x_0 - x_*\|^2}{K+1}.
    \end{eqnarray*}
\end{theorem}
This theorem establishes a sublinear convergence rate for $\alpha \in (0, 1)$. In the special case where $L_1 = 0$ (i.e., the standard $L$-Lipschitz setting), we have $K_1 = K_2 = 0$ by Proposition~\ref{prop:equiv_formulation}. Thus, our result recovers the best-known rate $\mathcal{O} \left( \nicefrac{L_0^2 \|x_0 - x_*\|^2}{K+1} \right)$ from~\citet{gorbunov2022extragradient}. On the other hand, when $L_1 > 0$, we obtain a convergence rate of $\mathcal{O} \left( \nicefrac{\|x_0 - x_*\|^{\frac{2 + 4 \alpha - 2\alpha^2}{1 - \alpha}}}{K+1}\right)$. Furthermore, as $\alpha \to 0$—which corresponds again to the $L$-Lipschitz setting—our step sizes $\gamma_k$ and $\omega_k$ become constant, and we recover the standard convergence rate $\mathcal{O} \left( \nicefrac{\|x_0 - x_*\|^2}{K+1}\right)$. This matches the classical result for monotone $L$-Lipschitz operators up to constants, emphasizing the tightness of our analysis.

\subsection{Local Convergence Guarantees for Weak Minty Operators}
Beyond the monotone operators, it is also possible to provide convergence for weak Minty operators~\eqref{eq:weak_minty} under some restrictions on $\rho > 0$. In contrast to the monotone problems where we used the same extrapolation and update step $\gamma_k, \omega_k$, here we use smaller update step size $\omega_k$. Specifically, we employ $\omega_k = \nicefrac{\gamma_k}{2}$, similar to \citet{diakonikolas2021efficient} for handling weak Minty $L$-Lipschitz operators. 
\begin{theorem}\label{theorem:weak_minty_alpha1}
     Suppose $F$ is weak Minty and $1$-symmetric $(L_0, L_1)$-Lipschitz assumption. Moreover we assume 
    \begin{equation}\label{eq:restriction_rho}
    \textstyle
        \Delta_1 \eqdef \frac{\nu}{L_0 \left( 1 + L_1 \|x_0 - x_*\| e^{L_1 \|x_0 - x_*\|}\right)} - 4 \rho > 0.
    \end{equation} Then \algname{EG} with step size $\gamma_k = \frac{\nu}{L_0 + L_1 \|F(x_k)\|}$ and $\omega_k = \nicefrac{\gamma_k}{2}$ satisfies
    \begin{equation}\label{eq:monotone_asymmetric}
    \textstyle
        \min_{0 \leq k \leq K} \|F(\hx_k)\|^2 \leq \frac{4 L_0 \left( 1 + L_1 \exp{\left( L_1 \|x_0 - x_*\|\right) \|x_0 - x_*\|}\right) \|x_0 - x_*\|^2}{\nu \Delta_1 (K+1)} 
    \end{equation}
    where $\nu \exp{\nu} = 1$.
\end{theorem}
To the best of our knowledge, this is the first result establishing convergence guarantees for weak Minty, $\alpha$-symmetric $(L_0, L_1)$-Lipschitz operators. Similar to the monotone case, we obtain a sublinear convergence rate for weak Minty operators. However, the condition in~\eqref{eq:restriction_rho} indicates that the initialization point $x_0$ must be sufficiently close to the solution $x_*$ in order to ensure convergence. Consequently, Theorem~\ref{theorem:weak_minty_alpha1} only provides a local convergence guarantee.

In the special case where $L_1 = 0$, i.e., the standard $L$-Lipschitz setting, condition~\eqref{eq:restriction_rho} reduces to the simpler requirement $\rho < \nicefrac{\nu}{4L_0}$. Similar assumptions on $\rho$ have been made in prior works such as~\citet{diakonikolas2021efficient} and~\citet{pethick2023escaping} for the $L$-Lipschitz weak Minty setting. Finally, we extend our analysis to the case $\alpha \in (0, 1)$, and present a corresponding theorem establishing sublinear convergence under analogous restrictions on $\rho$.

\begin{theorem}\label{theorem:weak_minty_alpha01}
    Suppose $F$ is weak Minty and $\alpha$-symmetric $(L_0, L_1)$-Lipschitz operator with $\alpha \in (0, 1)$. Moreover we assume 
    \begin{equation}\label{eq:restriction_rho_alpha}
    \textstyle
        \Delta_{\alpha} \eqdef \frac{1}{2 \sqrt{2} K_0 + 2 \sqrt{2} (K_1 + 2^{\nicefrac{-3}{2}} K_2^{1 - \alpha}) (K_0 + K_2 \|x_0 - x_*\|^{\nicefrac{\alpha}{1 - \alpha}})^{\alpha} \|x_0 - x_*\|^{\alpha}} - 4 \rho > 0.
    \end{equation} Then \algname{EG} with step size $\gamma_k = \frac{1}{2 \sqrt{2} K_0 + \left(2 \sqrt{2}K_1 + 2^{\nicefrac{3 (1 - \alpha)}{2}} K_2^{1 - \alpha} \right) \|F(x_k)\|^{\alpha}}$ and $\omega_k = \nicefrac{\gamma_k}{2}$ satisfy 
    \begin{eqnarray*}
    \textstyle
        \min_{0 \leq k \leq K} \|F(\hx_k)\|^2 \leq \frac{4 \left( K_0 + \left(K_1 + 2^{\nicefrac{-3}{2}} K_2^{1 - \alpha} \right) (K_0 + K_2 \|x_0 - x_*\|^{\nicefrac{\alpha}{1 - \alpha}})^{\alpha} \|x_0 - x_*\|^{\alpha}\right) \|x_0 - x_*\|^2}{\Delta_{\alpha}(K+1)}.
    \end{eqnarray*}
\end{theorem}

\section{Numerical Experiments}\label{sec:experiments}

In this section, we conduct experiments to validate the efficiency of our proposed step size strategy $\gamma_k = \frac{1}{c_0 + c_1 \|F(x_k)\|^{\alpha}}$ with $\alpha = 1$. In the first experiment, we compare our step size choice with that of \citet{pmlr-v235-vankov24a} on a strongly monotone problem, and in the second experiment, we make a comparison with the constant step size strategy for solving a monotone problem. Finally, we evaluate our scheme for solving the GlobalForsaken problem from \citet{pethick2023escaping}. All experiments in this work were conducted using a personal MacBook with an Apple M3 chip and 16GB of RAM. We provide the code for all
of our experiments at \href{https://github.com/isayantan/L0L1extragradient}{https://github.com/isayantan/L0L1extragradient}.

\begin{figure}[h]
\centering
\begin{subfigure}[b]{0.46\textwidth}
    \centering
    \includegraphics[width=.9\textwidth]{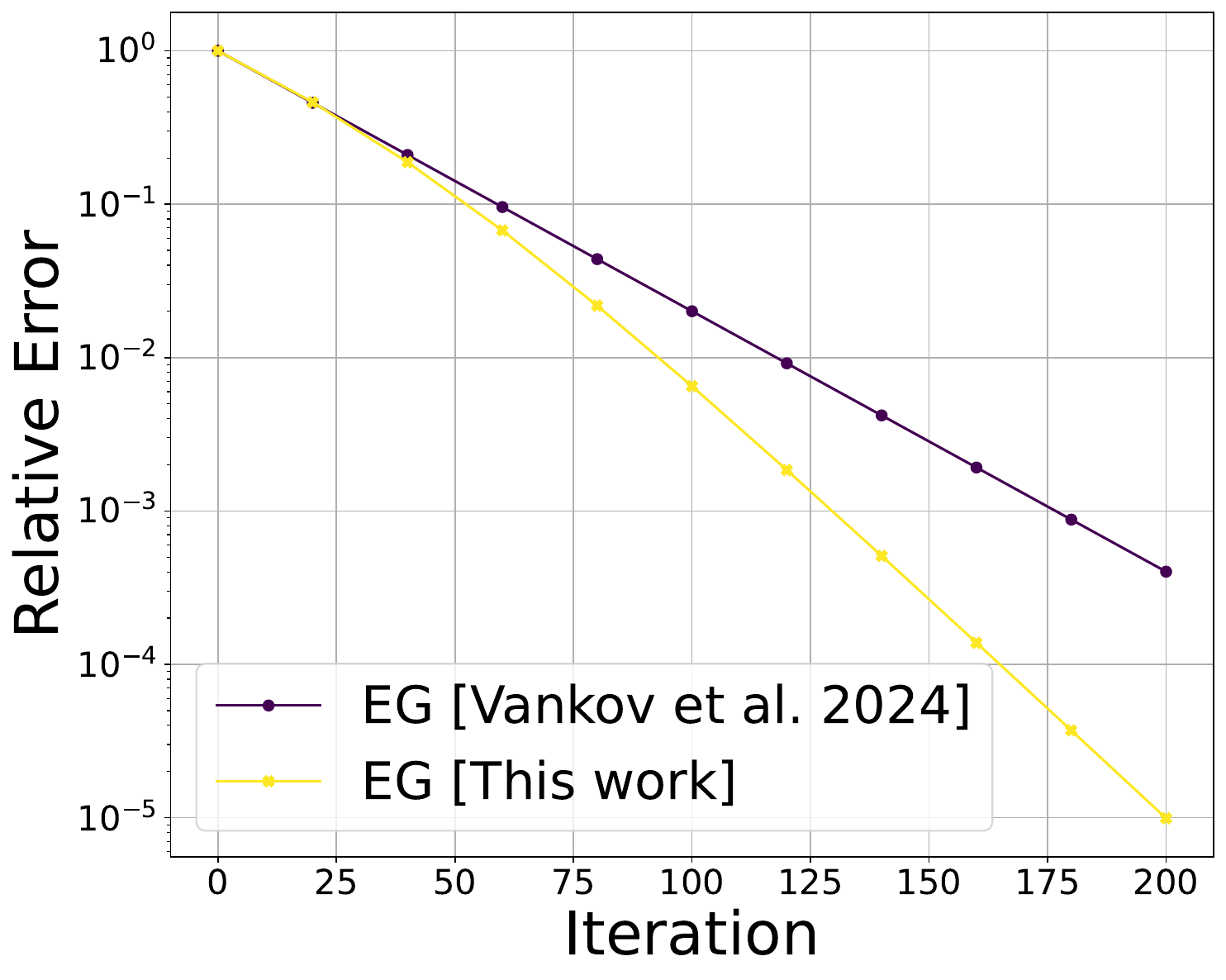}
    \caption{}\label{fig:L0L1comparison_opt_dist}
\end{subfigure}
\hfill
\begin{subfigure}[b]{0.46\textwidth}
    \centering
    \includegraphics[width=.9\textwidth]{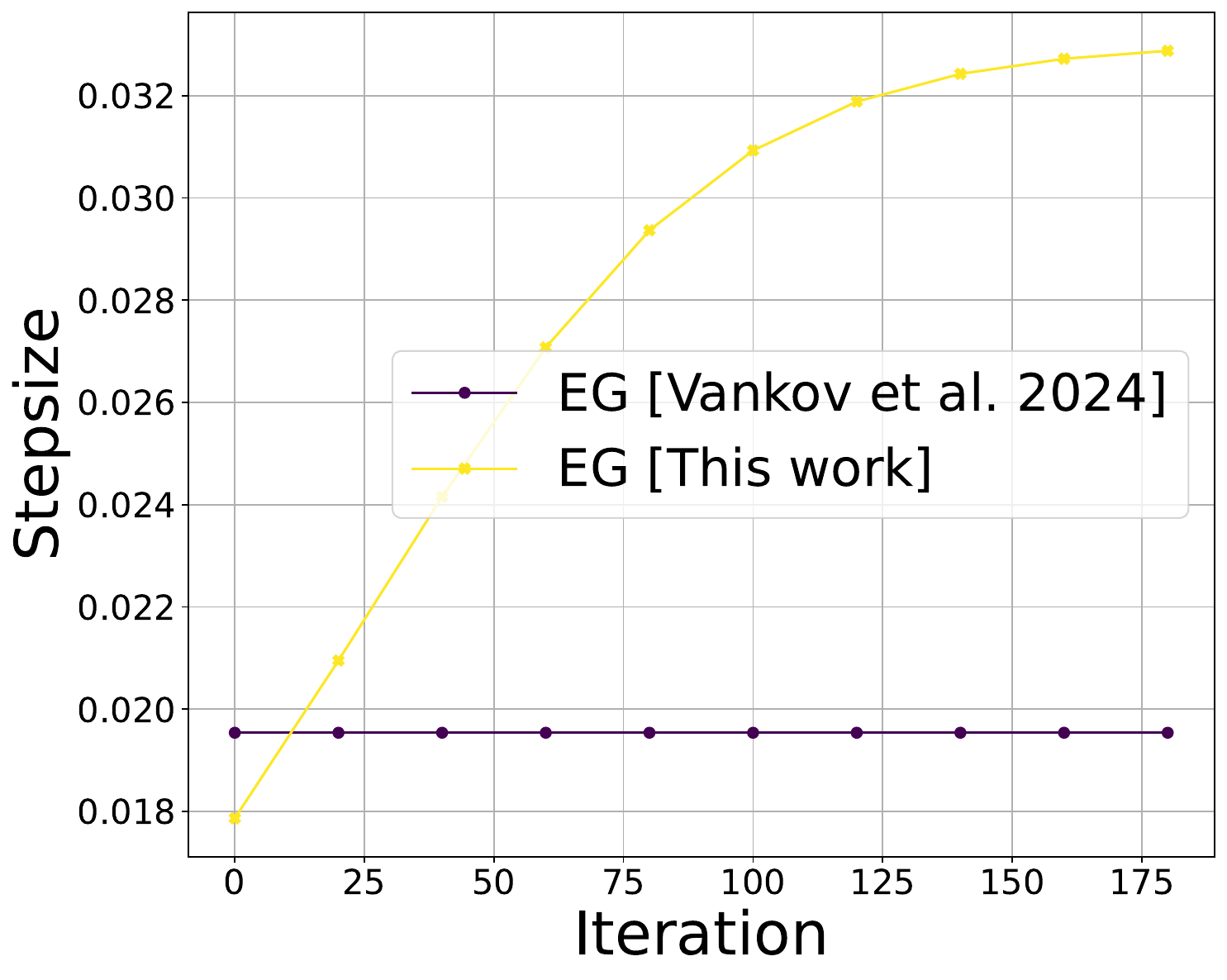}
    \caption{}\label{fig:L0L1comparison_stepsize}
\end{subfigure}

    \caption{\small In Figures~\ref{fig:L0L1comparison_opt_dist} and~\ref{fig:L0L1comparison_stepsize}, we compare our proposed adaptive step size strategy with that of \citet{pmlr-v235-vankov24a}. We report the relative error and the magnitude of the step size over iterations.
    }\label{fig:monotone_and_vankov}
\end{figure}

\begin{figure}[h]

\begin{subfigure}[b]{0.46\textwidth}
    \centering
    \includegraphics[width=.9\textwidth]{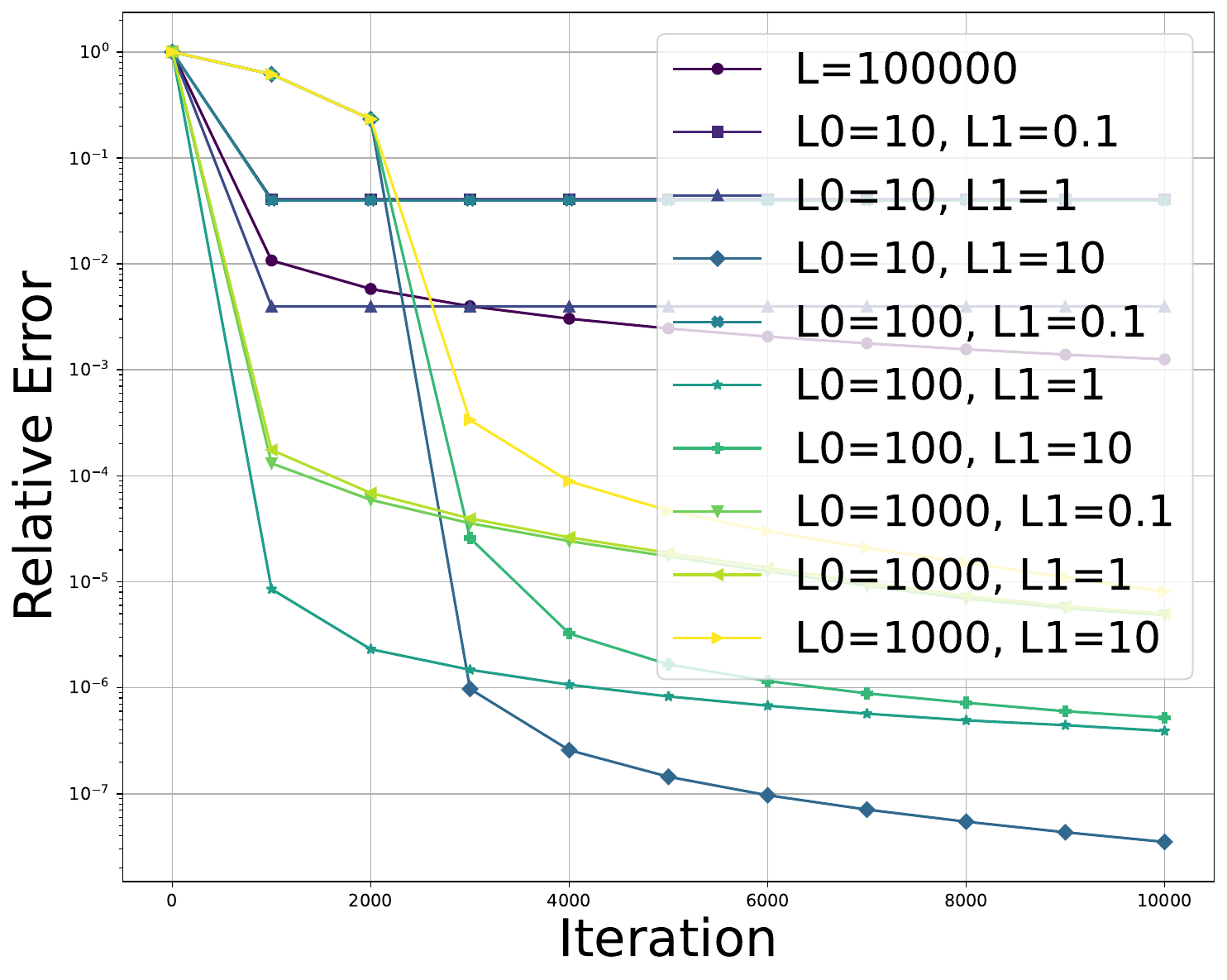}
    \caption{}\label{fig:monotone_cubic_relative_error}
\end{subfigure}
\hfill
\begin{subfigure}[b]{0.46\textwidth}
    \centering
    \includegraphics[width=.9\textwidth]{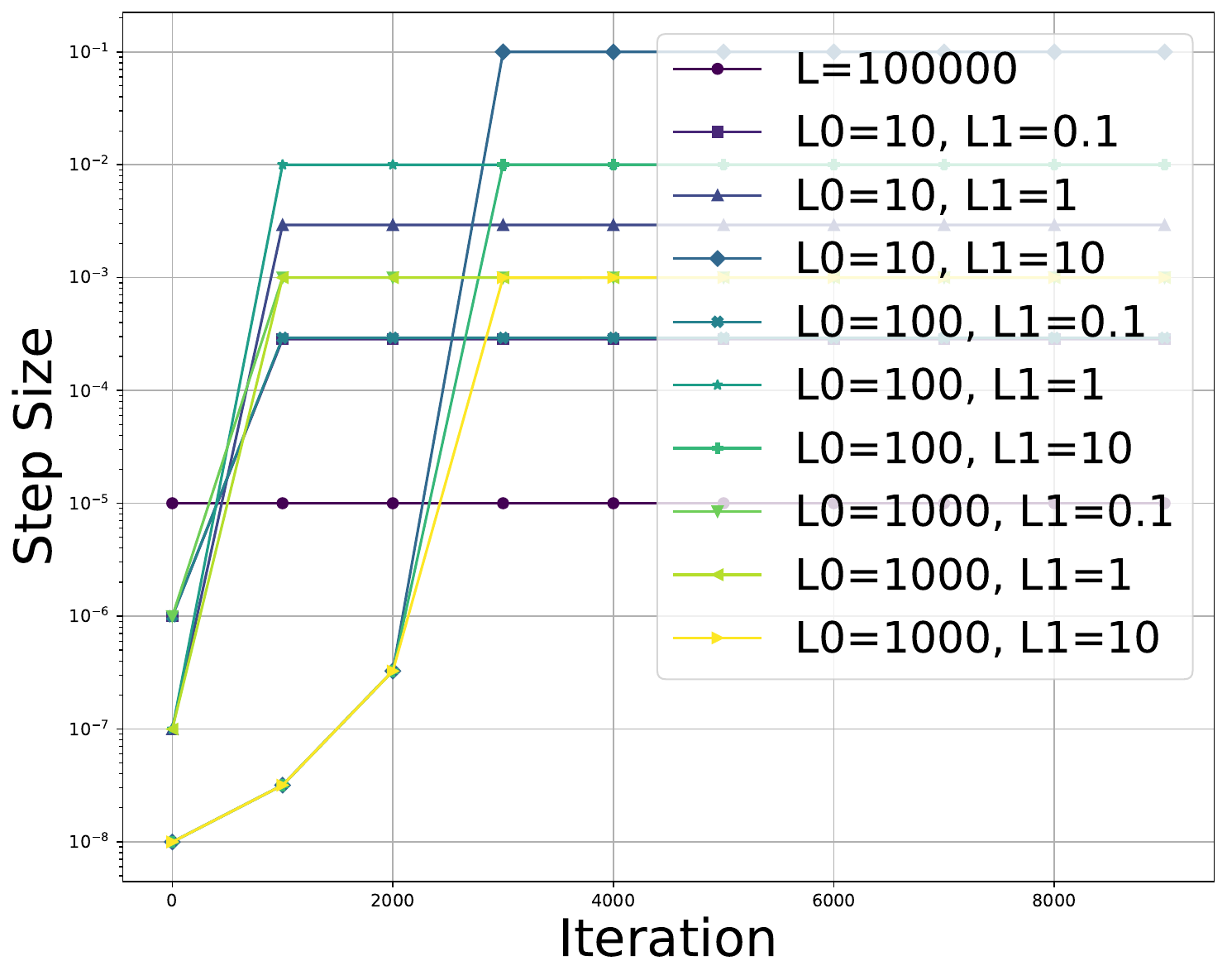}
    \caption{}\label{fig:monotone_cubic_step_size}
\end{subfigure}
    \caption{\small In Figures~\ref{fig:monotone_cubic_relative_error} and~\ref{fig:monotone_cubic_step_size}, we evaluate the performance of the \algname{EG} method on the problem in~\eqref{eq:min_max_cubic_Rd}, using both a constant step size and the $(L_0, L_1)$-adaptive step size. We report the relative error and the magnitude of the step size over iterations.
    }\label{fig:monotone_and_vankov}
\end{figure}

\textbf{Performance on a Strongly Monotone Problem.}
In this experiment, we compare our theoretical step sizes with those from \cite{pmlr-v235-vankov24a}. Here, we implement \algname{EG} for solving the operator $F(x) = (\text{sign}\left(u_1 \right) \left|u_1 \right| + u_2, \text{sign}\left(u_2 \right) \left|u_2 \right| - u_1)$.
This problem has constants $L_0 = 1 + 2 \sqrt{2}$ and $L_1 = 2\sqrt{2}$. For our method, we use $\gamma_k = \omega_k =  \frac{\nu}{L_0 + L_1 \|F(x_k)\|}$ while for \algname{EG}~\citep{pmlr-v235-vankov24a} we use stepsize $\gamma_k = \omega_k = \min \left\{ \frac{1}{4 \mu}, \frac{1}{2 \sqrt{2}e L_0}, \frac{1}{2 \sqrt{2}e L_1 \|F(x_k)\|} \right\}$. In Figure \ref{fig:L0L1comparison_opt_dist}, we plot the relative error $\frac{\|x_k - x_*\|^2}{\|x_0 - x_*\|^2}$ on the $y$-axis while number of iterations on the $x$-axis. We find that our proposed step size outperforms that of \cite{pmlr-v235-vankov24a}. Moreover, in Figure \ref{fig:L0L1comparison_stepsize}, we compare the magnitude of the step size and how it evolves over the iterations. We find that the step size of \cite{pmlr-v235-vankov24a} remains constant at approximately $0.02$, whereas our proposed step size increases to a value larger than $0.032$. These experiments highlight the efficiency of our proposed step size. 

\textbf{Performance on a Monotone Problem.} Here we consider the following min-max optimization problem
\begin{eqnarray}\label{eq:min_max_cubic_Rd}
\textstyle
    \min_{w_1 \in \R^d} \max_{w_2 \in \R^d} \mathcal{L}(w_1, w_2) = \frac{1}{3} \left( w_1^{\top} \A w_1 \right)^{\nicefrac{3}{2}} + w_1^{\top} \B w_2 - \frac{1}{3} \left(w_2^{\top} \C w_2 \right)^{\nicefrac{3}{2}}.
\end{eqnarray}
where $\A, \B, \C \in \R^{d \times d}$ are positive definite matrices. Note that, when $d = 1$, and $\A, \B, \C$ are just scalars equal to $1$, this problem reduces to \eqref{eq:min_max_cubic}.
The corresponding operator of this problem is given by 
\begin{equation*}
\textstyle
    F(x)  
    = \begin{bmatrix}
        \nabla_{w_1} \mathcal{L}(w_1, w_2) \\
        - \nabla_{w_2} \mathcal{L}(w_1, w_2)
    \end{bmatrix}
    = \begin{bmatrix}
        \left( w_1^{\top} \A w_1\right)^{\nicefrac{1}{2}} \A w_1 + \B w_2 \\
        \left( w_2^{\top} \C w_2\right)^{\nicefrac{1}{2}} \C w_2 - \B^{\top} w_1
    \end{bmatrix}.
\end{equation*} 
Furthermore, we show that $\mathcal{L}$ is convex-concave and has an equilibrium only at $w_1, w_2 = 0 \in \R^d$ (check Appendix \ref{appendix:num_exp}). To solve~\eqref{eq:min_max_cubic_Rd}, we implement the \algname{EG} method using two types of step size strategies: (1) a constant step size $\gamma_k = \omega_k = \nicefrac{1}{c}$, and (2) an adaptive step size $\gamma_k = \omega_k = \nicefrac{1}{\left(c_0 + c_1 \|F(x_k)\| \right)}$. For the constant step size \algname{EG}, we perform a grid search over $c \in \{10^2, 10^3, 10^4, 10^5, 10^6, 10^7\}$. We find that $c = 10^5$ yields the best performance: larger values lead to slower convergence, while smaller values cause divergence. Figures~\ref{fig:monotone_cubic_relative_error} and~\ref{fig:monotone_cubic_step_size} present the relative error and step size for the case $c = 10^5$. For our adaptive \algname{EG} method, we perform a grid search over $c_0 \in \{10, 100, 1000\}$ and $c_1 \in \{0.1, 1, 10\}$, evaluating all $9$ possible combinations. The performance of all adaptive variants is plotted in Figures~\ref{fig:monotone_cubic_relative_error} and~\ref{fig:monotone_cubic_step_size}. We observe that most combinations outperform the constant step size \algname{EG}, with $(c_0, c_1) = (10, 10)$ achieving the best results (see Figure \ref{fig:monotone_cubic_relative_error}).

Interestingly, while the adaptive \algname{EG} starts with smaller step sizes compared to the constant step size \algname{EG}, its step sizes increase over time and eventually surpass those of the constant step size approach (see Figure \ref{fig:monotone_cubic_step_size}). This highlights the practical effectiveness of our proposed method in handling non-Lipschitz operators such as the one in~\eqref{eq:min_max_cubic_Rd}.

\textbf{Performance on a Weak Minty Problem.} Here we consider the unconstrained GlobalForsaken problem from \citet{pethick2023escaping} given by
\begin{equation}\label{eq:globalforsaken}
    \min_{w_1 \in \R} \max_{w_2 \in \R} \mathcal{L}(w_1, w_2) \eqdef w_1 w_2 + \psi(w_1) - \psi(w_2), 
\end{equation}
where $\psi(w) = \frac{2w^6}{21} - \frac{w^4}{3} + \frac{w^2}{3}$. As shown in~\citet{pethick2023escaping}, the saddle-point problem in~\eqref{eq:globalforsaken} admits a global Nash equilibrium at $(w_1, w_2) = (0, 0)$ and satisfies the weak Minty condition~\eqref{eq:weak_minty} 
\begin{figure}[h]
    \vspace{-3mm}
    \begin{minipage}[t]{0.56\textwidth}
    \vspace{0pt}
          with parameter $\rho \approx 0.119732$. We implement \algname{AdaptiveEG+}, \algname{EG+} and \algname{EG} with our step size strategy to solve this problem. For each algorithm, we perform step size tuning on a a grid of $\{10^{-5}, 10^{-4}, \cdots, 10^2 \}$. We observe that both \algname{AdaptiveEG+} and \algname{EG+} perform best with a fixed step size of $\gamma_k = 0.1$. For our method, we set the step size parameters as $(c_0, c_1) = (1, 1)$. In \Cref{fig:globalforsaken}, we present the trajectory plots of these algorithms, all initialized at $(w_1, w_2) = (1, 1)$. Our findings indicate that all algorithms eventually converge to the equilibrium $(0, 0)$, but the convergence of our method is significantly faster. This demonstrates the advantage of our step size strategy in solving challenging problems that satisfy only weak Minty conditions.
    \end{minipage}    \hfill
    \begin{minipage}[t]{0.4\textwidth}
    \vspace{0pt}
        \centering
        \includegraphics[width=\linewidth]{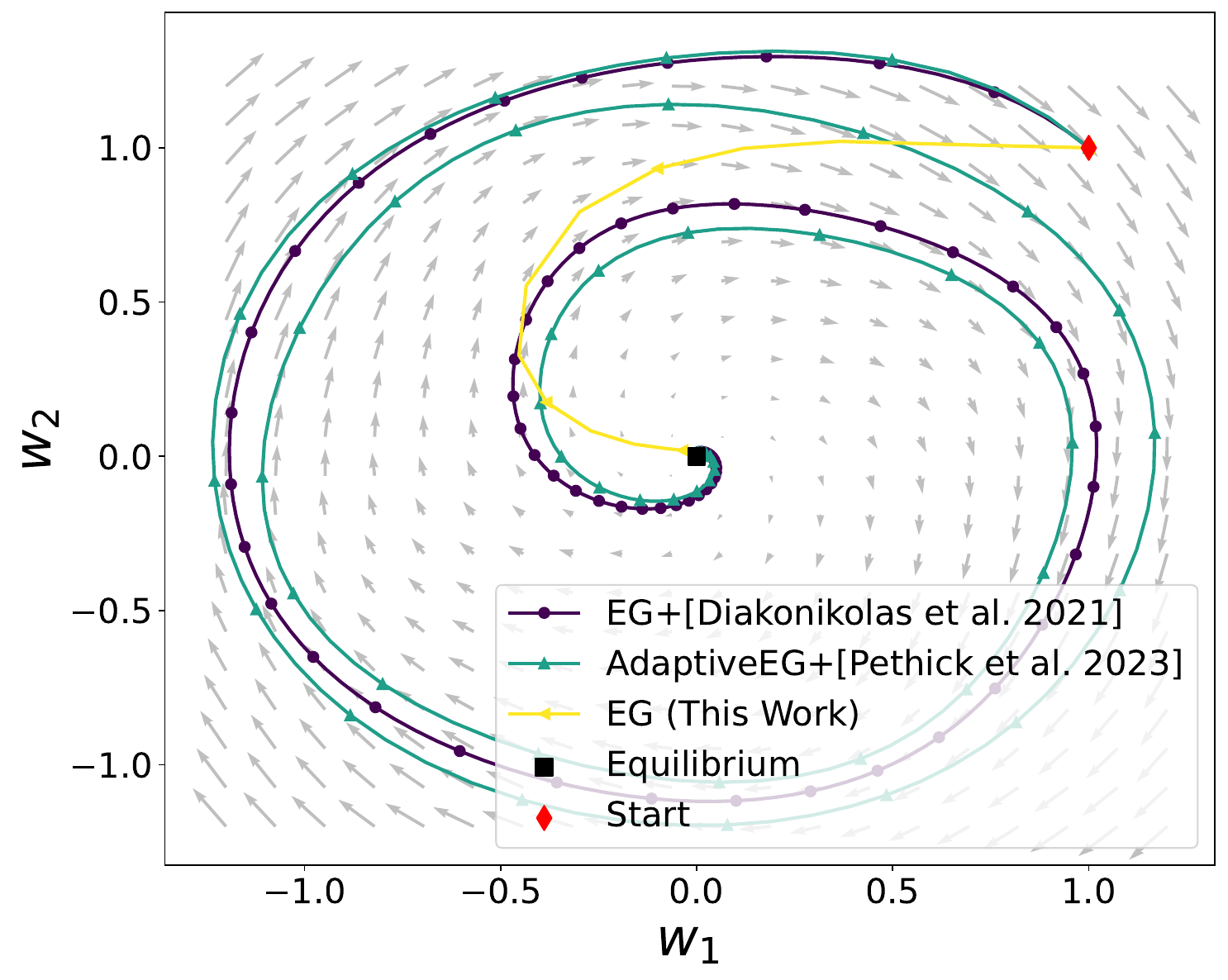}
        \caption{\small Trajectories of algorithms for solving problem~\eqref{eq:globalforsaken}.}\label{fig:globalforsaken}
    \end{minipage}
\end{figure}

\section{Conclusion} 
This work extends the analysis of the \algname{EG} method to a broader class of $\alpha$-symmetric $(L_0, L_1)$-Lipschitz operators. We establish new convergence guarantees for strongly monotone, monotone, and weak Minty settings, supported by a novel adaptive step size rule.

A limitation of our current analysis is that it focuses solely on root-finding problems and does not handle the constrained setup. However, the results included in this work advance the theoretical understanding of \algname{EG} and open several promising directions, including extensions to constrained and stochastic settings, as well as the analysis of optimistic gradient methods under this relaxed assumption. Furthermore, another important research direction is the estimation of the unknown constants $\alpha, L_0, L_1$, which would pave the way for fully adaptive algorithms.

\newpage
\small{\bibliography{bibfile}}

\begin{thebibliography}{51}
\providecommand{\natexlab}[1]{#1}
\providecommand{\url}[1]{\texttt{#1}}
\expandafter\ifx\csname urlstyle\endcsname\relax
  \providecommand{\doi}[1]{doi: #1}\else
  \providecommand{\doi}{doi: \begingroup \urlstyle{rm}\Url}\fi

\bibitem[Ahn et~al.(2024)Ahn, Cheng, Song, Yun, Jadbabaie, and Sra]{ahn2023linear}
K.~Ahn, X.~Cheng, M.~Song, C.~Yun, A.~Jadbabaie, and S.~Sra.
\newblock Linear attention is (maybe) all you need (to understand transformer optimization).
\newblock \emph{ICLR}, 2024.

\bibitem[Antonakopoulos et~al.(2021)Antonakopoulos, Belmega, and Mertikopoulos]{antonakopoulos2020adaptive}
K.~Antonakopoulos, E.~V. Belmega, and P.~Mertikopoulos.
\newblock Adaptive extra-gradient methods for min-max optimization and games.
\newblock \emph{ICLR}, 2021.

\bibitem[Balduzzi et~al.(2018)Balduzzi, Racaniere, Martens, Foerster, Tuyls, and Graepel]{balduzzi2018mechanics}
D.~Balduzzi, S.~Racaniere, J.~Martens, J.~Foerster, K.~Tuyls, and T.~Graepel.
\newblock The mechanics of n-player differentiable games.
\newblock In \emph{ICML}, 2018.

\bibitem[Bauschke et~al.(2017)Bauschke, Bolte, and Teboulle]{bauschke2017descent}
H.~H. Bauschke, J.~Bolte, and M.~Teboulle.
\newblock A descent lemma beyond lipschitz gradient continuity: first-order methods revisited and applications.
\newblock \emph{Mathematics of Operations Research}, 2017.

\bibitem[Beznosikov et~al.(2022)Beznosikov, Dvurechenskii, Koloskova, Samokhin, Stich, and Gasnikov]{beznosikov2022decentralized}
A.~Beznosikov, P.~Dvurechenskii, A.~Koloskova, V.~Samokhin, S.~U. Stich, and A.~Gasnikov.
\newblock Decentralized local stochastic extra-gradient for variational inequalities.
\newblock \emph{NeurIPS}, 2022.

\bibitem[Beznosikov et~al.(2023)Beznosikov, Gorbunov, Berard, and Loizou]{beznosikov2023stochastic}
A.~Beznosikov, E.~Gorbunov, H.~Berard, and N.~Loizou.
\newblock Stochastic gradient descent-ascent: Unified theory and new efficient methods.
\newblock In \emph{AISTATS}, 2023.

\bibitem[B{\"o}hm(2023)]{bohm2022solving}
A.~B{\"o}hm.
\newblock Solving nonconvex-nonconcave min-max problems exhibiting weak minty solutions.
\newblock \emph{TMLR}, 2023.

\bibitem[Brown et~al.(2020)Brown, Bakhtin, Lerer, and Gong]{brown2020combining}
N.~Brown, A.~Bakhtin, A.~Lerer, and Q.~Gong.
\newblock Combining deep reinforcement learning and search for imperfect-information games.
\newblock \emph{NeurIPS}, 2020.

\bibitem[Chambolle and Pock(2016)]{chambolle2016ergodic}
A.~Chambolle and T.~Pock.
\newblock On the ergodic convergence rates of a first-order primal--dual algorithm.
\newblock \emph{Mathematical Programming}, 2016.

\bibitem[Chen et~al.(2023)Chen, Zhou, Liang, and Lu]{chen2023generalized}
Z.~Chen, Y.~Zhou, Y.~Liang, and Z.~Lu.
\newblock Generalized-smooth nonconvex optimization is as efficient as smooth nonconvex optimization.
\newblock In \emph{ICML}, 2023.

\bibitem[Choudhury et~al.(2024)Choudhury, Gorbunov, and Loizou]{choudhury2024single}
S.~Choudhury, E.~Gorbunov, and N.~Loizou.
\newblock Single-call stochastic extragradient methods for structured non-monotone variational inequalities: Improved analysis under weaker conditions.
\newblock \emph{NeurIPS}, 2024.

\bibitem[Daskalakis et~al.(2017)Daskalakis, Ilyas, Syrgkanis, and Zeng]{daskalakis2017training}
C.~Daskalakis, A.~Ilyas, V.~Syrgkanis, and H.~Zeng.
\newblock Training gans with optimism.
\newblock \emph{arXiv preprint arXiv:1711.00141}, 2017.

\bibitem[Diakonikolas(2020)]{diakonikolas2020halpern}
J.~Diakonikolas.
\newblock Halpern iteration for near-optimal and parameter-free monotone inclusion and strong solutions to variational inequalities.
\newblock In \emph{COLT}, 2020.

\bibitem[Diakonikolas et~al.(2021)Diakonikolas, Daskalakis, and Jordan]{diakonikolas2021efficient}
J.~Diakonikolas, C.~Daskalakis, and M.~I. Jordan.
\newblock Efficient methods for structured nonconvex-nonconcave min-max optimization.
\newblock In \emph{AISTATS}, 2021.

\bibitem[Fan et~al.(2023)Fan, Li, and Chen]{fan2023weaker}
Y.~Fan, Y.~Li, and B.~Chen.
\newblock Weaker mvi condition: Extragradient methods with multi-step exploration.
\newblock In \emph{ICLR}, 2023.

\bibitem[Fox et~al.(2025)Fox, Mishkin, Vaswani, and Schmidt]{fox2025glocal}
C.~Fox, A.~Mishkin, S.~Vaswani, and M.~Schmidt.
\newblock Glocal smoothness: Line search can really help!
\newblock \emph{arXiv preprint arXiv:2506.12648}, 2025.

\bibitem[Gidel et~al.(2019)Gidel, Berard, Vignoud, Vincent, and Lacoste-Julien]{gidel2018variational}
G.~Gidel, H.~Berard, G.~Vignoud, P.~Vincent, and S.~Lacoste-Julien.
\newblock A variational inequality perspective on generative adversarial networks.
\newblock \emph{ICLR}, 2019.

\bibitem[Goodfellow et~al.(2020)Goodfellow, Pouget-Abadie, Mirza, Xu, Warde-Farley, Ozair, Courville, and Bengio]{goodfellow2020generative}
I.~Goodfellow, J.~Pouget-Abadie, M.~Mirza, B.~Xu, D.~Warde-Farley, S.~Ozair, A.~Courville, and Y.~Bengio.
\newblock Generative adversarial networks.
\newblock \emph{Communications of the ACM}, 2020.

\bibitem[Gorbunov et~al.(2022{\natexlab{a}})Gorbunov, Berard, Gidel, and Loizou]{gorbunov2022stochastic}
E.~Gorbunov, H.~Berard, G.~Gidel, and N.~Loizou.
\newblock Stochastic extragradient: General analysis and improved rates.
\newblock In \emph{AISTATS}, 2022{\natexlab{a}}.

\bibitem[Gorbunov et~al.(2022{\natexlab{b}})Gorbunov, Loizou, and Gidel]{gorbunov2022extragradient}
E.~Gorbunov, N.~Loizou, and G.~Gidel.
\newblock Extragradient method: O (1/k) last-iterate convergence for monotone variational inequalities and connections with cocoercivity.
\newblock In \emph{AISTATS}, 2022{\natexlab{b}}.

\bibitem[Gorbunov et~al.(2024)Gorbunov, Danilova, Shibaev, Dvurechensky, and Gasnikov]{gorbunov2021high}
E.~Gorbunov, M.~Danilova, I.~Shibaev, P.~Dvurechensky, and A.~Gasnikov.
\newblock High probability complexity bounds for non-smooth stochastic optimization with heavy-tailed noise.
\newblock \emph{JOTA}, 2024.

\bibitem[Gorbunov et~al.(2025)Gorbunov, Tupitsa, Choudhury, Aliev, Richt{\'a}rik, Horv{\'a}th, and Tak{\'a}{\v{c}}]{gorbunov2024methods}
E.~Gorbunov, N.~Tupitsa, S.~Choudhury, A.~Aliev, P.~Richt{\'a}rik, S.~Horv{\'a}th, and M.~Tak{\'a}{\v{c}}.
\newblock Methods for convex $(l\_0, l\_1) $-smooth optimization: Clipping, acceleration, and adaptivity.
\newblock \emph{ICLR}, 2025.

\bibitem[Hochreiter and Schmidhuber(1997)]{hochreiter1997long}
S.~Hochreiter and J.~Schmidhuber.
\newblock Long short-term memory.
\newblock \emph{Neural computation}, 1997.

\bibitem[Hsieh et~al.(2019)Hsieh, Iutzeler, Malick, and Mertikopoulos]{hsieh2019convergence}
Y.-G. Hsieh, F.~Iutzeler, J.~Malick, and P.~Mertikopoulos.
\newblock On the convergence of single-call stochastic extra-gradient methods.
\newblock \emph{NeurIPS}, 2019.

\bibitem[Korpelevich(1977)]{korpelevich1977extragradient}
G.~Korpelevich.
\newblock Extragradient method for finding saddle points and other problems.
\newblock \emph{Matekon}, 1977.

\bibitem[Lee and Kim(2021)]{lee2021fast}
S.~Lee and D.~Kim.
\newblock Fast extra gradient methods for smooth structured nonconvex-nonconcave minimax problems.
\newblock \emph{NeurIPS}, 2021.

\bibitem[Li et~al.(2022)Li, Yu, Loizou, Gidel, Ma, Le~Roux, and Jordan]{li2022convergence}
C.~J. Li, Y.~Yu, N.~Loizou, G.~Gidel, Y.~Ma, N.~Le~Roux, and M.~Jordan.
\newblock On the convergence of stochastic extragradient for bilinear games using restarted iteration averaging.
\newblock In \emph{AISTATS}, 2022.

\bibitem[Loizou et~al.(2020)Loizou, Berard, Jolicoeur-Martineau, Vincent, Lacoste-Julien, and Mitliagkas]{loizou2020stochastic}
N.~Loizou, H.~Berard, A.~Jolicoeur-Martineau, P.~Vincent, S.~Lacoste-Julien, and I.~Mitliagkas.
\newblock Stochastic hamiltonian gradient methods for smooth games.
\newblock In \emph{ICML}, 2020.

\bibitem[Loizou et~al.(2021)Loizou, Berard, Gidel, Mitliagkas, and Lacoste-Julien]{loizou2021stochastic}
N.~Loizou, H.~Berard, G.~Gidel, I.~Mitliagkas, and S.~Lacoste-Julien.
\newblock Stochastic gradient descent-ascent and consensus optimization for smooth games: Convergence analysis under expected co-coercivity.
\newblock \emph{NeurIPS}, 2021.

\bibitem[Luo and O’Neill(2025)]{luo2025adaptive}
Y.~Luo and M.~J. O’Neill.
\newblock Adaptive extragradient methods for root-finding problems under relaxed assumptions.
\newblock In \emph{AISTATS}, 2025.

\bibitem[Luo and Tran-Dinh(2022)]{luo2022extragradient}
Y.~Luo and Q.~Tran-Dinh.
\newblock Extragradient-type methods for co-monotone root-finding problems.
\newblock \emph{UNC-STOR Technical Report}, 2022.

\bibitem[Mishchenko et~al.(2020)Mishchenko, Kovalev, Shulgin, Richt{\'a}rik, and Malitsky]{mishchenko2020revisiting}
K.~Mishchenko, D.~Kovalev, E.~Shulgin, P.~Richt{\'a}rik, and Y.~Malitsky.
\newblock Revisiting stochastic extragradient.
\newblock In \emph{AISTATS}, 2020.

\bibitem[Mokhtari et~al.(2020)Mokhtari, Ozdaglar, and Pattathil]{mokhtari2020unified}
A.~Mokhtari, A.~Ozdaglar, and S.~Pattathil.
\newblock A unified analysis of extra-gradient and optimistic gradient methods for saddle point problems: Proximal point approach.
\newblock In \emph{AISTATS}, 2020.

\bibitem[Namkoong and Duchi(2016)]{namkoong2016stochastic}
H.~Namkoong and J.~C. Duchi.
\newblock Stochastic gradient methods for distributionally robust optimization with f-divergences.
\newblock \emph{NeurIPS}, 2016.

\bibitem[Pethick et~al.(2022)Pethick, Latafat, Patrinos, Fercoq, and Cevher]{pethick2023escaping}
T.~Pethick, P.~Latafat, P.~Patrinos, O.~Fercoq, and V.~Cevher.
\newblock Escaping limit cycles: Global convergence for constrained nonconvex-nonconcave minimax problems.
\newblock \emph{ICLR}, 2022.

\bibitem[Popov(1980)]{popov1980modification}
L.~D. Popov.
\newblock A modification of the arrow-hurwicz method for search of saddle points.
\newblock \emph{Mathematical notes of the Academy of Sciences of the USSR}, 1980.

\bibitem[Riabinin et~al.(2025)Riabinin, Shulgin, Gruntkowska, and Richt{\'a}rik]{riabinin2025gluon}
A.~Riabinin, E.~Shulgin, K.~Gruntkowska, and P.~Richt{\'a}rik.
\newblock Gluon: Making muon \& scion great again!(bridging theory and practice of lmo-based optimizers for llms).
\newblock \emph{arXiv preprint arXiv:2505.13416}, 2025.

\bibitem[Ryu and Yin(2022)]{ryu2022large}
E.~K. Ryu and W.~Yin.
\newblock \emph{Large-scale convex optimization: algorithms \& analyses via monotone operators}.
\newblock Cambridge University Press, 2022.

\bibitem[Sokota et~al.(2023)Sokota, D'Orazio, Kolter, Loizou, Lanctot, Mitliagkas, Brown, and Kroer]{sokotaunified}
S.~Sokota, R.~D'Orazio, J.~Z. Kolter, N.~Loizou, M.~Lanctot, I.~Mitliagkas, N.~Brown, and C.~Kroer.
\newblock A unified approach to reinforcement learning, quantal response equilibria, and two-player zero-sum games.
\newblock In \emph{ICLR}, 2023.

\bibitem[Tran-Dinh(2024)]{tran2024variance}
Q.~Tran-Dinh.
\newblock Variance-reduced forward-reflected-backward splitting methods for nonmonotone generalized equations.
\newblock In \emph{ICML}, 2024.

\bibitem[Tseng(1995)]{tseng1995linear}
P.~Tseng.
\newblock On linear convergence of iterative methods for the variational inequality problem.
\newblock \emph{Journal of Computational and Applied Mathematics}, 1995.

\bibitem[Vankov et~al.(2024)Vankov, Nedich, and Sankar]{pmlr-v235-vankov24a}
D.~Vankov, A.~Nedich, and L.~Sankar.
\newblock Generalized smooth variational inequalities: Methods with adaptive stepsizes.
\newblock In \emph{ICML}, 2024.

\bibitem[Vankov et~al.(2025)Vankov, Rodomanov, Nedich, Sankar, and Stich]{vankov2024optimizing}
D.~Vankov, A.~Rodomanov, A.~Nedich, L.~Sankar, and S.~U. Stich.
\newblock Optimizing $(l\_0, l\_1) $-smooth functions by gradient methods.
\newblock \emph{ICLR}, 2025.

\bibitem[Vaswani et~al.(2017)Vaswani, Shazeer, Parmar, Uszkoreit, Jones, Gomez, Kaiser, and Polosukhin]{vaswani2017attention}
A.~Vaswani, N.~Shazeer, N.~Parmar, J.~Uszkoreit, L.~Jones, A.~N. Gomez, {\L}.~Kaiser, and I.~Polosukhin.
\newblock Attention is all you need.
\newblock \emph{NeurIPS}, 2017.

\bibitem[Xian et~al.(2024)Xian, Chen, and Huang]{xian2024delving}
W.~Xian, Z.~Chen, and H.~Huang.
\newblock Delving into the convergence of generalized smooth minimax optimization.
\newblock In \emph{ICML}, 2024.

\bibitem[Yoon and Ryu(2021)]{yoon2021accelerated}
T.~Yoon and E.~K. Ryu.
\newblock Accelerated algorithms for smooth convex-concave minimax problems with rate on squared gradient norm.
\newblock \emph{ICML}, 2021.

\bibitem[Yoon et~al.(2025)Yoon, Choudhury, and Loizou]{yoon2025multiplayer}
T.~Yoon, S.~Choudhury, and N.~Loizou.
\newblock Multiplayer federated learning: Reaching equilibrium with less communication.
\newblock \emph{NeurIPS}, 2025.

\bibitem[Zhang et~al.(2020{\natexlab{a}})Zhang, Jin, Fang, and Wang]{zhang2020improved}
B.~Zhang, J.~Jin, C.~Fang, and L.~Wang.
\newblock Improved analysis of clipping algorithms for non-convex optimization.
\newblock \emph{NeurIPS}, 2020{\natexlab{a}}.

\bibitem[Zhang et~al.(2020{\natexlab{b}})Zhang, He, Sra, and Jadbabaie]{zhang2019gradient}
J.~Zhang, T.~He, S.~Sra, and A.~Jadbabaie.
\newblock Why gradient clipping accelerates training: A theoretical justification for adaptivity.
\newblock \emph{ICLR}, 2020{\natexlab{b}}.

\bibitem[Zhang et~al.(2024)Zhang, Choudhury, Stich, and Loizou]{zhangcommunication}
S.~Zhang, S.~Choudhury, S.~U. Stich, and N.~Loizou.
\newblock Communication-efficient gradient descent-accent methods for distributed variational inequalities: Unified analysis and local updates.
\newblock In \emph{ICLR}, 2024.

\bibitem[Zheng et~al.(2024)Zheng, Loizou, You, and Mallada]{zheng2024dissipative}
T.~Zheng, N.~Loizou, P.~You, and E.~Mallada.
\newblock Dissipative gradient descent ascent method: A control theory inspired algorithm for min-max optimization.
\newblock \emph{IEEE Control Systems Letters}, 2024.

\end{thebibliography}

\newpage
\section*{NeurIPS Paper Checklist}

\begin{enumerate}

\item {\bf Claims}
    \item[] Question: Do the main claims made in the abstract and introduction accurately reflect the paper's contributions and scope?
    \item[] Answer: \answerYes{} % Replace by \answerYes{}, \answerNo{}, or \answerNA{}.
    \item[] Justification: see Section \ref{sec:convergence_analysis} and \ref{sec:experiments}.

\item {\bf Limitations}
    \item[] Question: Does the paper discuss the limitations of the work performed by the authors?
    \item[] Answer: \answerYes{} % Replace by \answerYes{}, \answerNo{}, or \answerNA{}.
    \item[] Justification: see Section \ref{sec:convergence_analysis}.

\item {\bf Theory assumptions and proofs}
    \item[] Question: For each theoretical result, does the paper provide the full set of assumptions and a complete (and correct) proof?
    \item[] Answer: \answerYes{} % Replace by \answerYes{}, \answerNo{}, or \answerNA{}.
    \item[] Justification: see Section \ref{sec:intro} and \ref{sec:convergence_analysis}.

    \item {\bf Experimental result reproducibility}
    \item[] Question: Does the paper fully disclose all the information needed to reproduce the main experimental results of the paper to the extent that it affects the main claims and/or conclusions of the paper (regardless of whether the code and data are provided or not)?
    \item[] Answer: \answerYes{} % Replace by \answerYes{}, \answerNo{}, or \answerNA{}.
    \item[] Justification: see Section \ref{sec:experiments}.

\item {\bf Open access to data and code}
    \item[] Question: Does the paper provide open access to the data and code, with sufficient instructions to faithfully reproduce the main experimental results, as described in supplemental material?
    \item[] Answer: \answerYes{} % Replace by \answerYes{}, \answerNo{}, or \answerNA{}.
    \item[] Justification: see Section \ref{sec:experiments}.

\item {\bf Experimental setting/details}
    \item[] Question: Does the paper specify all the training and test details (e.g., data splits, hyperparameters, how they were chosen, type of optimizer, etc.) necessary to understand the results?
    \item[] Answer: \answerYes{} % Replace by \answerYes{}, \answerNo{}, or \answerNA{}.
    \item[] Justification: see Section \ref{sec:experiments}.

\item {\bf Experiment statistical significance}
    \item[] Question: Does the paper report error bars suitably and correctly defined or other appropriate information about the statistical significance of the experiments?
    \item[] Answer: \answerNA{} % Replace by \answerYes{}, \answerNo{}, or \answerNA{}.
    \item[] Justification: all experiments are deterministic.

\item {\bf Experiments compute resources}
    \item[] Question: For each experiment, does the paper provide sufficient information on the computer resources (type of compute workers, memory, time of execution) needed to reproduce the experiments?
    \item[] Answer: \answerYes{} % Replace by \answerYes{}, \answerNo{}, or \answerNA{}.
    \item[] Justification: see Section \ref{sec:experiments}.
    
\item {\bf Code of ethics}
    \item[] Question: Does the research conducted in the paper conform, in every respect, with the NeurIPS Code of Ethics \url{https://neurips.cc/public/EthicsGuidelines}?
    \item[] Answer: \answerYes{} % Replace by \answerYes{}, \answerNo{}, or \answerNA{}.
    \item[] Justification: the paper follows NeurIPS Code of Ethics.

\item {\bf Broader impacts}
    \item[] Question: Does the paper discuss both potential positive societal impacts and negative societal impacts of the work performed?
    \item[] Answer: \answerNA{} % Replace by \answerYes{}, \answerNo{}, or \answerNA{}.
    \item[] Justification: the paper is mostly theoretical and does not have a direct societal impact.
    
\item {\bf Safeguards}
    \item[] Question: Does the paper describe safeguards that have been put in place for responsible release of data or models that have a high risk for misuse (e.g., pretrained language models, image generators, or scraped datasets)?
    \item[] Answer: \answerNA{} % Replace by \answerYes{}, \answerNo{}, or \answerNA{}.
    \item[] Justification: we do not release data or models.

\item {\bf Licenses for existing assets}
    \item[] Question: Are the creators or original owners of assets (e.g., code, data, models), used in the paper, properly credited and are the license and terms of use explicitly mentioned and properly respected?
    \item[] Answer: \answerYes{}{} % Replace by \answerYes{}, \answerNo{}, or \answerNA{}.
    \item[] Justification: see Section \ref{sec:experiments}.

\item {\bf New assets}
    \item[] Question: Are new assets introduced in the paper well documented and is the documentation provided alongside the assets?
    \item[] Answer: \answerNA{} % Replace by \answerYes{}, \answerNo{}, or \answerNA{}.
    \item[] Justification: the paper does not release new assets.

\item {\bf Crowdsourcing and research with human subjects}
    \item[] Question: For crowdsourcing experiments and research with human subjects, does the paper include the full text of instructions given to participants and screenshots, if applicable, as well as details about compensation (if any)? 
    \item[] Answer: \answerNA{} % Replace by \answerYes{}, \answerNo{}, or \answerNA{}.
    \item[] Justification: not applicable.

\item {\bf Institutional review board (IRB) approvals or equivalent for research with human subjects}
    \item[] Question: Does the paper describe potential risks incurred by study participants, whether such risks were disclosed to the subjects, and whether Institutional Review Board (IRB) approvals (or an equivalent approval/review based on the requirements of your country or institution) were obtained?
    \item[] Answer: \answerNA{} % Replace by \answerYes{}, \answerNo{}, or \answerNA{}.
    \item[] Justification: not applicable.

\item {\bf Declaration of LLM usage}
    \item[] Question: Does the paper describe the usage of LLMs if it is an important, original, or non-standard component of the core methods in this research? Note that if the LLM is used only for writing, editing, or formatting purposes and does not impact the core methodology, scientific rigorousness, or originality of the research, declaration is not required.
    %this research? 
    \item[] Answer: \answerNA{} % Replace by \answerYes{}, \answerNo{}, or \answerNA{}.
    \item[] Justification: LLMs were not used.

\end{enumerate}

%%%%%%%%%%%%%%%%%%%%%%%%%%%%%%%%%%%%%%%%%%%%%%%%%%%%%%%%%%%%

\newpage
\appendix

\part*{Supplementary Material}
The supplementary material is organized as follows. In Section \ref{sec:related_work}, we discuss papers that are closely related to our work. Section~\ref{appendix:tech_lemma} presents several technical lemmas used in our theoretical analysis. Section~\ref{appendix:examples} provides illustrative examples of $\alpha$-symmetric $(L_0, L_1)$-Lipschitz operators, while Section~\ref{appendix:convergence_analysis} contains the detailed proofs of the main convergence theorems presented in the paper. Next, in Section~\ref{appendix:equiv_formulation}, we discuss an equivalent formulation of the $\alpha$-symmetric $(L_0, L_1)$-Lipschitz condition. Finally, Section~\ref{appendix:num_exp} offers additional details on the experimental setup and results.

\tableofcontents

\newpage

% %%%%%%%%%%%%%%%%%%%%%%%%%%%%%%%%%%%%%%%%%%%%%%%%%%%%%%%%%%%%

\section{Further Related Work}\label{sec:related_work}

\textbf{Relaxing Lipschitzness.}
The $(L_0, L_1)$-smoothness assumption was first introduced by \citet{zhang2019gradient} and further developed in \citet{zhang2020improved, chen2023generalized}. More recently, improved convergence guarantees for optimization under this relaxed smoothness were obtained by \citet{gorbunov2024methods, vankov2024optimizing}.
This assumption has also been used in the analysis of modern large-scale optimization algorithms, for example, Gluon for LLM training~\citep{riabinin2025gluon}.

Beyond $(L_0, L_1)$-smoothness, several alternative notions or relaxations of smoothness have been proposed in the literature, including relative smoothness~\citep{bauschke2017descent},
glocal smoothness~\citep{fox2025glocal}, and H\"older smoothness~\citep{gorbunov2021high},
among others. The $(L_0, L_1)$-smoothness framework has recently been extended to the
variational inequality setting by \citet{pmlr-v235-vankov24a}, while \citet{xian2024delving}
analyzed min--max problems under a similarly general smoothness perspective. A recent line of work by \citet{loizou2021stochastic,gorbunov2022extragradient,beznosikov2023stochastic} studies star-cocoercive operators, which could hold even when the operator is not Lipschitz.

\textbf{On Methods for Solving VIs.}
A variety of methods have been developed to solve root-finding problems and, more generally, variational inequalities. The most basic approach is the gradient method and its stochastic counterpart~\citep{loizou2021stochastic,beznosikov2023stochastic}. More variants of the gradient method have been proposed in different settings, like the dissipative gradient method~\citep{zheng2024dissipative} and distributed methods for VIs~\citep{zhangcommunication}. However the gradient method fails to converge in simple monotone problems~\citep{gidel2018variational}.To obtain convergence guarantees in more relaxed settings (including monotone problems), \citet{korpelevich1977extragradient} introduced the \algname{EG} method, which has since been extended to adaptive~\citep{antonakopoulos2020adaptive}, stochastic~\citep{gorbunov2022stochastic, gidel2018variational, mishchenko2020revisiting, diakonikolas2021efficient, li2022convergence}, and decentralized variants~\citep{beznosikov2022decentralized}. A key drawback of \algname{EG} is that it requires two oracle calls per iteration. To address this, \citet{popov1980modification} proposed the Optimistic Gradient (\algname{OG}) method, which uses only a single oracle call. Similar to \algname{EG}, \algname{OG} has been analyzed in stochastic settings~\citep{hsieh2019convergence, gidel2018variational, daskalakis2017training, choudhury2024single, bohm2022solving}. Beyond these methods, accelerated variants have also been developed, achieving faster convergence rates~\citep{yoon2021accelerated, diakonikolas2020halpern, lee2021fast}.

\newpage
\section{Technical Lemmas}\label{appendix:tech_lemma}
In this section, we present some technical lemmas, which will be used to prove the main results of the work in subsequent sections.

\begin{lemma}
    For $a, b \in \R^d$, we have
    \begin{equation}\label{eq:inner_prod}
        2 \la a, b\ra = \|a\|^2 + \|b\|^2 - \|a - b\|^2.
    \end{equation}
\end{lemma}

\begin{lemma}
    For $a, b \in \R^d$, we have
    \begin{equation}\label{eq:inequality1}
        -\|a\|^2 \leq - \frac{1}{2} \|a + b\|^2 + \|b\|^2.  
    \end{equation}
\end{lemma}

\begin{lemma}
    For positive numbers $a, b > 0$ we have
    \begin{equation}\label{eq:AMGM_corollary}
        a + b \leq \sqrt{2} \sqrt{a^2 + b^2}
    \end{equation}
\end{lemma}

\begin{proof}
    From AM-GM inequality on $a, b > 0$ we have
    \begin{equation*}
        2ab \leq a^2 + b^2 .
    \end{equation*}
    Adding $a^2 + b^2$ on both sides we have 
    \begin{equation*}
        (a + b)^2 \leq 2(a^2 + b^2).
    \end{equation*}
    Then the result follows by taking square root on both sides.
\end{proof}

\begin{lemma}[Cauchy-Schwarz Inequality]
    For vectors $a, b \in \R^d$, we have
    \begin{equation}\label{eq:cauchy}
        \la a, b\ra \leq \|a\| \|b\|.
    \end{equation}
\end{lemma}

\begin{lemma}\label{lemma:reform_integration}\cite{chen2023generalized}
    Operator $F$ is $\alpha$-symmetric $(L_0, L_1)$-Lipschitz if and only if
    \begin{equation}\label{eq:reform_integration}
        \|F(x) - F(y)\| \leq \left( L_0 + L_1 \int_{0}^1 \left\| F(\theta x + (1 - \theta)y)\right\|^{\alpha} d\theta \right) \|x - y\| \qquad \forall x, y \in \R^d.
    \end{equation}
\end{lemma}

\begin{lemma}
    For a $2 \times 2$ symmetric matrix, the maximum eigenvalue is given by
    \begin{equation}\label{eq:eigen_2dmatrix}
        \lambda_{\max} \left( 
        \begin{bmatrix}
            a & b \\
            b & d
        \end{bmatrix}
        \right) = \frac{(a + d) + \sqrt{(a - d)^2 + 4b^2}}{2}
    \end{equation}
    where $a, b, d \in \R$.
\end{lemma}

\begin{proof}
        Let $A$ be a symmetric $2 \times 2$ matrix given by
        \[
        A = \begin{bmatrix}
        a & b \\
        b & d
        \end{bmatrix}
        \]
        where $a, b, d \in \R$. Since $A$ is symmetric, it has real eigenvalues. The eigenvalues of $A$ are the roots of its characteristic polynomial:
        \[
        \det(A - \lambda I) = 
        \det\left(
        \begin{bmatrix}
        a - \lambda & b \\
        b & d - \lambda
        \end{bmatrix}
        \right)
        = (a - \lambda)(d - \lambda) - b^2.
        \]
        
        Expanding the determinant, we obtain the characteristic equation:
        \[
        \lambda^2 - (a + d)\lambda + (ad - b^2) = 0.
        \]
        
        This is a quadratic equation in $\lambda$, and its solutions are:
        \[
        \lambda = \frac{(a + d) \pm \sqrt{(a - d)^2 + 4b^2}}{2}.
        \]
        
        Thus, the maximum eigenvalue is the larger of the two roots:
        \[
        \lambda_{\max}(A) = \frac{(a + d) + \sqrt{(a - d)^2 + 4b^2}}{2}.
        \]
        
        This completes the proof.
\end{proof}

\begin{lemma}\label{lemma:jacobian_norm_quad}
    For the quadratic problem $\min_{w_1} \max_{w_2} \mathcal{L}(w_1, w_2) = \frac{1}{2} w_1^2 + w_1 w_2 - \frac{1}{2}w_2^2$, the Jacobian $\mathbf{J}(x)$ is given by
    \begin{eqnarray*}
        \mathbf{J}(x) = \begin{bmatrix}
            1 & 1 \\
            -1 & 1
        \end{bmatrix}.
    \end{eqnarray*}
    In this case we get $\| \mathbf{J}(x)\| = \sigma_{\max} \left(\mathbf{J}(x) \right) = \sqrt{2}$.
\end{lemma}

\begin{proof}
    Note that 
    \begin{eqnarray*}
        \mathbf{J}(x)^{\top} \mathbf{J}(x) = \begin{bmatrix}
            2 & 0 \\
            0 & 2
        \end{bmatrix} = 2 \mathbf{I}.
    \end{eqnarray*}
    which has maximum eigenvalue $\sqrt{2}$. Hence, $\|\mathbf{J}(x)\| = \sigma_{\max}(\mathbf{J}(x)) = \sqrt{\lambda_{\max}(\mathbf{J}(x)^{\top} \mathbf{J}(x))}= \sqrt{2}.$
\end{proof}

\begin{lemma}[\cite{vankov2024optimizing} Example 2.2]\label{lemma:L0L1_power}
    The function $f(x) = \frac{1}{p+1} \|x\|^{p+1}$, where $p > 1$ is $(\tau_0, \tau_1)$-smooth with arbitary $\tau_1 > 0$ and $\tau_0 = \left( \frac{p-1}{\tau_1}\right)^{p-1}$ i.e.
    \begin{equation}\label{eq:L0L1_power}
        \|\nabla^2 f(x)\| \leq \left( \frac{p-1}{\tau_1}\right)^{p-1} + \tau_1 \| \nabla f(x)\|
    \end{equation}
    for any $\tau_1 > 0$.
\end{lemma}

\begin{lemma}[Equivalence of $L_q$ Norm]
    For any vector $x \in \R^d$, and $0 < r < q$, we have
    \begin{equation}\label{eq:holder_1}
        \|x\|_q \leq \|x\|_r \leq d^{\frac{1}{r} - \frac{1}{q}}\|x\|_q.
    \end{equation}
    In particular, for any $a, b \in \R$ and $0 < r < q$ we have
    \begin{equation}\label{eq:holder_2}
        \left( a^q + b^q \right)^{\frac{1}{q}} \leq \left( a^r + b^r \right)^{\frac{1}{r}}
    \end{equation}
    and 
    \begin{equation}\label{eq:holder_3}
        \left( a^r + b^r \right)^{\frac{1}{r}} \leq 2^{\frac{1}{r} - \frac{1}{q}}\left( a^q + b^q \right)^{\frac{1}{q}}
    \end{equation}
\end{lemma}

\begin{proof}
    The proof of \eqref{eq:holder_1} follows from Holder's Inequality. \eqref{eq:holder_2} and \eqref{eq:holder_3} follows from \eqref{eq:holder_1} with $d = 2$ and $x = (a, b)$.
\end{proof}

\newpage
\section{Further Examples of $\alpha$-Symmetric $(L_0, L_1)$-Lipschitz Operators}\label{appendix:examples}
In this section, we provide further examples of operators that satisfy the $\alpha$-symmetric $(L_0, L_1)$-Lipschitz assumption. We first start with the details of Examples 2 and 3 from Section \ref{sec:examples}, then we provide three more examples for the min-max optimization problem and $N$-player game that satisfy the assumption.\\
\newline
\textbf{Example 2:}\label{subsec:example2} Here we consider the operator $F(x) = (u_1^2, u_2^2)$ for $x = (u_1, u_2)$. Then for $y = (v_1, v_2)$, we have
\begin{eqnarray*}
    \|F(x) - F(y)\| & = & \left\|\left(u_1^2 - v_1^2, u_2^2 - v_2^2 \right) \right\| \\
    & = & \left( \left( u_1^2 - v_1^2\right)^2 + \left(u_2^2 - v_2^2 \right)^2 \right)^{\nicefrac{1}{2}} \\
    & = & \left( \left( u_1 - v_1\right)^2 \left( u_1 + v_1\right)^2 + \left(u_2 - v_2 \right)^2 \left(u_2 + v_2 \right)^2 \right)^{\nicefrac{1}{2}} \\
    & \leq & \left( (u_1 - v_1)^4 + (u_2 - v_2)^4 \right)^{\nicefrac{1}{4}} \left( (u_1 + v_1)^4 + (u_2 + v_2)^4 \right)^{\nicefrac{1}{4}} \\
    & \leq & \left( (u_1 - v_1)^2 + (u_2 - v_2)^2 \right)^{\nicefrac{1}{2}} \left( (u_1 + v_1)^4 + (u_2 + v_2)^4 \right)^{\nicefrac{1}{4}} \\
    & = &  \left( (u_1 + v_1)^4 + (u_2 + v_2)^4 \right)^{\nicefrac{1}{4}} \|x - y\| \\
    & = & 2 \left(\left(\frac{u_1 + v_1}{2} \right)^4 + \left(\frac{u_2 + v_2}{2} \right)^4 \right)^{\nicefrac{1}{4}} \|x - y\| \\
    & = & 2 \left\| \left( \frac{u_1 + v_1}{2}\right)^2, \left( \frac{u_2 + v_2}{2}\right)^2 \right\|^{\nicefrac{1}{2}} \|x - y\| \\
    & = & 2 \left\| F \left(\frac{x+y}{2} \right)\right\|^{\nicefrac{1}{2}} \|x - y\| \\
    & \leq & 2 \max_{\theta \in [0, 1]} \left\|  F \left( \theta x + (1 - \theta) y \right)\right\|^{\nicefrac{1}{2}} \|x - y\|.
\end{eqnarray*}
Here, the first inequality follows from the Cauchy-Schwarz inequality. This completes the proof of $\frac{1}{2}$-symmetric-$(0, 2)$ Lipschitz property of $F$. Now, we consider the vectors $x = \alpha \mathbf{1}_2$ and $y = \mathbf{1}_2$ where $\mathbf{1}_2 = (1, 1)$. Then we have 
\begin{eqnarray*}
    \frac{\| F(x) - F(y)\|}{\| x - y\|} & = & \frac{\sqrt{(\alpha^2 - 1)^2 + (\alpha^2 - 1)^2}}{\sqrt{(\alpha - 1)^2 + (\alpha - 1)^2}} \\
    & = & \frac{\sqrt{2(\alpha^2 - 1)^2}}{\sqrt{2(\alpha - 1)^2}} \\
    & = & \sqrt{\frac{(\alpha - 1)^2 (\alpha + 1)^2}{(\alpha - 1)^2}} \\ 
    & = & \lvert \alpha + 1 \rvert.
\end{eqnarray*}
Therefore, 
\begin{eqnarray*}
    \lim_{\alpha \to \infty} \frac{\| F(x) - F(y)\|}{\| x - y\|} & = & \lim_{\alpha \to \infty} \lvert \alpha + 1 \rvert \\
    & = & \infty
\end{eqnarray*}
This proves that $F$ is not Lipschitz.

\newpage
\textbf{Example 3.} Consider the min-max optimization problem 
\begin{equation}
    \min_{w_1} \max_{w_2} \mathcal{L}(w_1, w_2) = \frac{1}{p+1} \|w_1\|^{p+1} + w_1^\top \B w_2 - \frac{1}{p+1} \|w_2\|^{p+1}.
    \label{bjaskxa}
\end{equation}
    Then the corresponding operator $F$ defined in \eqref{eq:minmax_operator} is $1$-symmetric $\left(2\tau_0 + \|\M\|, 2^{\frac{2p^2 - 1}{p^2}}\tau_1 \right)$-Lipschitz with any arbitary $\tau_1 > 0$ and $\tau_0 = \left( \frac{p-1}{\tau_1}\right)^{p-1}$
    where $\M = \begin{bmatrix}
            0 & \B \\
            -\B^\top & 0
        \end{bmatrix}$. Moreover, $F$ is not Lipschitz for any finite $L$.

\begin{proof}
    For the simplicity of the proof, we define the function $f(w) \eqdef \frac{1}{p+1} \|w\|^{p+1}$. The following properties of this function will be used to prove our result.
    
    \textbf{1. Gradient of $f(w)$:} The gradient of $f(w) = \frac{1}{p+1}\|w\|^{p+1} = \frac{1}{p+1} \left(\|w\|^2 \right)^{\frac{p+1}{2}}$ is given by
    \begin{equation}\label{eq:minmax_example6_1by16}
        \nabla f(w) = \frac{1}{p+1} \cdot \frac{p+1}{2} \left(\|w\|^2 \right)^{\frac{p+1}{2} - 1} \cdot \nabla \left( \|w\|^2 \right) = \frac{1}{2} \left( \|w\|^2 \right)^{\frac{p-1}{2}} \cdot 2w = \|w\|^{p-1} w.
    \end{equation}
    Here, the first equality follows from chain rule.

    \textbf{2. $f$ is $(\tau_0, \tau_1)$-smooth:} From Lemma \ref{lemma:L0L1_power}, we know $f$ is $(\tau_0, \tau_1)$-smooth. Choose arbitary $\tau_1 > 0$, then from Lemma \ref{lemma:L0L1_power}, $f$ satisfies 
    \begin{equation}\label{eq:minmax_example6_1by8}
        \|\nabla^2 f(w)\| \overset{\eqref{eq:L0L1_power}}{\leq} \tau_0 + \tau_1 \| \nabla f(w)\|
    \end{equation}
    where $\tau_0 = \left( \frac{p-1}{\tau_1}\right)^{p-1}$. 

    Now, we will get back to the original min-max problem. Note that, using the definition of $f(w)$, we can rewrite the objective function in \eqref{bjaskxa} as 
    \begin{equation}\label{eq:minmax_example6_1by4}
        \mathcal{L}(w_1, w_2) = f(w_1) + w_1^\top \B w_2 - f(w_2).    
    \end{equation}
    Then the operator $F$ corresponding to the min-max problem \eqref{bjaskxa} is given by
    \begin{equation}\label{eq:minmax_example6_1by2}
        F(x) = \begin{bmatrix}
            \nabla_{w_1} \mathcal{L}(w_1, w_2)\\
            -\nabla_{w_2} \mathcal{L}(w_1, w_2)
        \end{bmatrix} \overset{\eqref{eq:minmax_example6_1by4}}{=} \begin{bmatrix}
            \nabla f(w_1) + \B w_2 \\
            \nabla f(w_2) - \B^\top w_1
        \end{bmatrix} \overset{\eqref{eq:minmax_example6_1by16}}{=} \begin{bmatrix}
            \|w_1\|^{p-1}w_1 + \B w_1\\
            \|w_2\|^{p-1}w_2 - \B^\top w_2
        \end{bmatrix}.
    \end{equation}
    Moreover, the Jacobian matrix corresponding to this min-max problem~\eqref{bjaskxa} is 
    \begin{equation}\label{eq:minmax_example6_0}
        \mathbf{J}(x) = \begin{bmatrix}
            \nabla^2_{w_1 w_1} \mathcal{L}(w_1, w_2) & \nabla^2_{w_2 w_1} \mathcal{L}(w_1, w_2) \\
            -\nabla^2_{w_1 w_2} \mathcal{L}(w_1, w_2) & -\nabla^2_{w_2 w_2} \mathcal{L}(w_1, w_2)
            \end{bmatrix}
            \overset{\eqref{eq:minmax_example6_1by4}}{=} \begin{bmatrix}
                \nabla^2 f(w_1) & \B \\
                -\B^\top & \nabla^2 f(w_2)
            \end{bmatrix}.
    \end{equation}
    Furthermore, note that, we can break the operator $F$ from \eqref{eq:minmax_example6_1by2} as $F(x) = H(x) + \M x$ with 
    \begin{eqnarray}\label{eq:minmax_example6_1}
        H(x) = \begin{bmatrix}
            \nabla f(w_1)\\
            \nabla f(w_2)
        \end{bmatrix} = \begin{bmatrix}
            \|w_1\|^{p-1}w_1 \\
            \|w_2\|^{p-1}w_2
        \end{bmatrix}
    \end{eqnarray}
    and $\M = \begin{bmatrix}
            0 & \B \\
            -\B^\top & 0
        \end{bmatrix}.$
    Then the Jacobian $\mathbf{J}(x)$ of this min-max problem satisfies
    \begin{eqnarray}\label{eq:minmax_example6_2}
        \| \mathbf{J}(x) \| & \overset{\eqref{eq:minmax_example6_0}}{=} & \left\| \begin{bmatrix}
            \nabla^2 f(w_1) & \B \\
            -\B^\top & \nabla^2 f(w_2)
        \end{bmatrix}\right\| \notag\\
        & = & \left\| \begin{bmatrix}
            \nabla^2 f(w_1) & 0 \\
            0 & \nabla^2 f(w_2)
        \end{bmatrix} + \M \right\| \notag\\ 
        & \leq & \left\| \begin{bmatrix} 
            \nabla^2 f(w_1) & 0 \\
            0 & \nabla^2 f(w_2)
        \end{bmatrix}\right\| + \left\| \M \right\| \notag \\
        &\leq& \max \left\{ \|\nabla^2 f(w_1)\|, \|\nabla^2 f(w_2)\| \right\} + \left\| \M \right\| \notag\\
        & \leq & \|\nabla^2 f(w_1)\| + \|\nabla^2 f(w_2)\| + \left\| \M \right\|  \notag\\
        & \overset{\eqref{eq:minmax_example6_1by8}}{\leq} & \tau_0 + \tau_1 \|\nabla f(w_1)\|  +  \tau_0 + \tau_1 \| \nabla f(w_2)\| + \left\| \M \right\|  \notag\\
        & = & 2 \tau_0 + \left\| \M \right\| + \tau_1 \left( \|\nabla f(w_1)\| + \|\nabla f(w_2)\|\right)  \notag\\
        & \overset{\eqref{eq:AMGM_corollary}}{\leq} & (2 \tau_0 + \left\| \M \right\|) + \sqrt{2} \tau_1 \sqrt{\|\nabla f(w_1)\|^2 + \|\nabla f(w_2)\|^2} \notag\\
        & \overset{\eqref{eq:minmax_example6_1}}{=} & (2\tau_0 + \|\M \|) + \sqrt{2} \tau_1 \|H(x)\|.
    \end{eqnarray}
    where $\tau_1 > 0$ is arbitary and $\tau_0 = \left( \frac{p-1}{\tau_1}\right)^{p-1}$. Above, the second inequality follows from the fact that the maximum eigenvalue of a block diagonal matrix is less than or equal to the maximum eigenvalues of it's diagonal matrices. Now, we want to find a bound on $\|H(x)\|$ in terms of $\|F(x)\|$. Towards that end, we will first find the relation between $\|H(x)\|$ and $\la F(x), x\ra$. Note that 
    \begin{eqnarray}\label{eq:minmax_example6_3}
        \la F(x), x\ra & = & \la H(x), x\ra + \la \M x, x\ra \notag\\ 
        & = & \la H(x), x\ra + \la \begin{bmatrix}
            \B w_2\\
            -\B^\top w_1
        \end{bmatrix}, \begin{bmatrix}
            w_1\\
            w_2
        \end{bmatrix}\ra \notag\\
        & = & \la H(x), x\ra \notag\\
        & \overset{\eqref{eq:minmax_example6_1}}{=} & \la \begin{bmatrix}
            \|w_1\|^{p-1}w_1 \\
            \|w_2\|^{p-1}w_2
        \end{bmatrix}, \begin{bmatrix}
            w_1 \\
            w_2
        \end{bmatrix}\ra \notag\\
        & = & \|w_1\|^{p+1} + \|w_2\|^{p+1}
    \end{eqnarray}
    and 
    \begin{eqnarray}\label{eq:minmax_example6_4}
        \| H(x)\| & \overset{\eqref{eq:minmax_example6_1}}{=} & \left\| \begin{bmatrix}
            \|w_1\|^{p-1}w_1 \\
            \|w_2\|^{p-1}w_2
        \end{bmatrix} \right\| \notag\\
        & = & \left( \left(\|w_1\|^{p-1} \right)^2 \|w_1\|^2 + \left(\|w_2\|^{p-1} \right)^2 \|w_2\|^2 \right)^{\frac{1}{2}} \notag \\
        & = & \left(\|w_1\|^{2p} + \|w_2\|^{2p} \right)^{\frac{1}{2}}
    \end{eqnarray}
    Then note that
    \begin{eqnarray*}
        \| H(x)\| 
        & \overset{\eqref{eq:minmax_example6_4}}{=} &\left(\|w_1\|^{2p} + \|w_2\|^{2p} \right)^{\frac{1}{2}} \\
        & = & \left( \left(\|w_1\|^{2p} + \|w_2\|^{2p} \right)^{\frac{1}{2p}} \right)^{p} \\
        & \overset{\eqref{eq:holder_2}}{\leq} & \left( \left(\|w_1\|^{p+1} + \|w_2\|^{p+1} \right)^{\frac{1}{p+1}} \right)^{p} \\
        & = & \left(\|w_1\|^{p+1} + \|w_2\|^{p+1} \right)^{\frac{p}{p+1}} \\
        & \overset{\eqref{eq:minmax_example6_3}}{=} & \la F(x), x\ra^{\frac{p}{p+1}} \\
        & \overset{\eqref{eq:cauchy}}{\leq} & \|F(x)\|^{\frac{p}{p+1}} \|x\|^{\frac{p}{p+1}} \\
        & = & \|F(x)\|^{\frac{p}{p+1}} \left( \left( \|w_1\|^2 + \|w_2\|^2 \right)^{\frac{1}{2}} \right)^{\frac{p}{p+1}} \\
        & \overset{\eqref{eq:holder_3}}{\leq} & \|F(x)\|^{\frac{p}{p+1}} \left( 2^{\frac{1}{2} - \frac{1}{2p}} \left( \|w_1\|^{2p} + \|w_2\|^{2p} \right)^{\frac{1}{2p}} \right)^{\frac{p}{p+1}} \\
        & = & 2^{\frac{p-1}{2p}}\|F(x)\|^{\frac{p}{p+1}} \left(  \left( \|w_1\|^{2p} + \|w_2\|^{2p} \right)^{\frac{1}{2p}} \right)^{\frac{p}{p+1}} \\
        & = & 2^{\frac{p-1}{2p}}\|F(x)\|^{\frac{p}{p+1}} \left(  \left( \|w_1\|^{2p} + \|w_2\|^{2p} \right)^{\frac{1}{2}} \right)^{\frac{1}{p+1}} \\
        & \overset{\eqref{eq:minmax_example6_4}}{=} & 2^{\frac{p-1}{2p}}\|F(x)\|^{\frac{p}{p+1}} \|H(x)\|^{\frac{1}{p+1}}
    \end{eqnarray*}
    Therefore, dividing both sides by $\|H(x)\|^{\frac{1}{p+1}}$ we get
    \begin{equation*}
        \|H(x)\|^{\frac{p}{p+1}} \leq  2^{\frac{p-1}{2p}}\|F(x)\|^{\frac{p}{p+1}}.
    \end{equation*}
    Thus raising both sides to the power $\frac{p+1}{p}$ we get
    \begin{equation}\label{eq:minmax_example6_5}
        \|H(x)\| \leq 2^{\frac{p^2-1}{2p^2}} \|F(x)\|.
    \end{equation}
    Combining this bound on $\|H(x)\|$ with \eqref{eq:minmax_example6_2} we get
    \begin{eqnarray*}
        \|\mathbf{J}(x)\| & \overset{\eqref{eq:minmax_example6_2}}{\leq} & (2\tau_0 + \|\M \|) + 2^{\frac{1}{2}} \tau_1 \|H(x)\| \\ 
        & \overset{\eqref{eq:minmax_example6_5}}{\leq} & (2\tau_0 + \|\M\|) + 2^{\frac{1}{2} + \frac{p^2-1}{2p^2}} \tau_1 \| F(x)\| \\
        & = & (2\tau_0 + \|\M\|) + 2^{\frac{2p^2-1}{2p^2}} \tau_1 \| F(x)\|.
    \end{eqnarray*}
    Thus, the result follows using Theorem \ref{thm:equiv_formulation}.  
    
    Next we show $F$ is not Lipschitz. Consider, $x = (tw_1, 0)$ and $y = (0, 0)$ with $\|w_1\| = 1$. Then we get
    \begin{eqnarray*}
        \frac{\|F(x) - F(y)\|}{\|x - y\|} = \frac{1}{p+1} \frac{t^{p+1}}{t} = \frac{t^{p}}{p+1}.
    \end{eqnarray*}
    As we take $t \to \infty$, we get $\frac{\|F(x) - F(y)\|}{\|x - y\|} \to \infty$. Hence, there doesn't exist any finite bound on $\frac{\|F(x) - F(y)\|}{\|x - y\|} $ or $F$ is not $L$-Lipschitz.
\end{proof}

\textbf{Example 4:} Consider the min-max problem 
$$\min_{w_1} \max_{w_2} \mathcal{L}(w_1, w_2) = \frac{1}{p+1}\|w_1\|^{p+1} - \frac{1}{p+1}\|w_2\|^{p+1}.$$ Then the corresponding operator $F$ defined in \eqref{eq:minmax_operator} is $1$-symmetric $(L_0, L_1)$-Lipschitz where $L_1 = \sqrt{2}\tau_1$ and $L_0 = 2 \left( \frac{p-1}{\tau_1}\right)^{p-1}$ for any $\tau_1 > 0$. Moreover, $F$ is not $L$-Lipschitz for any finite $L$.

\begin{proof}
    Define $f(w) = \frac{1}{p+1}\|w\|^{p+1}$. From Lemma \ref{lemma:L0L1_power}, we know $f$ is $(\tau_0, \tau_1)$-smooth. Choose arbitary $\tau_1 > 0$, then from Lemma \ref{lemma:L0L1_power}, $f$ satisfies 
    \begin{equation}\label{eq:minmax_example4_123}
        \|\nabla^2 f(w)\| \overset{\eqref{eq:L0L1_power}}{\leq} \tau_0 + \tau_1 \| \nabla f(w)\|
    \end{equation}
    where $\tau_0 = \left( \frac{p-1}{\tau_1}\right)^{p-1}$. Then we have $\mathcal{L}(w_1, w_2) = f(w_1) - f(w_2)$ and the corresponding operator and Jacobian are given by
    \begin{eqnarray*}
    F(x) = \begin{bmatrix}
        \nabla f(w_1) \\
        \nabla f(w_2)
    \end{bmatrix}
    \quad \text{and} \quad 
    \mathbf{J}(x) = \begin{bmatrix}
        \nabla^2 f(w_1) & 0 \\
        0 & \nabla^2 f(w_2)
    \end{bmatrix},
    \end{eqnarray*}
    respectively. Note that the Jacobian matrix $\mathbf{J}(x)$ is block diagonal. Hence, the maximum eigenvalue of $\mathbf{J}(x)$ is less than the maxiumum eigenvalue of both $\nabla^2 f(w_1)$ and $\nabla^2 f(w_2)$. Therefore,we get 
    \begin{eqnarray*}
        \left\| \mathbf{J}(x) \right\| & \leq & \max \left\{ \|\nabla^2 f(w_1)\|, \|\nabla^2 f(w_2)\| \right\} \\
        & \leq & \|\nabla^2 f(w_1)\| + \|\nabla^2 f(w_2)\| \\
        & \overset{\eqref{eq:minmax_example4_123}}{\leq} & \tau_0 + \tau_1 \|\nabla f(w_1)\|  +  \tau_0 + \tau_1 \| \nabla f(w_2)\| \\
        & = & 2 \tau_0 + \tau_1 \left( \|\nabla f(w_1)\| + \|\nabla f(w_2)\|\right) \\
        & \overset{\eqref{eq:AMGM_corollary}}{\leq} & 2 \tau_0 + \sqrt{2} \tau_1 \sqrt{\|\nabla f(w_1)\|^2 + \|\nabla f(w_2)\|^2} \\
        & = & 2\tau_0 + \sqrt{2} \tau_1 \| F(x)\|.
    \end{eqnarray*}
    Then $F$ is $1$-symmetric $(L_0, L_1)$-Lipschitz from Theorem \ref{thm:equiv_formulation} with $L_1 = \sqrt{2}\tau_1$ and $L_0 = 2 \tau_0 = 2 \left( \frac{p-1}{\tau_1}\right)^{p-1}$ for any $\tau_1 > 0$. 
    
    Next we show $F$ is not Lipschitz. Consider, $x = (tw_1, 0)$ and $y = (0, 0)$ with $\|w_1\| = 1$. Then we get
    \begin{eqnarray*}
        \frac{\|F(x) - F(y)\|}{\|x - y\|} = \frac{1}{p+1} \frac{t^{p+1}}{t} = \frac{t^{p}}{p+1}.
    \end{eqnarray*}
    As we take $t \to \infty$, we get $\frac{\|F(x) - F(y)\|}{\|x - y\|} \to \infty$. Hence, there doesn't exist any finite bound on $\frac{\|F(x) - F(y)\|}{\|x - y\|} $. Therefore, $F$ is not $L$-Lipschitz. 

\end{proof}

\textbf{Example 5:}
Min-max optimization problems can be studied using operators. Here we consider one such example, the bilinearly coupled min-max optimization~\cite{chambolle2016ergodic} problem
\begin{equation*}
    \min_{\|w_1 \| \leq R} \max_{\|w_2 \| \leq R} \mathcal{L}(w_1, w_2) \eqdef f(w_1) + w_1^{\top} \B w_2 - g(w_2)
\end{equation*}
for matrix $\B \in \R^{d \times d}$ and functions $f, g: \R^d \to \R$. The associated operator for this problem is given by $F(x) = H(x) + \M x$ where
\begin{eqnarray*}
\textstyle
    H(x) = \begin{bmatrix}
        \nabla f(w_1) \\
        \nabla g(w_2)
    \end{bmatrix} \quad \text{ and } \quad \M = \begin{bmatrix}
        0 & \B \\
        - \B^{\top} & 0
    \end{bmatrix}.
\end{eqnarray*}
If $f, g$ are individually $(L_0, L_1)$-smooth~\cite{zhang2019gradient}, we show that
\begin{eqnarray*}
\textstyle
    \|F(x) - F(y)\| & \leq & \left( 2 L_0 + (1 + 2 L_1 R) \| \M \|+ \sqrt{2} L_1 \|F(x)\| \right) \|x - y\|.
\end{eqnarray*}
Thus, $F$ is $1$-symmetric $ \left(2 L_0 + (1 + 2 L_1 R) \| \M \|, \sqrt{2} L_1 \right)$-Lipschitz. Consider the min-max problem
\begin{equation*}
    \min_{\|w_1\| \leq R} \max_{\|w_2\| \leq R} \mathcal{L}(w_1, w_2) \eqdef f(w_1) + w_1^{\top} \B w_2 - g(w_2).
\end{equation*}
where $f, g$ are $(L_0, L_1)$-smooth. Then for $x = (w_1, w_2)$ we have
\begin{eqnarray*}
    F(x) = \begin{pmatrix}
        \nabla f(w_1) \\
        \nabla g(w_2)
    \end{pmatrix} + \begin{pmatrix}
        0 & \B \\
        - \B^{\top} & 0
    \end{pmatrix} \begin{pmatrix}
        w_1 \\
        w_2
    \end{pmatrix} = H(x) + \M x
\end{eqnarray*}
where $\M$ is a matrix and $H$ is an operator. Now note that for $x = (w_1, w_2)$ and $y = (v_1, v_2)$ we get
\begin{eqnarray*}
    \|H(x) - H(y)\| & = &\left\| \begin{bmatrix}
        \nabla f(w_1) - \nabla f(v_1) \\
        \nabla g(w_2) - \nabla g(v_2)
    \end{bmatrix}\right\| \\
    & = & \left( \| \nabla f(w_1) - \nabla f(v_1)\|^2 + \| g(w_2) - g(v_2)\|^2 \right)^{\nicefrac{1}{2}} \\
    & \leq & \left( \left(L_0 + L_1 \| \nabla f(w_1)\| \right)^2 \| w_1 - v_1\|^2 + \left(L_0 + L_1 \| \nabla g(w_2)\| \right)^2\| w_2 - v_2 \|^2 \right)^{\nicefrac{1}{2}} \\
    & \leq & \left( \left(L_0 + L_1 \| \nabla f(w_1)\| \right)^4 + \left(L_0 + L_1 \| \nabla g(w_2)\| \right)^4 \right)^{\nicefrac{1}{4}} \left( \|w_1 - v_1\|^4 + \|w_2 - v_2\|^4 \right)^{\nicefrac{1}{4}} \\
    & \leq & \left( \left(L_0 + L_1 \| \nabla f(w_1)\| \right)^2 + \left(L_0 + L_1 \| \nabla g(w_2)\| \right)^2 \right)^{\nicefrac{1}{2}} \left( \|w_1 - v_1\|^2 + \|w_2 - v_2\|^2 \right)^{\nicefrac{1}{2}} \\
    & \leq & \left( 4 L_0^2 + 2L_1^2 \left(\| \nabla f(w_1) \|^2 + \| \nabla g(w_2)\|^2 \right) \right)^{\nicefrac{1}{2}} \left( \|w_1 - v_1\|^2 + \|w_2 - v_2\|^2 \right)^{\nicefrac{1}{2}} \\
    & = & \left( 4 L_0^2 + 2L_1^2 \|H(x)\|^2 \right)^{\nicefrac{1}{2}} \|x - y\| \\
    & \leq & \left( 2 L_0 + \sqrt{2} L_1 \|H(x)\| \right) \|x - y\|
\end{eqnarray*}
Therefore, using the above inequality, we derive
\begin{eqnarray*}
    \|F(x) - F(y)\| & \leq & \|H(x) - H(y)\| + \|\M x - \M y\| \\
    & \leq & \left( 2 L_0 + \sqrt{2} L_1 \|H(x)\| \right) \|x - y\|+ \|\M\| \|x - y\| \\
    & \leq & \left( 2 L_0 + \|\M\| + \sqrt{2} L_1 \|H(x)\| \right) \|x - y\| 
\end{eqnarray*}
Now using $\|H(x)\| = \|H(x) + \M x - \M x\| \leq \|F(x)\| + \| \M \| \|x\| \leq \|F(x)\| + \sqrt{2}R \|\M \|$ we get
\begin{eqnarray*}
    \|F(x) - F(y)\| & \leq & \left( 2 L_0 + (1 + 2L_1 R)\|\M\| + \sqrt{2} L_1 \|F(x)\| \right) \|x - y\|.
\end{eqnarray*}

\textbf{Example 6:}
% \label{subsec:example4}
In this example, we consider an $N$-player game~\cite{balduzzi2018mechanics, loizou2021stochastic, yoon2025multiplayer}, where each player $i \in [N]$ selects an action $w_i \in \R^{d_i}$, and the joint action vector of all players is denoted as $x = (w_1, w_2, \cdots, w_N) \in \R^{d_1 + \cdots d_N}$. Each player $i$ aims to minimize their loss function $f_i$ for their action $w_i$. The objective is to find an equilibrium $x_* = (w_{1*}, w_{2*}, \cdots, w_{N*})$ such that  
\begin{eqnarray*}
    w_{i*} = \argmin_{w_i \in \R^{d_i}} f_i (w_i, w_{-i*}).
\end{eqnarray*}
Here we abuse the notation to denote $w_{-i} = (w_1, \cdots, w_{i - 1}, w_{i+1}, \cdots, w_N)$ and $f_i(w_i, w_{-i}) = f_i (w_1, \cdots, w_N)$. When the functions $f_i$ are convex, this equilibrium corresponds to solving $F(x_*) = 0$, where the operator $F$ is defined as  
\begin{eqnarray*}
    F(x) = \left( \nabla_1 f_1 (x), \nabla_2 f_2 (x), \dots, \nabla_N f_N(x) \right).
\end{eqnarray*}

In case each of these partial gradients $\nabla_i f_i$ are $(L_0, L_1)$-Lipschitz i.e.
\begin{eqnarray*}
    \| \nabla_i f_i(x) - \nabla_i f_i (y)\| & \leq & (L_0 + L_1 \| \nabla_i f_i (x)\|) \| x - y\|
\end{eqnarray*}
then we obtain
\begin{eqnarray*}
    \| F(x) - F(y) \|^2 & = & \sum_{i = 1}^N \| \nabla_i f_i(x) - \nabla_i f_i(y) \|^2 \\
    & \leq & \sum_{i = 1}^N \left( L_0 + L_1 \left\| \nabla_i f_i(x) \right\| \right)^2 \|x - y\|^2 \\
    & \leq & \|x - y\|^2 \sum_{i = 1}^N (2L_0^2 + 2 L_1^2 \| \nabla_i f_i (x)\|^2) \\
    & = & \|x - y\|^2 (2N L_0^2 + 2 L_1^2 \| F(x)\|^2) \\
    & \leq & \|x - y\|^2 ( \sqrt{2N}L_0 + \sqrt{2} L_1 \| F(x)\|)^2.
\end{eqnarray*}
This completes the proof.

\newpage
\section{Convergence Analysis}\label{appendix:convergence_analysis}
In this section, we present the missing proofs from Section \ref{sec:convergence_analysis}. We start with the proof of Proposition \ref{prop:equiv_formulation} and then provide the results related to strongly monotone, monotone and weak Minty problems. 

\subsection{Proof of Proposition \ref{prop:equiv_formulation}}

\begin{proposition}
    Suppose $F$ is $\alpha$-symmetric $(L_0, L_1)$-Lipschitz operator. Then, for $\alpha = 1$
    \begin{equation*}
    \textstyle
        \| F(x) - F(y) \| \leq (L_0 + L_1 \| F(x)\|) \exp{(L_1 \|x - y\|)} \| x- y \|,
    \end{equation*}
    and for $\alpha \in (0, 1)$ we have
    \begin{equation*}
    \textstyle
        \| F(x) - F(y) \| \leq \left(K_0 + K_1 \|F(x)\|^{\alpha} + K_2 \|x - y\|^{\nicefrac{\alpha}{1 - \alpha}} \right) \|x - y\|
    \end{equation*}
    where $K_0 = L_0 (2^{\nicefrac{\alpha^2}{1 - \alpha}} + 1)$, $K_1 = L_1 \cdot 2^{\nicefrac{\alpha^2}{1 - \alpha}}$ and $K_2 = L_1^{\nicefrac{1}{1 - \alpha}} \cdot 2^{\nicefrac{\alpha^2}{1 - \alpha}} \cdot 3^{\alpha} (1 - \alpha)^{\nicefrac{\alpha}{1 - \alpha}}$.
\end{proposition}

\begin{proof}
For proving this theorem, we follow the proof technique similar to \cite{chen2023generalized}. We start with $\alpha = 1$ case. Let $x, y \in \mathbb{R}^d$ and define $x_\theta := \theta x + (1 - \theta) y$. Since $F$ is $1$-symmetric $(L_0, L_1)$-Lipschitz , we have for all $\theta \in [0, 1]$,
$$
\| F(x_\theta) - F(y) \| \overset{\eqref{eq:reform_integration}}{\leq} \left(L_0 + L_1 \int_0^1 \| F(x_{\theta \tau}) \| d\tau \right) \| x_\theta - y \|.
$$
Note that
$$
x_{\theta \tau} = \tau x_\theta + (1 - \tau) y = \tau (\theta x + (1 - \theta) y) + (1 - \tau) y = \theta \tau x + (1 - \theta \tau) y.
$$
Let us define a function
$$
H(\theta) := L_0 \theta + L_1 \int_0^\theta \| F(x_u) \| du.
$$

Then, note that $H'(\theta) = L_0 + L_1 \| F(x_\theta) \|$. Moreover, we have
\begin{eqnarray*}
    \| F(x_\theta) - F(y) \| & \leq & \left(L_0 + L_1 \int_0^1 \| F(x_{\theta \tau}) \| d\tau \right) \| x_\theta - y \| \\
    & = & \left(L_0 + L_1 \int_0^1 \| F(x_{\theta \tau}) \| d\tau \right) \| \theta x + (1 - \theta) y - y \| \\
    & = & \left(L_0 + L_1 \int_0^1 \| F(x_{\theta \tau}) \| d\tau \right) \| \theta x - \theta y \| \\
    & = & \left(L_0 \theta + L_1 \int_0^1 \| F(x_{\theta \tau}) \| \theta d\tau \right) \| x -  y \| \\
    & = & \left(L_0 \theta + L_1 \int_0^1 \| F(\theta \tau x + (1 - \theta \tau) y) \| \theta d\tau \right) \| x -  y \| \\
    & = & \left(L_0 \theta + L_1 \int_0^{\theta} \| F(u x + (1 - u) y) \| du \right) \| x -  y \| \\
    & = & \left(L_0 \theta + L_1 \int_0^{\theta} \| F(x_u) \| du \right) \| x -  y \| \\
    & = & H(\theta) \| x -  y \|.
\end{eqnarray*}
Therefore we obtain
\begin{eqnarray*}
    H'(\theta) & = & L_0 + L_1 \| F(x_\theta) \|\\ 
    & \leq & L_0 + L_1 \left(\|F(x_\theta) - F(y)\| + \|F(y)\|\right)\\
    & \le & L_0 + L_1 \left(H(\theta) \|x - y\| + \|F(y)\|\right) \\
    & = & a H(\theta) + b,
\end{eqnarray*}
where  $a = L_1 \|x - y\|, b = L_0 + L_1 \|F(y)\|$.
Then we integrate both sides for $\theta \in [0, \theta']$ to get
$$
H(\theta') \le \frac{b}{a} (e^{a \theta'} - 1) .
$$
Here, we set $\theta' = 1$ to obtain
\begin{eqnarray*}
    H(1) &\le & \frac{b}{a} (e^{a} - 1) \\
    & = & \frac{L_0 + L_1 \|F(y)\|}{L_1 \|x - y\|} (e^{L_1 \|x - y\|} - 1).    
\end{eqnarray*}
Now, put this back into the original inequality
\begin{eqnarray*}
    \|F(x) - F(y)\| & \le & H(1) \|x - y\| \\
    & \le &  \left( L_0 + L_1 \|F(y)\| \right) \cdot \frac{e^{L_1 \|x - y\|} - 1}{L_1} \\
    & = & \left( \frac{L_0}{L_1} + \|F(y)\| \right) (e^{L_1 \|x - y\|} - 1).    
\end{eqnarray*}

Finally, using the inequality $e^z - 1 \le z e^z$ for $z \ge 0$, we get
\begin{eqnarray*}
    \|F(x) - F(y)\| & \le & \left( \frac{L_0}{L_1} + \|F(y)\| \right) L_1 \|x - y\| e^{L_1 \|x - y\|} \\
    & \leq & (L_0 + L_1 \|F(y)\|) e^{L_1 \|x - y\|} \|x - y\|
\end{eqnarray*}
This completes the proof for $\alpha = 1$. The proof for $\alpha \in (0, 1)$ follows similarly from \cite{chen2023generalized}. 
\end{proof}

% \newpage
\subsection{Convergence Guarantees for Strongly Monotone Operators}
\begin{lemma}
    Suppose $F$ is $\mu$-strongly monotone operator. Then Extragradient method with step size $\gamma_k = \omega_k$ satisfy
    \begin{equation}\label{eq:strong_monotone_EG}
        \|x_{k+1} - x_*\|^2 \leq (1 - \gamma_k \mu) \left\| x_k - x_* \right\|^2 -\gamma_k^2(1 - 2 \gamma_k \mu) \|F(x_k)\|^2 + \gamma_k^2 \|F(x_k) - F(\hx_k)\|^2.
    \end{equation}
\end{lemma}
\begin{proof}
    From the update step of the Extragradient method, we obtain
    \begin{eqnarray*}
        \|x_{k+1} - x_*\|^2 & = & \| x_k - \gamma_k F(\hx_k) - x_* \|^2 \notag\\
        & = & \| x_k - x_* \|^2 - 2 \gamma_k \la F(\hx_k), x_k - x_* \ra + \gamma_k^2  \| F(\hx_k) \|^2 \notag\\
        & = & \| x_k - x_* \|^2 - 2 \gamma_k \la F(\hx_k), \hx_k - x_* \ra - 2 \gamma_k \la F(\hx_k), x_k - \hx_k \ra + \gamma_k^2  \| F(\hx_k) \|^2 \notag\\
        & \overset{\eqref{eq:strong_monotone}}{\leq} & \| x_k - x_* \|^2 - 2 \gamma_k \mu \| \hx_k - x_*\|^2 - 2 \gamma_k \la F(\hx_k), x_k - \hx_k \ra  + \gamma_k^2  \| F(\hx_k) \|^2 \notag\\
        & \overset{\eqref{eq:inequality1}}{\leq} & \| x_k - x_* \|^2 - \gamma_k \mu \| x_k - x_*\|^2 + 2 \gamma_k \mu \|x_k - \hx_k\|^2 - 2 \gamma_k \la F(\hx_k), x_k - \hx_k \ra \notag \\ 
        && + \gamma_k^2  \| F(\hx_k) \|^2 \notag\\
        & = & (1 - \gamma_k \mu) \| x_k - x_* \|^2 + 2 \gamma_k \mu \|x_k - \hx_k\|^2 - 2 \gamma_k \la F(\hx_k), x_k - \hx_k \ra + \gamma_k^2  \| F(\hx_k) \|^2 \notag\\
        & = & (1 - \gamma_k \mu) \| x_k - x_* \|^2 + 2 \gamma_k^3 \mu \|F(x_k)\|^2 - 2 \gamma_k^2 \la F(\hx_k), F(x_k) \ra + \gamma_k^2  \| F(\hx_k) \|^2 \notag\\
        & \overset{\eqref{eq:inner_prod}}{=} & (1 - \gamma_k \mu) \left\| x_k - x_* \right\|^2 + 2 \gamma_k^3 \mu \|F(x_k)\|^2 \notag \\
        && - \gamma_k^2 \left(\|F(\hx_k)\|^2 + \|F(x_k)\|^2 - \|F(x_k) - F(\hx_k)\|^2 \right) + \gamma_k^2 \|F(\hx_k)\|^2 \notag\\
        & = & (1 - \gamma_k \mu) \left\| x_k - x_* \right\|^2 -\gamma_k^2(1 - 2 \gamma_k \mu) \|F(x_k)\|^2 + \gamma_k^2 \|F(x_k) - F(\hx_k)\|^2.
    \end{eqnarray*}
\end{proof}

\subsubsection{Proof of Theorem \ref{theorem:1symm_strongmonotone}}
\begin{theorem}
    Suppose $F$ is $\mu$-strongly monotone and $1$-symmetric $(L_0, L_1)$-Lipschitz operator. Then Extragradient method with step size $\gamma_k = \omega_k =  \frac{\nu}{L_0 + L_1 \|F(x_k)\|}$ satisfy
    \begin{eqnarray*}
        \|x_{k+1} - x_*\|^2 & \leq & \left( 1 - \frac{\nu \mu}{L_0 \left( 1 + L_1 \exp{\left( L_1 \|x_0 - x_*\|\right)} \|x_0 - x_* \|\right)} \right)^{k+1} \|x_0 - x_*\|^2
    \end{eqnarray*}
    where $\nu > 0$ root of $ 1 - 2 \nu - \nu^2 \exp{2\nu} = 0$.
\end{theorem}

\begin{proof}
    From the update rule of Extragradient and using $\mu$-strong monotonicty, we obtain
    \begin{eqnarray}\label{eq:1symm_strongmonotone_eq}
        \|x_{k+1} - x_*\|^2 
        & \overset{\eqref{eq:strong_monotone_EG}}{\leq} & (1 - \gamma_k \mu) \left\| x_k - x_* \right\|^2 -\gamma_k^2(1 - 2 \gamma_k \mu) \|F(x_k)\|^2 + \gamma_k^2 \|F(x_k) - F(\hx_k)\|^2 \notag\\
        & \overset{\eqref{eq:alpha=1}}{\leq} & (1 - \gamma_k \mu) \left\| x_k - x_* \right\|^2 -\gamma_k^2(1 - 2 \gamma_k \mu) \|F(x_k)\|^2 \notag \\
        && + \gamma_k^2 (L_0 + L_1 \|F(x_k)\|)^2 \exp{(2L_1 \|x_k - \hx_k\|)} \|x_k -\hx_k\|^2 \notag\\
        & = & (1 - \gamma_k \mu) \|x_k - x_*\|^2 \notag \\
        && - \gamma_k^2 \left(1 - 2 \gamma_k \mu - \gamma_k^2 (L_0 + L_1 \|F(x_k)\|)^2 \exp{(2 \gamma_k L_1 \|F(x_k)\|)} \right) \|F(x_k)\|^2.
    \end{eqnarray}
    Now we set $\gamma_k = \frac{\nu}{L_0 + L_1 \|F(x_k)\|}$. Then we want to choose $\nu$ such that 
    \begin{eqnarray*}
        1 - \frac{2 \nu \mu}{L_0 + L_1 \|F(x_k)\|} - \nu^2 \exp{2 \nu} \geq 0.
    \end{eqnarray*}
    Note that $\mu \leq L_0 + L_1 \| F(x_k) \|$ for any $x_k$. Therefore, we get $1 - \frac{2 \nu \mu}{L_0 + L_1 \|F(x_k)\|} - \nu^2 \exp{2 \nu} \geq 1 - 2 \nu - \nu^2 \exp{2 \nu}$ and it is enough to choose $\nu$ such that 
    \begin{eqnarray*}
        1 - 2 \nu - \nu^2 \exp{2\nu} \geq 0.
    \end{eqnarray*}
    This inequality holds for any $\nu \leq 0.22$. Hence, for this choice of $\nu$ we get the following inequality from \eqref{eq:1symm_strongmonotone_eq}.
    \begin{eqnarray}\label{eq:1symm_strongmonotone_eq1}
        \|x_{k+1} - x_*\|^2 & \leq & (1 - \gamma_k \mu) \|x_k - x_*\|^2.
    \end{eqnarray}
    This proves that the distance of the iterates $x_k$ from $x_*$ are bounded by $\|x_0 - x_*\|$.
    Now note that using \eqref{eq:alpha=1} with $x = x_k, y = x_*$ with $\|x_k - x_*\| \leq \|x_0 - x_*\|$ we get
    \begin{eqnarray}\label{eq:1symm_strongmonotone_eq2}
        \|F(x_k)\| & \overset{\eqref{eq:alpha=1}}{\leq} & L_0 \exp{\left(L_1 \|x_k - x_*\| \right)} \|x_k - x_*\| \notag \\
        & \overset{\eqref{eq:1symm_strongmonotone_eq1}}{\leq} & L_0 \exp{\left(L_1 \|x_0 - x_*\| \right)} \|x_0 - x_*\|.
    \end{eqnarray}
    Therefore, we have the following lower bound on the step size 
    \begin{eqnarray}\label{eq:1symm_strongmonotone_eq3}
        \gamma_k & = & \frac{\nu}{L_0 + L_1 \|F(x_k)\|} \notag \\
        & \overset{\eqref{eq:1symm_strongmonotone_eq2}}{\geq} & \frac{\nu}{L_0 \left( 1 + L_1 \exp{\left( L_1 \|x_0 - x_*\|\right)} \|x_0 - x_* \|\right)}.
    \end{eqnarray}
    Hence, we get 
    \begin{eqnarray*}
        \|x_{k+1} - x_*\|^2 & \overset{\eqref{eq:1symm_strongmonotone_eq1},\eqref{eq:1symm_strongmonotone_eq3}}{\leq} & \left( 1 - \frac{\nu \mu}{L_0 \left[ 1 + L_1 \exp{\left( L_1 \|x_0 - x_*\|\right)} \|x_0 - x_* \|\right]} \right) \|x_k - x_*\|^2 \\
        & \leq & \left( 1 - \frac{\nu \mu}{L_0 \left[ 1 + L_1 \exp{\left( L_1 \|x_0 - x_*\|\right)} \|x_0 - x_* \|\right]} \right)^{k+1} \|x_0 - x_*\|^2.
    \end{eqnarray*}
\end{proof}

\subsubsection{Proof of Corollary \ref{corollary:1symm_strongmonotone}}
\begin{corollary}
    Suppose $F$ is $\mu$-strongly monotone and $1$-symmetric $(L_0, L_1)$-Lipschitz operator. Then Extragradient with step size $\gamma_k = \omega_k = \frac{\nu}{L_0 + L_1 \|F(x_k)\|}$ satisfy $\|x_{K+1} - x_*\|^2  \leq  \varepsilon$ after
    \begin{equation*}
        K = \frac{2 L_0}{\nu \mu} \log \left( \frac{\|x_0 - x_*\|^2}{\varepsilon}\right) + \frac{1}{\gamma \mu} \log \left( \frac{2L_1 \|x_0 - x_*\|^2}{\gamma^2 L_0} \right)
    \end{equation*}
    many iterations, where $\nu > 0$ satisfy $1 - 4 \nu - 2
    \nu^2 \exp{2\nu} = 0$ and 
    \begin{equation*}
        \gamma \eqdef \frac{\nu}{L_0 \left( 1 + L_1 \exp{\left( L_1 \|x_0 - x_*\|\right)} \|x_0 - x_* \|\right)}.
    \end{equation*}
\end{corollary}

\begin{proof}
    We set $\gamma_k = \frac{\nu}{L_0 + L_1 \|F(x_k)\|}$ and we choose $\nu \in (0, 1)$ such that $1 - 4 \nu - 2
    \nu^2 \exp{2\nu} = 0$. Then we have
    \begin{eqnarray}\label{eq:1symm_strongmonotone_cor_eq1}
        1 - 2 \gamma_k \mu - \gamma_k^2 (L_0 + L_1 \|F(x_k)\|)^2 \exp{(2 \gamma_k L_1 \|F(x_k)\|)} & \geq & 1 - \frac{2 \nu \mu}{L_0 + L_1 \|F(x_k)\|} - \nu^2 \exp{2 \nu} \notag \\
        & \geq & 1 - 2 \nu - \nu^2 \exp{2 \nu} \notag \\
        & = & \frac{1}{2}. 
    \end{eqnarray}
    Therefore, from \eqref{eq:1symm_strongmonotone_eq} and \eqref{eq:1symm_strongmonotone_cor_eq1} we get 
    \begin{eqnarray}\label{eq:1symm_strongmonoton_noexp_eq1}
        \|x_{k+1} - x_*\|^2 & \leq & (1 - \gamma_k \mu) \|x_k - x_*\|^2 - \frac{\gamma_k^2}{2} \|F(x_k)\|^2.
    \end{eqnarray}
    In \eqref{eq:1symm_strongmonotone_eq3}, we found that a lower bound of $\gamma_k$ is 
    \begin{eqnarray*}
        \gamma_k \geq \gamma \eqdef \frac{\nu}{L_0 \left( 1 + L_1 \exp{\left( L_1 \|x_0 - x_*\|\right)} \|x_0 - x_* \|\right)}.
    \end{eqnarray*}
    Then from \eqref{eq:1symm_strongmonoton_noexp_eq1} we get
    \begin{equation}\label{eq:1symm_strongmonoton_noexp_eq2}
        \frac{\gamma^2}{2} \|F(x_k)\|^2 \leq (1 - \gamma \mu)\|x_k - x_*\|^2 \leq (1 - \gamma \mu)^{k+1} \|x_0 - x_*\|^2
    \end{equation}
    This can be rearranged to write
    \begin{equation*}
        \|F(x_k)\|^2 \leq \frac{2(1 - \gamma \mu)^{k+1}}{\gamma^2} \|x_0 - x_*\|^2.
    \end{equation*}
    This implies $\|F(x_k)\|^2 \leq \frac{L_0}{L_1}$ after 
    \begin{equation}
        K' = \frac{1}{\gamma \mu} \log \left( \frac{2L_1 \|x_0 - x_*\|^2}{\gamma^2 L_0} \right)
    \end{equation}
    many iterations. Hence for $k \geq K'$ we have
    \begin{eqnarray*}
        \gamma_k & = & \frac{\nu}{L_0 + L_1 \|F(x_k)\|} \\
        & \geq & \frac{\nu}{2 L_0}.
    \end{eqnarray*}
    In the last inequality we used $ \|F(x_k)\|^2 \leq \frac{L_0}{L_1}$ for $k \geq K'$. Therefore for $k \geq K'$ we obtain
    \begin{eqnarray*}
        \|x_{k+1} - x_*\|^2 & \leq & \left(1 - \frac{\nu \mu}{2 L_0} \right) \|x_k - x_*\|^2 \\
        & \leq & \left(1 - \frac{\nu \mu}{2 L_0} \right)^{k+1 - K'} \|x_{K'} - x_*\|^2 \\
        & \leq & \left(1 - \frac{\nu \mu}{2 L_0} \right)^{k+1 - K'} \|x_0 - x_*\|^2
    \end{eqnarray*}
    Thus we conclude that, $\|x_{K+1} - x_*\|^2 \leq \varepsilon$ after atmost 
    \begin{equation*}
        K = \frac{2 L_0}{\nu \mu} \log \left( \frac{\|x_0 - x_*\|^2}{\varepsilon}\right) + K'
    \end{equation*}
    many iterations.
\end{proof}

\subsubsection{Proof of Theorem \ref{theorem:1symm_strongmonotone_alhpha01}}
\begin{theorem}
    Suppose $F$ is $\mu$-strongly monotone and $\alpha$-symmetric $(L_0, L_1)$-Lipschitz operator with $\alpha \in (0, 1)$. Then Extragradient method with $\gamma_k = \omega_k = \frac{\nu}{2K_0 + \left( 2 K_1 + 2^{1 - \alpha} K_2^{1 - \alpha} \right) \| F(x_k)\|^{\alpha}}$ satisfy
    \begin{eqnarray*}
    \textstyle
        \|x_{k+1} - x_*\|^2 \leq \left( 1 - \frac{\nu \mu}{2 K_0 + (2 K_1 + 2^{1 - \alpha} K_2^{1 - \alpha})(K_0 + K_2 \|x_0 - x_*\|^{\nicefrac{\alpha}{1 - \alpha}})^{\alpha} \|x_0 - x_*\|^{\alpha}}\right)^{k+1} \|x_0 - x_*\|^2.
    \end{eqnarray*}
    where $\nu \in (0, 1)$ is a constant such that $1 - \nu - \nu^2 \geq 0$.
\end{theorem}

\begin{proof} 
    Using the update steps of the Extragradient method and using $\mu$-strong monotonicity, we have
    \begin{eqnarray*}
        \|x_{k+1} - x_*\|^2 
        & \overset{\eqref{eq:strong_monotone_EG}}{\leq} & (1 - \gamma_k \mu) \left\| x_k - x_* \right\|^2 -\gamma_k^2(1 - 2 \gamma_k \mu) \|F(x_k)\|^2 + \gamma_k^2 \|F(x_k) - F(\hx_k)\|^2 \\
        & \overset{\eqref{eq:alpha(0,1)}}{\leq} & (1 - \gamma_k \mu) \left\| x_k - x_* \right\|^2 -\gamma_k^2(1 - 2 \gamma_k \mu) \|F(x_k)\|^2 \\
        && + \gamma_k^2 \left(K_0 + K_1 \|F(x_k)\|^{\alpha} + K_2 \| x_k - \hx_k \|^{\nicefrac{\alpha}{1 - \alpha}} \right)^2 \|x_k - \hx_k\|^2 \\
        & = & (1 - \gamma_k \mu) \|x_k - x_*\|^2 \\
        && + \gamma_k^2 \left(1 - 2 \gamma_k \mu - \gamma_k^2 \left(K_0 + K_1 \|F(x_k)\|^{\alpha} + \gamma_k^{\nicefrac{\alpha}{1 - \alpha}} K_2 \| F(x_k) \|^{\nicefrac{\alpha}{1 - \alpha}} \right)^2 \right) \|F(x_k)\|^2.
    \end{eqnarray*}
    Here we will choose $\gamma_k > 0$ such that
    \begin{equation*}
        1 - 2 \gamma_k \mu  - \gamma_k^2 \left(K_0 + K_1 \|F(x_k)\|^{\alpha} + \gamma_k^{\nicefrac{\alpha}{1 - \alpha}} K_2 \| F(x_k) \|^{\nicefrac{\alpha}{1 - \alpha}} \right)^2 \geq 0
    \end{equation*} 
    Let us choose $\gamma_k = \frac{\nu}{2K_0 + \left( 2 K_1 + 2^{1 - \alpha} K_2^{1 - \alpha} \right) \| F(x_k)\|^{\alpha}}$ for some $\nu \in (0, 1)$. Then we observe that
    \begin{eqnarray*}
        \gamma_k \left( K_0 + K_1 \| F(x_k)\|^{\alpha} \right) + \gamma_k^{\nicefrac{1}{1 - \alpha}} K_2 \| F(x_k)\|^{\nicefrac{\alpha}{1 - \alpha}} & \leq &\frac{\nu \left( K_0 + K_1 \| F(x_k)\|^{\alpha} \right)}{2K_0 + \left( 2 K_1 + 2^{1 - \alpha} K_2^{1 - \alpha} \right) \| F(x_k)\|^{\alpha}} \\
        && + \frac{\nu^{\nicefrac{1}{1 - \alpha}} K_2 \|F(x_k)\|^{\nicefrac{\alpha}{1 - \alpha}}}{\left( 2K_0 + \left( 2K_1 + 2^{1 - \alpha} K_2^{1 - \alpha} \right) \| F(x_k)\|^{\alpha}\right)^{\nicefrac{1}{1 - \alpha}}} \\
        & \leq & \frac{\nu}{2} + \frac{\nu^{\nicefrac{1}{1 - \alpha}}}{2} \\
        & \leq & \nu.        
    \end{eqnarray*}
    The last inequality follows from $\nu \in (0, 1)$. Therefore it is enough to choose $\nu \in (0, 1)$ such that 
    \begin{equation*}
        1 - \frac{2 \nu \mu}{2K_0 + \left( 2 K_1 + 2^{1 - \alpha} K_2^{1 - \alpha} \right) \| F(x_k)\|^{\alpha}} - \nu^2 \geq 0
    \end{equation*}
    However, note that, we always have $\mu \leq K_0$, thus it is enough to choose $\nu \in (0, 1)$ such that 
    \begin{equation*}
        1 - \nu - \nu^2 \geq 0.
    \end{equation*}
    Hence, for this choice of $\nu$ we get
    \begin{eqnarray*}
        \|x_{k+1} - x_*\|^2 & \leq & (1 - \gamma_k \mu) \|x_0 - x_*\|^2.
    \end{eqnarray*}
    Here we lower bound the step size $\gamma_k$ with 
    \begin{eqnarray*}
        \gamma_k \geq \frac{\nu}{2 K_0 + (2 K_1 + 2^{1 - \alpha} K_2^{1 - \alpha})(K_0 + K_2 \|x_0 - x_*\|^{\nicefrac{\alpha}{1 - \alpha}})^{\alpha} \|x_0 - x_*\|^{\alpha}}. 
    \end{eqnarray*}
    Hence we obtain
    \begin{eqnarray*}
        \|x_{k+1} - x_*\|^2 \leq & \left( 1 - \frac{\nu \mu}{2 K_0 + (2 K_1 + 2^{1 - \alpha} K_2^{1 - \alpha})(K_0 + K_2 \|x_0 - x_*\|^{\nicefrac{\alpha}{1 - \alpha}})^{\alpha} \|x_0 - x_*\|^{\alpha}}\right) \|x_k - x_*\|^2 \\
        \leq & \left( 1 - \frac{\nu \mu}{2 K_0 + (2 K_1 + 2^{1 - \alpha} K_2^{1 - \alpha})(K_0 + K_2 \|x_0 - x_*\|^{\nicefrac{\alpha}{1 - \alpha}})^{\alpha} \|x_0 - x_*\|^{\alpha}}\right)^{k+1} \|x_0 - x_*\|^2.
    \end{eqnarray*}
    This completes the proof of the theorem.
\end{proof}

%%%%%%%% MONOTONE %%%%%%%%
\newpage
\subsection{Convergence Guarantees for Monotone Operators}\label{subsec:monotone_proofs}

\begin{lemma}
    Suppose $F$ is a monotone operator. Then \algname{EG} with step size $\gamma_k = \omega_k$ satisfy
    \begin{equation}\label{eq:monotone_EG}
         \left\| x_{k+1} - x_* \right\|^2 \leq \left\| x_k - x_* \right\|^2 - \gamma_k^2 \|F(x_k)\|^2 + \gamma_k^2 \| F(x_k) - F(\hx_k)\|^2.
    \end{equation}
\end{lemma}

\begin{proof}
    From the update rule of the Extragradient method, we have
    \begin{eqnarray*}
        \left\| x_{k+1} - x_* \right\|^2 & = & \| x_k - \gamma_k F(\hx_k) - x_* \|^2 \notag \\
        & = & \| x_k - x_* \|^2 - 2 \gamma_k \la F(\hx_k), x_k - x_* \ra + \gamma_k^2  \| F(\hx_k) \|^2 \notag \\
        & = & \| x_k - x_* \|^2 - 2 \gamma_k \la F(\hx_k), \hx_k - x_* \ra - 2 \gamma_k \la F(\hx_k), x_k - \hx_k \ra + \gamma_k^2  \| F(\hx_k) \|^2 \notag \\
        & \overset{\eqref{eq:monotone}}{\leq} & \| x_k - x_* \|^2 - 2 \gamma_k \la F(\hx_k), x_k - \hx_k \ra + \gamma_k^2  \| F(\hx_k) \|^2 \\
        & = & \| x_k - x_* \|^2 - 2 \gamma_k^2 \la F(\hx_k), F(x_k) \ra + \gamma_k^2  \| F(\hx_k) \|^2 \\
        & \overset{\eqref{eq:inner_prod}}{=} & \left\| x_k - x_* \right\|^2 - \gamma_k^2 \|F(\hx_k)\|^2 - \gamma_k^2 \|F(x_k)\|^2 + \gamma_k^2 \| F(x_k) - F(\hx_k)\|^2 + \gamma_k^2 \| F(\hx_k)\|^2 \\
        & = & \left\| x_k - x_* \right\|^2 - \gamma_k^2 \|F(x_k)\|^2 + \gamma_k^2 \| F(x_k) - F(\hx_k)\|^2.
    \end{eqnarray*}
\end{proof}

\subsubsection{Proof of Theorem \ref{theorem:1symm_monotone}}

\begin{theorem}
    Suppose $F$ is monotone and $1$-symmetric $(L_0, L_1)$-Lipschitz operator. Then \algname{EG} with step size $\gamma_k = \omega_k = \frac{\nu}{L_0 + L_1 \|F(x_k)\|}$ satisfy 
    \begin{equation*}
        \min_{0 \leq k \leq K} \|F(x_k)\|^2 \leq \frac{2L_0^2 \left( 1 + L_1 \exp{\left( L_1 \|x_0 - x_*\|\right) \|x_0 - x_*\|}\right)^2 \|x_0 - x_*\|^2}{\nu^2(K+1)}.
    \end{equation*}
    where $\nu \exp{\nu} = \nicefrac{1}{\sqrt{2}}$.
\end{theorem}
\begin{proof}
    From the update rule of the Extragradient method and using monotonicity, we have
    \begin{eqnarray*}
        \left\| x_{k+1} - x_* \right\|^2 
        & \overset{\eqref{eq:monotone_EG}}{\leq} & \left\| x_k - x_* \right\|^2 - \gamma_k^2 \|F(x_k)\|^2 + \gamma_k^2 \| F(x_k) - F(\hx_k)\|^2 \\
        & \overset{\eqref{eq:alpha=1}}{\leq} & \left\| x_k - x_* \right\|^2 - \gamma_k^2 \| F(x_k)\|^2 \\
        && + \gamma_k^2 (L_0 + L_1 \|F(x_k)\|)^2 \exp{\left( 2 L_1 \|x_k - \hx_k \| \right) } \|x_k - \hx_k\|^2 \\
        & = & \left\| x_k - x_* \right\|^2 - \gamma_k^2 \left(1 - \gamma_k^2 (L_0 + L_1 \|F(x_k)\|)^2 \exp{\left( 2 \gamma_k L_1 \|F(x_k) \| \right)} \right) \| F(x_k)\|^2 \\
        & \leq & \left\| x_k - x_* \right\|^2 \\
        && - \gamma_k^2 \left(1 - \gamma_k^2 (L_0 + L_1 \|F(x_k)\|)^2 \exp{\left( 2 \gamma_k (L_0 + L_1 \|F(x_k) \|) \right)} \right) \| F(x_k)\|^2.
    \end{eqnarray*}
     Here we use $\gamma_k = \frac{\nu}{L_0 + L_1 \|F(x_k)\|}$ for some $\nu \in (0, 1)$ to get
    \begin{eqnarray*}
         \left\| x_{k+1} - x_* \right\|^2 & \leq & \left\| x_k - x_* \right\|^2 - \gamma_k^2 \left(1 - \nu^2  \exp{\left( 2 \nu \right)} \right) \| F(x_k)\|^2
    \end{eqnarray*}
    Then we choose $\nu$ such that $\nu \exp{\nu} = \nicefrac{1}{\sqrt{2}}$ to obtain
    \begin{eqnarray}\label{eq:1symm_monotone_eq1}
        \|x_{k+1} - x_*\|^2 & \leq & \|x_k - x_*\|^2 - \frac{\gamma_k^2}{2} \|F(x_k)\|^2.
    \end{eqnarray}
    In particular the distance of the iterates $x_k$ from $x_*$ are bounded i.e. $\|x_{k+1} - x_*\| \leq \|x_k - x_*\| \leq \|x_0 - x_*\|$. Therefore, using \eqref{eq:alpha=1} with $y = x_*$ and $x = x_k$, we get
    \begin{eqnarray*}
        \|F(x_k)\| & \leq & L_0 \exp{ \left( L_1 \|x_k - x_*\|\right)} \|x_k - x_*\| \\
        & \leq & L_0 \exp{\left(L_1\|x_0 - x_*\| \right)} \|x_0 - x_*\|.
    \end{eqnarray*}
    Then we have the lower bound on step size given as follows
    \begin{eqnarray}\label{eq:1symm_monotone_eq2}
        \gamma_k = \frac{\nu}{L_0 + L_1 \|F(x_k)\|} \geq \frac{\nu}{L_0 \left( 1 + L_1 \exp{\left(L_1 \|x_0 - x_*\| \right) \|x_0 - x_*\|}\right)}.
    \end{eqnarray}
    Rearranging \eqref{eq:1symm_monotone_eq1} we have
    \begin{eqnarray*}
        \frac{\gamma_k^2}{2} \|F(x_k)\|^2 & \leq & \|x_{k} - x_*\|^2 - \|x_{k+1} - x_*\|^2.
    \end{eqnarray*}
    Summing up the above inequality for $k = 0, 1, \cdots, K$ and dividing by $K+1$ we get
    \begin{equation}\label{eq:1symm_monotone_eq3}
        \frac{1}{K+1} \sum_{k = 0}^K \frac{\gamma_k^2}{2}\|F(x_k)\|^2 \leq \frac{\|x_0 - x_*\|^2 - \|x_{K+1} - x_*\|^2}{K+1} \leq \frac{\|x_0 - x_*\|^2}{K+1}.
    \end{equation}
    Here we will use the lower bound on step size $\gamma_k$ given in \eqref{eq:1symm_monotone_eq2} to get
    \begin{eqnarray*}
        \frac{1}{K+1} \sum_{k = 0}^K \|F(x_k)\|^2 & \leq & \frac{2L_0^2 \left( 1 + L_1 \exp{\left( L_1 \|x_0 - x_*\|\right) \|x_0 - x_*\|}\right)^2 \|x_0 - x_*\|^2}{\nu^2(K+1)}.
    \end{eqnarray*}
    Finally, note that $\min_{0 \leq k \leq K} \|F(x_k)\|^2 \leq \frac{1}{K+1} \sum_{k = 0}^K \|F(x_k)\|^2$. Therefore we have
    \begin{eqnarray*}
        \min_{0 \leq k \leq K} \|F(x_k)\|^2 & \leq &  \frac{2L_0^2 \left( 1 + L_1 \exp{\left( L_1 \|x_0 - x_*\|\right) \|x_0 - x_*\|}\right)^2 \|x_0 - x_*\|^2}{\nu^2(K+1)}.
    \end{eqnarray*}
    This completes the proof of the Theorem.
\end{proof}

\subsubsection{Proof of Theorem \ref{theorem:1symm_monotone_noexp}}
\begin{theorem}
    Suppose $F$ is monotone and $1$-symmetric $(L_0, L_1)$-Lipschitz operator. Then Extragradient method with step size $\gamma_k = \frac{\nu}{L_0 + L_1 \|F(x_k)\|}$ satisfy 
    \begin{eqnarray*}
        \min_{0 \leq k \leq K} \| F(x_k)\| & \leq & \frac{\sqrt{2} L_0 \|x_0 - x_*\|}{\nu \sqrt{K+1} - \sqrt{2} L_1 \|x_0 - x_*\|}
    \end{eqnarray*}
    where $\nu \exp{\nu} = \nicefrac{1}{\sqrt{2}}$ and $K+1 \geq \frac{2L_1^2 \|x_0 - x_*\|^2}{\nu^2}$.
\end{theorem}

\begin{proof}
    From \eqref{eq:1symm_monotone_eq3}, we know steps of Extragradient method satisfy
    \begin{eqnarray*}
        \frac{1}{K+1} \sum_{k = 0}^K \frac{\gamma_k^2}{2} \|F(x_k)\|^2 & \leq & \frac{\|x_0 - x_*\|^2}{K+1}.
    \end{eqnarray*}
    Taking the minimum on the left-hand side we have
    \begin{eqnarray*}
        \min_{0 \leq k \leq K} \gamma_k^2 \|F(x_k)\|^2 & \leq & \frac{2\|x_0 - x_*\|^2}{K+1},
    \end{eqnarray*}
    or equivalently, 
    \begin{eqnarray*}
        \min_{0 \leq k \leq K} \frac{\nu^2 \|F(x_k)\|^2}{(L_0 + L_1 \|F(x_k)\|)^2} & \leq & \frac{2\|x_0 - x_*\|^2}{K+1}.
    \end{eqnarray*}
    Taking the square root on both sides we have
    \begin{eqnarray*}
        \min_{0 \leq k \leq K} \frac{\nu \|F(x_k)\|}{L_0 + L_1 \|F(x_k)\|} & \leq & \frac{\sqrt{2}\|x_0 - x_*\|}{\sqrt{K+1}}.
    \end{eqnarray*}
    Therefore, for some $0 \leq k_0 \leq K$ we have
    \begin{eqnarray*}
        \frac{\|F(x_{k_0})\|}{L_0 + L_1 \|F(x_{k_0})\|} & \leq & \frac{\sqrt{2}\|x_0 - x_*\|}{\nu \sqrt{K+1}}.
    \end{eqnarray*}
    Therefore, rearranging these terms, we get
    \begin{eqnarray*}
       \left( \nu \sqrt{K+1} - \sqrt{2}L_1 \|x_0 - x_*\| \right) \| F(x_{k_0})\| & \leq & \sqrt{2} L_0 \|x_0 - x_*\|.
    \end{eqnarray*}
    When we have $K+1 \geq \frac{2L_1^2 \|x_0 - x_*\|^2}{\nu^2}$ then we can rearrange the terms to obtain
    \begin{eqnarray*}
        \| F(x_{k_0})\| & \leq & \frac{\sqrt{2} L_0 \|x_0 - x_*\|}{\left( \nu \sqrt{K+1} - \sqrt{2}L_1 \|x_0 - x_*\| \right)}.
    \end{eqnarray*}
    for some $0 \leq k_0 \leq K$. Hence, we complete the proof of the theorem
    \begin{eqnarray*}
        \min_{0 \leq k \leq K} \| F(x_k)\| & \leq & \frac{\sqrt{2} L_0 \|x_0 - x_*\|}{ \nu \sqrt{K+1} - \sqrt{2} L_1 \|x_0 - x_*\|}.
    \end{eqnarray*}
\end{proof}

\subsubsection{Proof of Theorem \ref{theorem:alpha01}}
\begin{theorem}
    Suppose $F$ is monotone and $\alpha$-symmetric $(L_0, L_1)$-Lipschitz operator with $\alpha \in (0, 1)$. Then Extragradient method with step size $\gamma_k = \frac{1}{2 \sqrt{2} K_0 + \left(2 \sqrt{2}K_1 + 2^{\nicefrac{3 (1 - \alpha)}{2}} K_2^{1 - \alpha} \right) \|F(x_k)\|^{\alpha}}$ satisfy 
    \begin{eqnarray*}
    \textstyle
        \min_{0 \leq k \leq K} \|F(x_k)\|^2 \leq \frac{16 \left( K_0 + \left(K_1 + 2^{\nicefrac{-3}{2}} K_2^{1 - \alpha} \right) (K_0 + K_2 \|x_0 - x_*\|^{\nicefrac{\alpha}{1 - \alpha}})^{\alpha} \|x_0 - x_*\|^{\alpha}\right)^2 \|x_0 - x_*\|^2}{K+1}.
    \end{eqnarray*}
\end{theorem}

\begin{proof}
     Here operator $F$ is a $\alpha$-symmetric $(L_0, L_1)$-Lipschitz i.e. it satisfies \eqref{eq:alpha(0,1)}. For the update steps of the Extragradient method, we have
    \begin{eqnarray}\label{eq:alpha01}
        \left\| x_{k+1} - x_* \right\|^2 
        & \overset{\eqref{eq:monotone_EG}}{\leq} & \left\| x_k - x_* \right\|^2 - \gamma_k^2 \|F(x_k)\|^2 + \gamma_k^2 \| F(x_k) - F(\hx_k)\|^2 \notag \\
        & \overset{\eqref{eq:alpha(0,1)}}{\leq} & \left\| x_k - x_* \right\|^2 - \gamma_k^2 \| F(x_k)\|^2 \notag \\
        && + \gamma_k^2 \left( K_0 + K_1 \|F(x_k)\|^{\alpha} + K_2 \| x_k - \hx_k \|^{\nicefrac{\alpha}{1 - \alpha}}\right)^2 \|x_k - \hx_k \|^2 \notag\\
        & = & \left\| x_k - x_* \right\|^2 \notag \\
        && \hspace{-3mm} - \gamma_k^2 \left(1 - \gamma_k^2 \left( K_0 + K_1 \|F(x_k)\|^{\alpha} +  \gamma_k^{\nicefrac{\alpha}{1 - \alpha}} K_2 \| F(x_k) \|^{\nicefrac{\alpha}{1 - \alpha}}\right)^2 \right) \|F(x_k)\|^2.
    \end{eqnarray}
    Here we want to choose $\gamma_k$ such that
    \begin{eqnarray*}
        \gamma_k \left( K_0 + K_1 \|F(x_k)\|^{\alpha} \right) + \gamma_k^{\nicefrac{1}{1 - \alpha}} K_2 \| F(x_k) \|^{\nicefrac{\alpha}{1 - \alpha}} & \leq & \frac{1}{\sqrt{2}}.
    \end{eqnarray*}
    For this, it is enough to make sure
    \begin{eqnarray*}
        \gamma_k \left( K_0 + K_1 \| F(x_k)\|^{\alpha} \right) \leq \frac{1}{2 \sqrt{2}} & \text{ and } & \gamma_k^{\nicefrac{1}{1 - \alpha}} K_2 \| F(x_k)\|^{\nicefrac{\alpha}{1 - \alpha}} \leq \frac{1}{2 \sqrt{2}}.
    \end{eqnarray*}
    Therefore, we choose $\gamma_k = \frac{1}{2 \sqrt{2} (K_0 + K_1 \| F(x_k)\|^{\alpha}) + 2^{\nicefrac{3 (1 - \alpha)}{2}} K_2^{1 - \alpha} \|F(x_k)\|^{\alpha}}$ and we get the following from \eqref{eq:alpha01}
    \begin{eqnarray}\label{eq:alpha01_eq3}
        \| x_{k+1} - x_*\|^2 & \leq & \| x_k - x_*\|^2 - \frac{\gamma_k^2}{2} \| F(x_k) \|^2.
    \end{eqnarray}
    Rearranging this inequality, we have
    \begin{eqnarray*}
        \frac{\gamma_k^2}{2} \| F(x_k)\|^2 & \leq & \|x_k - x_*\|^2 - \|x_{k+1} -  x_*\|^2.
    \end{eqnarray*}
    Then we sum up this inequality for $k = 0, 1, \cdots K$ to get
    \begin{eqnarray}\label{eq:alpha01_eq2}
        \frac{1}{K+1}\sum_{k = 0}^K \gamma_k^2 \| F(x_k)\|^2 & \leq & \frac{2\|x_0 - x_*\|^2}{K+1}.
    \end{eqnarray}
    For this step size, we also have $\|x_k - x_0\|^2 \leq \|x_0 - x_*\|^2$ from \eqref{eq:alpha01_eq3}. Now note that from \eqref{eq:alpha(0,1)} we obtain the following bound with $x = x_k$ and $y = x_*$
    \begin{eqnarray*}
        \|F(x_k)\|^{\alpha} & \leq & (K_0 + K_2 \|x_k - x_*\|^{\nicefrac{\alpha}{1 - \alpha}})^{\alpha} \|x_k - x_*\|^{\alpha} \\
        & \overset{\eqref{eq:alpha01_eq3}}{\leq} & (K_0 + K_2 \|x_0 - x_*\|^{\nicefrac{\alpha}{1 - \alpha}})^{\alpha} \|x_0 - x_*\|^{\alpha}.
    \end{eqnarray*}
    We use this to lower bound the step size $\gamma_k$ as follows 
    \begin{eqnarray*}
        \gamma_k & = & \frac{1}{2 \sqrt{2} (K_0 + K_1 \| F(x_k)\|^{\alpha}) + 2^{\nicefrac{3 (1 - \alpha)}{2}} K_2^{1 - \alpha} \|F(x_k)\|^{\alpha}} \\ 
        & \geq & \frac{1}{2 \sqrt{2} K_0 + 2 \sqrt{2} (K_1 + 2^{\nicefrac{-3}{2}} K_2^{1 - \alpha}) (K_0 + K_2 \|x_0 - x_*\|^{\nicefrac{\alpha}{1 - \alpha}})^{\alpha} \|x_0 - x_*\|^{\alpha}}.
    \end{eqnarray*}
    Therefore from \eqref{eq:alpha01_eq2} we obtain
    \begin{eqnarray*}
        \min_{0 \leq k \leq K} \|F(x_k)\|^2 \leq \frac{16 \left( K_0 + (K_1 + 2^{\nicefrac{-3}{2}} K_2^{1 - \alpha}) (K_0 + K_2 \|x_0 - x_*\|^{\nicefrac{\alpha}{1 - \alpha}})^{\alpha} \|x_0 - x_*\|^{\alpha}\right)^2 \|x_0 - x_*\|^2}{K+1}.
    \end{eqnarray*}
\end{proof}

%%%%%%%% WEAK MINTY %%%%%%%%%%%%

\newpage
\subsection{Local Convergence Guarantees for Weak Minty Operator}
\subsubsection{Proof of Theorem \ref{theorem:weak_minty_alpha1}}

\begin{theorem}
    Suppose $F$ is weak Minty and $1$-symmetric $(L_0, L_1)$-Lipschitz assumption. Moreover we assume 
    \begin{equation}
        \Delta_1 \eqdef \frac{\nu}{L_0 \left( 1 + L_1 \|x_0 - x_*\| e^{L_1 \|x_0 - x_*\|}\right)} - 4 \rho > 0.
    \end{equation} Then \algname{EG} with step size $\gamma_k = \frac{\nu}{L_0 + L_1 \|F(x_k)\|}$ and $\omega_k = \nicefrac{\gamma_k}{2}$ satisfies
    \begin{eqnarray}
        \min_{0 \leq k \leq K} \|F(\hx_k)\|^2 & \leq & \frac{4 L_0 \left( 1 + L_1 \exp{\left( L_1 \|x_0 - x_*\|\right) \|x_0 - x_*\|}\right) \|x_0 - x_*\|^2}{\nu \Delta_1 (K+1)} 
    \end{eqnarray}
    where $\nu \exp{\nu} = 1$.
\end{theorem}

\begin{proof}
    From the update rule of the Extragradient method, we have
    \begin{eqnarray*}
        \left\| x_{k+1} - x_* \right\|^2 & = & \left\| x_k - \frac{\gamma_k}{2} F(\hx_k) - x_* \right\|^2 \notag \\
        & = & \| x_k - x_* \|^2 - \gamma_k \la F(\hx_k), x_k - x_* \ra + \frac{\gamma_k^2}{4}  \| F(\hx_k) \|^2 \notag \\
        & = & \| x_k - x_* \|^2 - \gamma_k \la F(\hx_k), \hx_k - x_* \ra - \gamma_k \la F(\hx_k), x_k - \hx_k \ra + \frac{\gamma_k^2}{4} \| F(\hx_k) \|^2 \notag \\
        & \overset{\eqref{eq:weak_minty}}{\leq} & \| x_k - x_* \|^2 + \gamma_k \rho \| F(\hx_k) \|^2  - \gamma_k \la F(\hx_k), x_k - \hx_k \ra + \frac{\gamma_k^2}{4} \| F(\hx_k) \|^2 \\
        & = & \| x_k - x_* \|^2 + \gamma_k \rho \| F(\hx_k) \|^2 - \gamma_k^2 \la F(\hx_k), F(x_k) \ra + \frac{\gamma_k^2}{4}  \| F(\hx_k) \|^2 \\
        & \overset{\eqref{eq:inner_prod}}{=} & \left\| x_k - x_* \right\|^2 + \gamma_k \rho \| F(\hx_k) \|^2 - \frac{\gamma_k^2}{2} \|F(\hx_k)\|^2 - \frac{\gamma_k^2}{2} \|F(x_k)\|^2  \\
        && + \frac{\gamma_k^2}{2} \| F(x_k) - F(\hx_k)\|^2 + \frac{\gamma_k^2}{4} \| F(\hx_k)\|^2 \\
        & = & \left\| x_k - x_* \right\|^2 + \gamma_k \rho \| F(\hx_k) \|^2 - \frac{\gamma_k^2}{4} \|F(\hx_k)\|^2 - \frac{\gamma_k^2}{2} \|F(x_k)\|^2 \\
        && + \frac{\gamma_k^2}{2} \| F(x_k) - F(\hx_k)\|^2  \\
        & \overset{\eqref{eq:alpha=1}}{\leq} & \left\| x_k - x_* \right\|^2 + \gamma_k \rho \| F(\hx_k) \|^2 - \frac{\gamma_k^2}{4} \|F(\hx_k)\|^2  - \frac{\gamma_k^2}{2} \| F(x_k)\|^2 \\
        && + \frac{\gamma_k^2}{2} (L_0 + L_1 \|F(x_k)\|)^2 \exp{\left( 2 L_1 \|x_k - \hx_k \| \right) } \|x_k - \hx_k\|^2 \\
        & = & \left\| x_k - x_* \right\|^2 - \frac{\gamma_k}{4} \left( \gamma_k - 4 \rho \right) \| F(\hx_k) \|^2 \\
        && - \frac{\gamma_k^2}{2} \left(1 - \gamma_k^2 (L_0 + L_1 \|F(x_k)\|)^2 \exp{\left( 2 \gamma_k L_1 \|F(x_k) \| \right)} \right) \| F(x_k)\|^2 \\
        & \leq & \left\| x_k - x_* \right\|^2 - \frac{\gamma_k}{4} \left( \gamma_k - 4 \rho \right) \| F(\hx_k) \|^2 \\ 
        && - \frac{\gamma_k^2}{2} \left(1 - \gamma_k^2 (L_0 + L_1 \|F(x_k)\|)^2 \exp{\left( 2 \gamma_k (L_0 + L_1 \|F(x_k) \|) \right)} \right) \| F(x_k)\|^2.
    \end{eqnarray*}
    Similar to the proof of Theorem \ref{theorem:1symm_monotone}, we have
    \begin{eqnarray*}
        \|x_{k+1} - x_*\|^2 & \leq & \|x_k - x_*\|^2 - \frac{\gamma_k}{4}(\gamma_k - 4 \rho) \| F(\hx_k)\|^2
    \end{eqnarray*}
    for $\gamma_k = \frac{\nu}{L_0 + L_1 \| F(x_k)\|}$ and $\nu \exp{\nu} = 1$. Again similar to Theorem \ref{theorem:1symm_monotone}, step size $\gamma_k$ is lower bounded with 
    \begin{equation*}
       \gamma_k = \frac{\nu}{L_0 + L_1 \|F(x_k)\|} \geq \frac{\nu}{L_0 \left( 1 + L_1 \exp{\left(L_1 \|x_0 - x_*\| \right) \|x_0 - x_*\|}\right)}.
    \end{equation*}
    Hence from \eqref{eq:restriction_rho} we get
    \begin{eqnarray*}
        \frac{\gamma_k \Delta_1}{4} \|F(\hx_k)\|^2 & \leq & \|x_k - x_*\|^2 - \|x_{k+1} - x_*\|^2
    \end{eqnarray*}
    Then we sum up this inequality for $k = 0, 1, \cdots, K$ to get
    \begin{eqnarray*}
        \frac{1}{K+1}\sum_{k = 0}^K \frac{\gamma_k \Delta_1}{4} \| F(\hx_k)\|^2 & \leq & \frac{\|x_0 - x_*\|^2}{K+1}.
    \end{eqnarray*}
    Therefore, we get
    \begin{equation*}
        \min_{0 \leq k \leq K} \| F(\hx_k)\|^2 \leq \frac{4 \|x_0 - x_*\|^2}{\gamma \Delta_1 (K+1)}.
    \end{equation*}
\end{proof}

\subsubsection{Proof of Theorem \ref{theorem:weak_minty_alpha01}}
\begin{theorem}
    Suppose $F$ is weak Minty and $\alpha$-symmetric $(L_0, L_1)$-Lipschitz with $\alpha \in (0, 1)$. Moreover we assume 
    \begin{equation*}
        \Delta_{\alpha} \eqdef \frac{1}{2 \sqrt{2} K_0 + 2 \sqrt{2} (K_1 + 2^{\nicefrac{-3}{2}} K_2^{1 - \alpha}) (K_0 + K_2 \|x_0 - x_*\|^{\nicefrac{\alpha}{1 - \alpha}})^{\alpha} \|x_0 - x_*\|^{\alpha}} - 4 \rho > 0.
    \end{equation*} 
    Then \algname{EG} with step size $\gamma_k = \frac{1}{2 \sqrt{2} K_0 + \left(2 \sqrt{2}K_1 + 2^{\nicefrac{3 (1 - \alpha)}{2}} K_2^{1 - \alpha} \right) \|F(x_k)\|^{\alpha}}$ and $\omega_k = \frac{\gamma_k}{2}$ satisfy 
    \begin{eqnarray*}
    \textstyle
        \min_{0 \leq k \leq K} \|F(\hx_k)\|^2 \leq \frac{4 \left( K_0 + \left(K_1 + 2^{\nicefrac{-3}{2}} K_2^{1 - \alpha} \right) (K_0 + K_2 \|x_0 - x_*\|^{\nicefrac{\alpha}{1 - \alpha}})^{\alpha} \|x_0 - x_*\|^{\alpha}\right) \|x_0 - x_*\|^2}{\Delta_{\alpha}(K+1)}.
    \end{eqnarray*}
\end{theorem}

\begin{proof}
    From the update rule of the Extragradient method, we have
    \begin{eqnarray*}
        \left\| x_{k+1} - x_* \right\|^2 & = & \left\| x_k - \frac{\gamma_k}{2} F(\hx_k) - x_* \right\|^2 \notag \\
        & = & \| x_k - x_* \|^2 - \gamma_k \la F(\hx_k), x_k - x_* \ra + \frac{\gamma_k^2}{4}  \| F(\hx_k) \|^2 \notag \\
        & = & \| x_k - x_* \|^2 - \gamma_k \la F(\hx_k), \hx_k - x_* \ra - \gamma_k \la F(\hx_k), x_k - \hx_k \ra + \frac{\gamma_k^2}{4} \| F(\hx_k) \|^2 \notag \\
        & \overset{\eqref{eq:weak_minty}}{\leq} & \| x_k - x_* \|^2 + \gamma_k \rho \| F(\hx_k) \|^2  - \gamma_k \la F(\hx_k), x_k - \hx_k \ra + \frac{\gamma_k^2}{4} \| F(\hx_k) \|^2 \\
        & = & \| x_k - x_* \|^2 + \gamma_k \rho \| F(\hx_k) \|^2 - \gamma_k^2 \la F(\hx_k), F(x_k) \ra + \frac{\gamma_k^2}{4}  \| F(\hx_k) \|^2 \\
        & \overset{\eqref{eq:inner_prod}}{=} & \left\| x_k - x_* \right\|^2 + \gamma_k \rho \| F(\hx_k) \|^2 - \frac{\gamma_k^2}{2} \|F(\hx_k)\|^2 - \frac{\gamma_k^2}{2} \|F(x_k)\|^2 \\
        && + \frac{\gamma_k^2}{2} \| F(x_k) - F(\hx_k)\|^2  + \frac{\gamma_k^2}{4} \| F(\hx_k)\|^2 \\
        & = & \left\| x_k - x_* \right\|^2 + \gamma_k \rho \| F(\hx_k) \|^2 - \frac{\gamma_k^2}{4} \|F(\hx_k)\|^2 - \frac{\gamma_k^2}{2} \|F(x_k)\|^2 \\
        && + \frac{\gamma_k^2}{2} \| F(x_k) - F(\hx_k)\|^2  \\
        & \overset{\eqref{eq:alpha=1}}{\leq} & \left\| x_k - x_* \right\|^2 + \gamma_k \rho \| F(\hx_k) \|^2 - \frac{\gamma_k^2}{4} \|F(\hx_k)\|^2  - \frac{\gamma_k^2}{2} \| F(x_k)\|^2 \\
        && + \frac{\gamma_k^2}{2} \left( K_0 + K_1 \|F(x_k)\|^{\alpha} +  K_2 \| x_k - \hx_k \|^{\nicefrac{\alpha}{1 - \alpha}}\right)^2 \|x_k - \hx_k\|^2 \\
        & = & \left\| x_k - x_* \right\|^2 - \frac{\gamma_k}{4} \left( \gamma_k - 4 \rho \right) \| F(\hx_k) \|^2 \\
        && - \frac{\gamma_k^2}{2} \left(1 - \gamma_k^2 \left( K_0 + K_1 \|F(x_k)\|^{\alpha} +  \gamma_k^{\nicefrac{\alpha}{1 - \alpha}} K_2 \| F(x_k) \|^{\nicefrac{\alpha}{1 - \alpha}}\right)^2 \right) \| F(x_k)\|^2 .
    \end{eqnarray*}
     Here we choose $\gamma_k = \frac{1}{2 (K_0 + K_1 \| F(x_k)\|^{\alpha}) + 2^{\nicefrac{3 (1 - \alpha)}{2}} K_2^{1 - \alpha} \|F(x_k)\|^{\alpha}}$ and we get the following from \eqref{eq:alpha01}
    \begin{eqnarray}\label{eq:alpha01_eq3}
        \| x_{k+1} - x_*\|^2 & \leq & \| x_k - x_*\|^2 - \frac{\gamma_k}{4} \left( \gamma_k - 4 \rho \right) \| F(\hx_k) \|^2.
    \end{eqnarray}
    Rearranging this inequality, we have
    \begin{eqnarray*}
       \frac{\gamma_k}{4} \left( \gamma_k - 4 \rho \right) \| F(\hx_k) \|^2 & \leq & \|x_k - x_*\|^2 - \|x_{k+1} -  x_*\|^2.
    \end{eqnarray*}
    Then we sum up this inequality for $k = 0, 1, \cdots K$ to get
    \begin{eqnarray}\label{eq:alpha01_eq2}
        \frac{1}{K+1}\sum_{k = 0}^K \frac{\gamma_k}{4} \left( \gamma_k - 4 \rho \right) \| F(\hx_k) \|^2 & \leq & \frac{2\|x_0 - x_*\|^2}{K+1}.
    \end{eqnarray}
    For this step size, we also have $\|x_k - x_*\|^2 \leq \|x_0 - x_*\|^2$ from \eqref{eq:alpha01_eq3}. Now note that from \eqref{eq:alpha(0,1)} we obtain the following bound with $x = x_k$ and $y = x_*$
    \begin{eqnarray*}
        \|F(x_k)\|^{\alpha} & \leq & (K_0 + K_2 \|x_k - x_*\|^{\nicefrac{\alpha}{1 - \alpha}})^{\alpha} \|x_k - x_*\|^{\alpha} \\
        & \overset{\eqref{eq:alpha01_eq3}}{\leq} & (K_0 + K_2 \|x_0 - x_*\|^{\nicefrac{\alpha}{1 - \alpha}})^{\alpha} \|x_0 - x_*\|^{\alpha}.
    \end{eqnarray*}
    We use this to lower bound the step size $\gamma_k$ as follows 
    \begin{eqnarray*}
        \gamma_k & = & \frac{1}{2 \sqrt{2} (K_0 + K_1 \| F(x_k)\|^{\alpha}) + 2^{\nicefrac{3 (1 - \alpha)}{2}} K_2^{1 - \alpha} \|F(x_k)\|^{\alpha}} \\ 
        & \geq & \frac{1}{2 \sqrt{2} K_0 + 2 \sqrt{2} (K_1 + 2^{\nicefrac{-3}{2}} K_2^{1 - \alpha}) (K_0 + K_2 \|x_0 - x_*\|^{\nicefrac{\alpha}{1 - \alpha}})^{\alpha} \|x_0 - x_*\|^{\alpha}}.
    \end{eqnarray*}
    Therefore from \eqref{eq:alpha01_eq2} we obtain
    \begin{eqnarray*}
        \min_{0 \leq k \leq K} \|F(x_k)\|^2 \leq \frac{4 \|x_0 - x_*\|^2}{\gamma \Delta (K+1)}.
    \end{eqnarray*}
\end{proof}

\newpage
\section{Equivalent Formulation of $\alpha$-Symmetric $(L_0, L_1)$-Lipschitz Assumption}\label{appendix:equiv_formulation}
In this section, we consider the min-max optimization problem given by $\min_{w_1} \max_{w_2} \mathcal{L}(w_1, w_2)$ and provide an equivalent formulation of $\alpha$-symmetric $(L_0, L_1)$-Lipschitz operator. Next, we provide an example where we use this formulation to compute the constants $\alpha, L_0, L_1$.

\subsection{Proof of Theorem \ref{thm:equiv_formulation}}
\begin{theorem}
    Suppose $F$ is the operator for the problem
    \begin{equation*}
        \min_{w_1} \max_{w_2} \mathcal{L}(w_1, w_2).
    \end{equation*}
    Then $F$ satisfies $\alpha$-symmetric $(L_0, L_1)$-Lipschitz assumption if and only if
    \begin{equation*}
        \|\mathbf{J}(x)\| = \sup_{\|u\| = 1}\|\mathbf{J}(x) u\| \leq L_0 + L_1 \|F(x)\|^{\alpha}
    \end{equation*}
    where 
    \begin{equation*}
        \mathbf{J}(x) = \begin{bmatrix}
        \nabla^2_{w_1 w_1} \mathcal{L}(w_1, w_2) & \nabla^2_{w_2 w_1} \mathcal{L}(w_1, w_2) \\
        -\nabla^2_{w_1 w_2} \mathcal{L}(w_1, w_2) & -\nabla^2_{w_2 w_2} \mathcal{L}(w_1, w_2)
        \end{bmatrix}.
    \end{equation*}
    Here $\| \mathbf{J}(x)\| = \sigma_{\max}(\mathbf{J}(x))$ i.e. maximum singular value of $\mathbf{J}(x)$.
\end{theorem}

\begin{proof}
    Following \eqref{eq:reform_integration}, we have the equivalent characterization of $F$ given by
    \begin{equation*}
        \|F(y) - F(x)\| \leq \left( L_0 + L_1 \int_{0}^1 \left\| F(\theta y + (1 - \theta)x)\right\|^{\alpha} d\theta \right) \|y - x\| \qquad \forall x, y \in \R^d.
    \end{equation*}
    As this inequality holds for any $x, y \in \R^d$, we choose $y = x + \theta' u$ where $\|u\| = 1$ and $\theta' \in (0, 1)$. Then we get
    \begin{equation*}
        \|F(x + \theta' u) - F(x)\| \leq \left( L_0 + L_1 \int_{0}^1 \left\| F(x + \theta' \theta u) \right\|^{\alpha} d\theta \right) \|\theta' u\| \qquad \forall x \in \R^d.
    \end{equation*}
    The right-hand side of this inequality can be rewritten as 
    \begin{eqnarray*}
        \left( L_0 + L_1 \int_{0}^1 \left\| F(x + \theta' \theta u) \right\|^{\alpha} d\theta \right) \|\theta' u\| & = & \theta' \left( L_0 + L_1 \int_{0}^1 \left\| F(x + \theta' \theta u) \right\|^{\alpha} d\theta \right) \\
        & = & L_0 \theta' + L_1 \int_{0}^1 \left\| F(x + \theta' \theta u) \right\|^{\alpha} \theta' d\theta \\
        & = & L_0 \theta' + L_1 \int_{0}^{\theta'} \left\| F(x + \varphi u) \right\|^{\alpha} d\varphi.
    \end{eqnarray*}
    In the last line, we used the change of variable with $\varphi = \theta' \theta$. Therefore, we get
    \begin{equation*}
        \frac{\|F(x + \theta' u) - F(x)\|}{\theta'} \leq L_0 + \frac{L_1}{\theta'} \int_{0}^{\theta'} \left\| F(x + \varphi u) \right\|^{\alpha} d\varphi.
    \end{equation*}
    Then we take $\theta' \to 0$ and use L'Hôpital's rule and Leibniz Integral rule to obtain
    \begin{equation*}
        \lim_{\theta' \to 0} \frac{\|F(x + \theta' u) - F(x)\|}{\theta'} \leq L_0 + L_1 \left\| F(x) \right\|^{\alpha}.
    \end{equation*}
    Moreover, note that the left-hand side is given by $\|\mathbf{J}(x)u\|$ where
    \begin{equation*}
        \mathbf{J}(x) = \begin{bmatrix}
        \nabla^2_{w_1 w_1} \mathcal{L}(w_1, w_2) & \nabla^2_{w_2 w_1} \mathcal{L}(w_1, w_2) \\
        -\nabla^2_{w_1 w_2} \mathcal{L}(w_1, w_2) & -\nabla^2_{w_2 w_2} \mathcal{L}(w_1, w_2)
        \end{bmatrix}.
    \end{equation*}
    Therefore, for any $\|u\| = 1$ we have 
    \begin{equation*}
        \|\mathbf{J}(x) u\| \leq L_0 + L_1 \|F(x)\|^{\alpha}.
    \end{equation*}
    Hence we get 
    \begin{equation*}
        \|\mathbf{J}(x)\| = \sup_{\|u\| = 1}\|\mathbf{J}(x) u\| \leq L_0 + L_1 \|F(x)\|^{\alpha}.
    \end{equation*}
    Now we want to show the other way, i.e. suppose we have $\|\mathbf{J}(x)\| \leq L_0 + L_1 \|F(x)\|^{\alpha}$. For this we define, 
    \begin{equation*}
        q(\theta) \eqdef F(\theta x + (1 - \theta) y).      
    \end{equation*}
    Then $q(1) = F(x)$ and $q(0) = F(y)$ and we have
    \begin{eqnarray*}
        \|F(x) - F(y)\| & = & \|q(1) - q(0)\|\\
        & = & \left\| \int_0^1 \frac{dq(\theta)}{d\theta} d\theta \right\| \\
        & = & \left\| \int_0^1 \frac{dF(\theta x + (1 - \theta) y)}{d\theta} d\theta \right\| \\
        & = & \left\| \int_0^1 \mathbf{J}(\theta x + (1 - \theta) y) (x - y) d\theta \right\| \\
        & \leq  &  \int_0^1 \left\|\mathbf{J}(\theta x + (1 - \theta) y)\right\| \left\|x - y\right\| d\theta \\
        & = &  \left(\int_0^1 \left\|\mathbf{J}(\theta x + (1 - \theta) y)\right\|d\theta \right) \left\|x - y\right\| \\
        & \leq &  \left(\int_0^1L_0 + L_1 \|F(\theta x + (1 - \theta) y)\|^{\alpha} d\theta \right) \left\|x - y\right\| \\
        & = &  \left(L_0 + L_1 \int_0^1 \|F(\theta x + (1 - \theta) y)\|^{\alpha} d\theta \right) \left\|x - y\right\|.
    \end{eqnarray*}
    Then, using Lemma \ref{lemma:reform_integration}, we have the result. 
\end{proof}

\subsection{Computation of $\alpha, L_0, L_1$ for $\mathcal{L}(w_1, w_2)$.}\label{sec:compute_L0L1}

We now revisit the min-max problem defined in~\eqref{eq:min_max_cubic}. Note that, the operator corresponding to this problem is given by 
\begin{equation*}
    F(x) = \begin{bmatrix}
        w_1^2 + w_2 \\
        w_2^2 - w_1
    \end{bmatrix}
\end{equation*}
Then the norm of operator is $\| F(x)\| = \sqrt{\left(w_1^2 + w_2 \right)^2 + \left(w_2^2 - w_1 \right)^2}$. Moreover, the Jacobian matrix is given by
\begin{equation*}
    \mathbf{J}(x) = \begin{bmatrix}
        2 w_1 & 1 \\
        -1 & 2 w_2
    \end{bmatrix}.
\end{equation*}
Then the maximum singular value at any point $x$ is given by
\begin{eqnarray}
    \|\mathbf{J}(x)\| & = & \lambda_{\max} \left( \mathbf{J}(x)^{\top} \mathbf{J}(x) \right) \notag \\
    & = & \lambda_{\max} \left( \begin{bmatrix}
        4 w_1^2 + 1 & 2(w_1 - w_2) \\
        2(w_1 - w_2) & 4 w_2^2 + 1
    \end{bmatrix} \right)  \notag \\
    & \overset{\eqref{eq:eigen_2dmatrix}}{=} & \sqrt{2 (w_1^2 + w_2^2) + 1 + 2 \sqrt{(w_1 - w_2)^2 + (w_1^2 - w_2^2)^2}}
\end{eqnarray}

\begin{figure}
    \centering
     \includegraphics[width=0.6\linewidth]{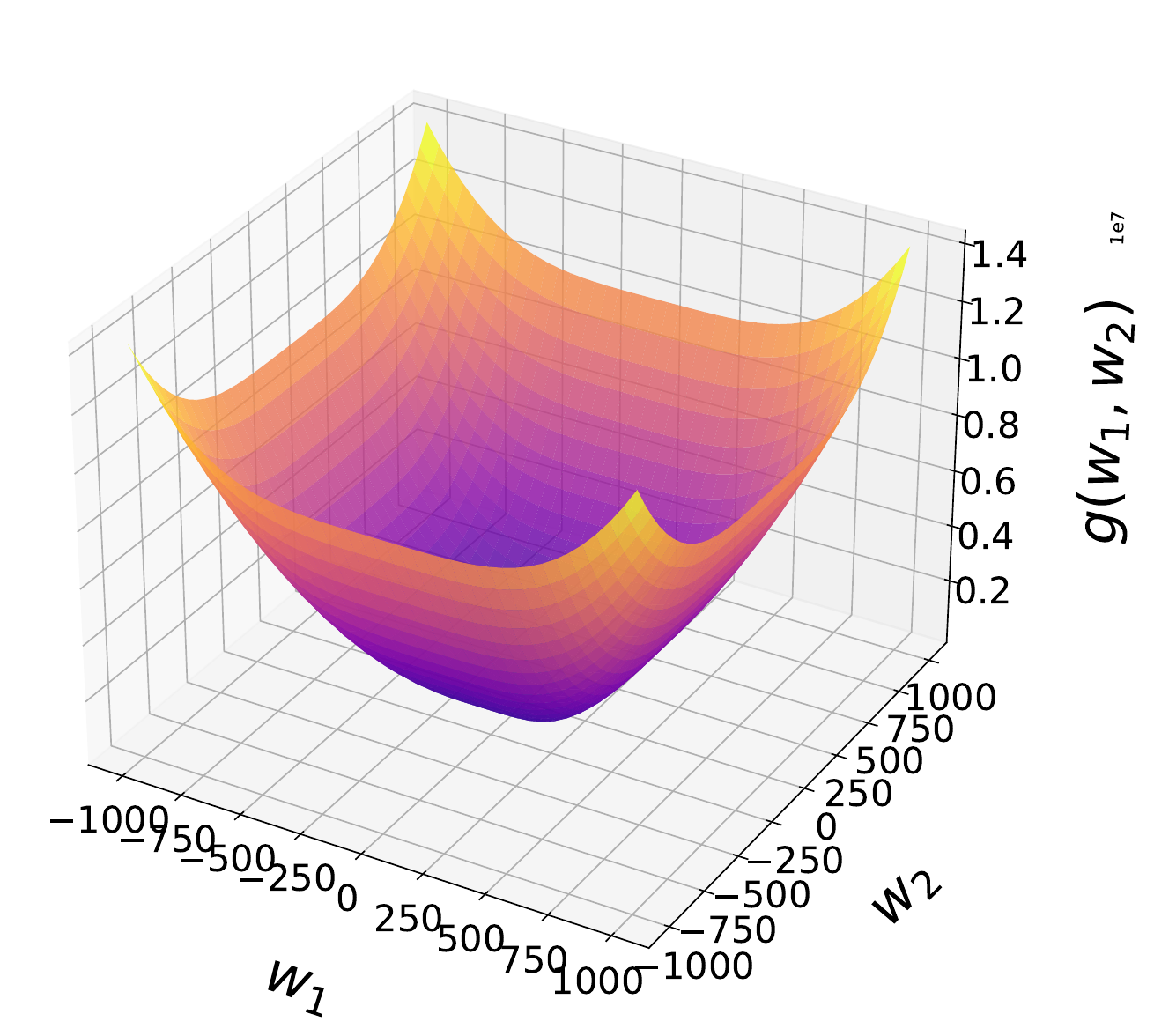}
     \caption{Plot of $g(w_1, w_2)$~\eqref{eq:compute_L0L1}. Here, the $z$-axis is in $10^7$ scale.}\label{fig:compute_L0L1}
\end{figure}

To validate whether the operator $F$ satisfies the condition~\eqref{eq:(L0,L1)-Lipschitz}, we examine whether the following function is non-negative:
\begin{equation}\label{eq:compute_L0L1}
\textstyle
    g(w_1, w_2) = L_0 + L_1 \|F(x)\| - \| \mathbf{J}(x)\|.
\end{equation}
In Figure~\ref{fig:compute_L0L1}, we plot $g(w_1, w_2)$ using $L_0 = 10$ and $L_1 = 10$. We observe that $g(w_1, w_2)$ has no real solution and remains positive for all $w_1, w_2 \in \mathbb{R}$, confirming that the function~\eqref{eq:min_max_cubic} satisfies~\eqref{eq:equiv_formulation} with $(\alpha, L_0, L_1) = (1, 10, 10)$. Thus, the corresponding operator $F$ is $1$-symmetric $(10, 10)$-Lipschitz.

\newpage
\section{Additional Details on Numerical Experiments}\label{appendix:num_exp}

In this section, we provide additional details on the second experiment related to the monotone problem. First, we show that $$\mathcal{L}(w_1, w_2) = \frac{1}{3} \left( w_1^{\top} \A w_1 \right)^{\nicefrac{3}{2}} + w_1^{\top} \B w_2 - \frac{1}{3} \left(w_2^{\top} \C w_2 \right)^{\nicefrac{3}{2}}$$ is convex-concave, and then we find the equilibrium point of $\mathcal{L}$.

\paragraph{Convex-Concave $\mathcal{L}(w_1, w_2)$.} Note that for $\mathcal{L}(w_1, w_2)$ in \eqref{eq:min_max_cubic_Rd} we have
$\nabla_{w_1} \mathcal{L}(w_1, w_2) = \left( w_1^{\top} \A w_1\right)^{\nicefrac{1}{2}} \A w_1 + \B w_2$ and $\nabla_{w_2} \mathcal{L}(w_1, w_2) = \B^{\top} w_1 - \left( w_2^{\top} \C w_2\right)^{\nicefrac{1}{2}} \C w_2$. Then the second-order derivatives are given by
\begin{equation*}
    \nabla_{w_1 w_1}^2 \mathcal{L} (w_1, w_2) = \| \A^{\nicefrac{1}{2}} w_1\| \A + \frac{\A w_1 w_1^{\top} \A^{\top}}{\| \A^{\nicefrac{1}{2}} w_1\|}
\end{equation*}
Here, $\A$ is positive definite and $\A w_1 w_1^{\top} \A^{\top}$ is a positive semidefinite matrix. Hence, $\nabla_{w_1 w_1}^2 \mathcal{L} (w_1, w_2)$ is a positive definite matrix as well and $\mathcal{L}(\cdot, w_2)$ is convex for any $w_2$. Similarly, we show that 
\begin{equation*}
    - \nabla_{w_2 w_2}^2 \mathcal{L} (w_1, w_2) = \| \C^{\nicefrac{1}{2}} w_2\| \C + \frac{\C w_2 w_2^{\top} \C^{\top}}{\| \C^{\nicefrac{1}{2}} w_2\|}
\end{equation*}
and $- \nabla_{w_2 w_2}^2 \mathcal{L} (w_1, w_2)$ is positive definite. Therefore, $\mathcal{L}(w_1, \cdot)$ is concave for any $w_1$. This proves that $\mathcal{L}(w_1, w_2)$ is convex with respect to $w_1$ and concave with respect to $w_2$. Thus, we conclude that the corresponding operator $F$ is monotone. 

\paragraph{Equilibrium of $\mathcal{L}(w_1, w_2)$.} To find the equilibrium points, we solve the set of equations given by $\nabla_{w_1} \mathcal{L}(w_1, w_2) = 0$ and $\nabla_{w_2} \mathcal{L}(w_1, w_2) = 0$, i.e., solve for 
\begin{eqnarray*}
    \left( w_1^{\top} \A w_1\right)^{\nicefrac{1}{2}} \A w_1 + \B w_2 & = & 0, \\
    \B^{\top} w_1 - \left( w_2^{\top} \C w_2\right)^{\nicefrac{1}{2}} \C w_2 & = & 0.
\end{eqnarray*}
Now multiplying the first equation with $w_1^{\top}$ and second one with $w_2^{\top}$, we have 
\begin{eqnarray*}
    \left( w_1^{\top} \A w_1\right)^{\nicefrac{1}{2}} w_1^{\top}\A w_1 + w_1^{\top} \B w_2 & = & 0 \\
    w_2^{\top}\B^{\top} w_1 - \left( w_2^{\top} \C w_2\right)^{\nicefrac{1}{2}} w_2^{\top}\C w_2 & = & 0
\end{eqnarray*}
Combining these two equations, we get
\begin{equation*}
    \left( w_1^{\top} \A w_1\right)^{\nicefrac{1}{2}} w_1^{\top}\A w_1 + \left( w_2^{\top} \C w_2\right)^{\nicefrac{1}{2}} w_2^{\top}\C w_2 = 0
\end{equation*}
which can be equivalently written as 
\begin{equation*}
    \left\| \A^{\nicefrac{1}{2}} w_1 \right\|^3 + \left\| \C^{\nicefrac{1}{2}} w_2 \right\|^3 = 0
\end{equation*}
which implies $w_1 = w_2 = 0$ as both $\A, \C$ are positive definite matrices (hence $\A^{\nicefrac{1}{2}}, \C^{\nicefrac{1}{2}}$ are invertible).

\end{document}